\documentclass[11pt]{article}
\usepackage[title]{appendix}
\usepackage{xcolor}
\usepackage{geometry}                
\usepackage{graphicx}
\usepackage{subcaption}
\usepackage{pdflscape}
\usepackage{amssymb}
\usepackage[normalem]{ulem}
\usepackage{hyperref}
\usepackage{enumitem}
\usepackage{mathrsfs}

\usepackage{epstopdf}
\usepackage{rotating}
\usepackage{longtable} 
\usepackage{adjustbox}
\usepackage{float}

\usepackage{color}
\usepackage{bbm, dsfont}
\usepackage{pst-node}
\usepackage{tikz-cd}
\usetikzlibrary{arrows.meta}
\usepackage{amsfonts} 
\usepackage{geometry}
\usepackage{amsthm}
\usepackage{amsmath}
\usepackage{titlesec}
\usepackage{amssymb}
\usepackage{enumitem}
\usepackage{float}
\restylefloat{table}
\usepackage{tikz}
\usetikzlibrary{tikzmark}
\usetikzlibrary{arrows}
\usetikzlibrary{cd}
\usetikzlibrary{automata, positioning, arrows}
\usetikzlibrary{arrows.meta}
\usepackage[english]{babel}
\usepackage[autostyle, english = american]{csquotes}
\usepackage{algorithm}
\usepackage[noend]{algpseudocode} 
\makeatletter
\def\BState{\State\hskip-\ALG@thistlm}
\makeatother

\def\X^n{\text{X}^n}

\def\Id{\textup{Id}}
\usepackage{hyperref}
\usepackage[normalem]{ulem}

\newlist{casess}{enumerate}{1}
\setlist[casess]{label=     \textbf{Case} \arabic*:}
\usepackage{mathtools}

\makeatletter
\newcommand*{\rom}[1]{\expandafter\@slowromancap\romannumeral #1@}
\makeatother

\theoremstyle{plain}
\newtheorem{thm}{Theorem}[section]
\newtheorem{prop}[thm]{Proposition}
\newtheorem{cor}[thm]{Corollary}
\newtheorem{lemma}[thm]{Lemma}
\newtheorem*{thmA}{Theorem A}
\newtheorem*{thmB}{Theorem B}
\newtheorem*{thmC}{Theorem C}
\newtheorem*{corD}{Corollary D}

\theoremstyle{definition}
\newtheorem{defn}[thm]{Definition}
\newtheorem{rmk}[thm]{Remark}

\newtheorem{exmp}[thm]{Example}

\usepackage{etoolbox}

\makeatletter
\patchcmd{\ttlh@hang}{\parindent\z@}{\parindent\z@\leavevmode}{}{}
\patchcmd{\ttlh@hang}{\noindent}{}{}{}
\makeatother

\usepackage{listings}
\usepackage{color} 
\definecolor{mygreen}{RGB}{28,172,0} 
\definecolor{mylilas}{RGB}{170,55,241}

\newlist{Assumptions}{enumerate}{1}
\setlist[Assumptions]{label=     \textbf{Assumption} \arabic*:}

\makeatletter

\newsavebox{\@brx}
\newcommand{\llangle}[1][]{\savebox{\@brx}{$\m@th{#1\langle}$}%
  \mathopen{\copy\@brx\kern-0.5\wd\@brx\usebox{\@brx}}}
\newcommand{\rrangle}[1][]{\savebox{\@brx}{$\m@th{#1\rangle}$}%
  \mathclose{\copy\@brx\kern-0.5\wd\@brx\usebox{\@brx}}}
\makeatother

\usepackage{lipsum} 
\usepackage{titlesec}
\titleformat{\subsection}[runin]
       {\normalfont\bfseries}
       {\thesubsection}
       {0.5em}
       {}
       [.]

\hypersetup{
  hidelinks,
  pdftitle={Near full groups of bounded type, II},
  pdfauthor={Zheng Kuang},
  pdfsubject={Groups of bounded type and topological full groups},
  pdfkeywords={groups of bounded type; AF alternating groups; topological full groups; Bratteli diagrams; fragmentation}
}

\begin{document}
\title{Near full groups of bounded type, \rom{2}}
\author{Zheng Kuang\thanks{School of Mathematics, South China University of Technology,  381 Wushan Rd, Tianhe District, Guangzhou 510641, China; Email: kzkzkzz@scut.edu.cn; ORCID: 0000-0001-6581-8207.}}
\date{}
\maketitle

\begin{abstract}
We study when a finitely generated group of bounded type contains the AF alternating group of the tail groupoid of its defining Bratteli diagram. This containment is an input to the full group completion theorem of Part I. We give sufficient conditions for this containment, in which selectors separate inherited diagonal actions, and an exact criterion for finite regular coverings. In the covering case, the intersection with the AF alternating group is the centralizer of the deck transformations commuting with the adjoined AF permutations. The formula classifies all subgroups between the canonical lift of the base AF alternating group and the AF alternating group of the covering. In particular, the lift is maximal exactly when the deck group has prime order, and cyclic coverings yield finitely generated groups whose intersections with the AF alternating group of the covering realize these subgroups. We identify fragmentations that produce the required selectors. For periodic fragmentations of the directed generator in the modified Fabrykowski--Gupta group, surjectivity of the activity subspace onto each pair of consecutive coordinates gives containment with absorption delay at most two. A five-periodic example also satisfies this condition, although its activity subspace has dual distance two. The original four-periodic fragmentation group has index four in its topological full group.
\end{abstract}


\setcounter{tocdepth}{1}
\tableofcontents

\section{Introduction}
\label{sec:introduction}

Groups of bounded type defined by tile inflations over simple Bratteli diagrams embed canonically in the topological full groups of their groupoids of germs. The completion theorem of Part I~\cite{kua26no1} combines containment of the AF alternating group with localization of singular germs and AF parity to compute the index under its parity hypotheses. For finite regular coverings, we compute the intersection with the AF alternating group and classify all subgroups between the canonical lift of the base AF alternating group and the AF alternating group of the covering. For a fragmentation group of the modified Fabrykowski--Gupta group, we prove this containment and show that the resulting group has index four in its topological full group.

Let $\mathsf{B}$ be a simple Bratteli diagram with space of infinite paths $\Omega(\mathsf{B})$ homeomorphic to a Cantor space and tail groupoid $\mathfrak{T}_{\mathsf{B}}$. Let $\widehat{G}$ be a finitely generated group of bounded type defined by a tile inflation over $\mathsf{B}$, and put $\mathfrak{G}:=\operatorname{Germ}(\widehat{G}\curvearrowright\Omega(\mathsf{B}))$ the groupoid of germs of its action on $\Omega(\mathsf{B})$. We assume that $\mathfrak{T}_{\mathsf{B}}$ is an open subgroupoid of $\mathfrak{G}$ and that $\mathfrak{G}\setminus\mathfrak{T}_{\mathsf{B}}$ is discrete. Hence $\mathfrak{G}$ is AF-by-discrete over $\mathsf{B}$ in the sense of \cite[Subsection~5.2.4]{nek22}.

We study when $\widehat{G}$ contains $\mathsf{A}(\mathfrak{T}_{\mathsf{B}})$. Write $\mathsf{A}_n$ for the product of alternating groups on the tiles at level $n$. Under refinement, an element of $\mathsf{A}_n$ is inherited diagonally on its descendant copies, whereas $\mathsf{A}_{n+1}$ contains the direct product of the alternating groups supported on those copies. Containment therefore requires elements of $\widehat{G}$ that separate the inherited diagonal actions. The arguments below use permutations of connecting points to construct these elements and propagate the alternating groups at finite levels.

\subsection{Earlier containment constructions}
\label{subsec:introduction-history}

Lavrenyuk and Nekrashevych studied inductive limits of direct products of finite alternating groups associated with Bratteli diagrams \cite{LN07}. Nekrashevych later defined dynamical alternating groups of \'{e}tale groupoids \cite{nnek19}, which for $\mathfrak{T}_{\mathsf{B}}$ give the AF alternating group $\mathsf{A}(\mathfrak{T}_{\mathsf{B}})$.

In Nekrashevych's construction embedding the Grigorchuk group into a simple torsion group of intermediate growth, the proof of \cite[Lemma~5.3.17]{nek22} enlarges inherited diagonal alternating groups by conjugating permutations with elements that move connecting points. The subsequent argument recovers the alternating group at level $n$ after two further refinements \cite[Subsection~5.3.5, p.~556]{nek22} (see also \cite[Section~3]{nek25}). The fragmentation of the golden mean dihedral group provides another example, in which fragments supply permutations needed to recover the topological full group \cite[Section~8]{nekrash18} \cite[Subsection~5.3.6]{nek22}.

Nekrashevych introduced fragmentations of involutions acting on Cantor sets \cite[Definition~3.1]{nekrash18}. Cantu generalized the construction to groups of homeomorphisms of Cantor sets and studied the subdirect products formed by restricting fragments to the pieces of a partition \cite[Chapter~4]{JC19}. Under refinement, the diagonal embeddings give subdirect products of alternating groups on descendant copies. When the supports contain at least five points, Scott's lemma identifies the diagonal blocks of these subdirect products \cite[p.~328]{Scott80} (see also \cite[Section~4.6]{PraegerSchneider18}).

\subsection{Main results}
\label{subsec:introduction-results}

An \emph{alternating propagation system} specifies groups at finite levels and the inclusions used to construct them inside $\widehat{G}$ (Definition~\ref{def:alternating-propagation-system}). A \emph{finite refinement scheme} records the inclusions needed to recover alternating groups supported on individual tiles after further refinement (Definition~\ref{def:finite-refinement-scheme}). Induction on the reachable graph proves containment and bounds the absorption delay under these conditions.

\begin{thmA}[Propagation and absorption]
  Let $\widehat{G}$ be a group of bounded type over a simple Bratteli diagram $\mathsf{B}$. Suppose that $\widehat{G}$ admits an alternating propagation system with a finite refinement scheme whose reachable graph is acyclic. Then
  \[
    \mathsf{A}(\mathfrak{T}_{\mathsf{B}})\leq\widehat{G}.
  \]
  Let $L$ be the maximal length of a directed path in the reachable graph. Then the absorption delay is at most $L+1$.
\end{thmA}

Theorem~A is Theorem~\ref{thm:finite-refinement-containment}. Conditions~\ref{cond:selector-system:A2} and~\ref{cond:refinement:R2} must be verified for the given action. The graph bounds absorption only after these inclusions hold. These hypotheses do not give a finite test from a generating set. The proof uses the conjugation and gluing argument of \cite[Lemma~5.3.17]{nek22}. Section~\ref{sec:fragmentations-producing-selectors} determines which permutations of connecting points can be realized by a fragmentation while prescribed connector neighborhoods remain fixed. In other words, \emph{selectors} separate these blocks by acting at designated connectors while fixing \emph{protected} neighborhoods. For an elementary abelian fragmentation with activity subspace $C\leq\mathbb{F}_p^m$, the standard dual-distance criterion \cite{MacWilliamsSloane77} shows that all prescriptions on at most $t$ pieces are realizable exactly when $d(C^\perp)>t$. Additional fixed points needed for conjugation are checked separately.

If the restrictions of $\widehat{G}\cap\mathsf{A}_n$ to identified supports form a proper subdirect product, Proposition~\ref{prop:intrinsic-synchronization-obstruction} gives $\mathsf{A}(\mathfrak{T}_{\mathsf{B}})\nleq\widehat{G}$. For finite regular coverings, the deck transformations commuting with the adjoined AF permutations determine the intersection with the AF alternating group of the covering.

Let $\mathsf{B}^{\Lambda}$ be a regular covering of $\mathsf{B}$ with finite deck group $\Lambda$ (Definition~\ref{def:finite-regular-covering}). Write $\operatorname{Lift}_{\Lambda}$ for the canonical lift, which applies each AF permutation of the base simultaneously in every fiber, and $\mathsf{S}(\mathfrak{T}_{\mathsf{B}^{\Lambda}})$ for the AF symmetric group of the cover. For $h\in\Lambda$, denote the corresponding deck transformation by $\delta_h$, and let $\delta(\Lambda)$ be the group of these transformations. For $L\leq\Lambda$, write $\operatorname{C}_J(\delta(L))$ for the centralizer of $\delta(L)$ in a group $J$.

\begin{thmB}[Criterion for finite regular coverings]
  Suppose that $\mathsf{B}^{\Lambda}$ is simple and its space of infinite paths is a Cantor set. Let $\Gamma_0\leq\operatorname{Homeo}(\Omega(\mathsf{B}^{\Lambda}))$ commute with $\delta(\Lambda)$ and contain $\operatorname{Lift}_{\Lambda}(\mathsf{A}(\mathfrak{T}_{\mathsf{B}}))$. For $s_1,\ldots,s_r\in\mathsf{S}(\mathfrak{T}_{\mathsf{B}^{\Lambda}})$ put
  \[
    \Gamma=\langle\Gamma_0,s_1,\ldots,s_r\rangle,
    \qquad
    \Lambda_0=\{h\in\Lambda:s_j\delta_h=\delta_hs_j\text{ for every }j\}.
  \]
  Then
  \[
    \Gamma\cap\mathsf{A}(\mathfrak{T}_{\mathsf{B}^{\Lambda}})
    =\operatorname{C}_{\mathsf{A}(\mathfrak{T}_{\mathsf{B}^{\Lambda}})}(\delta(\Lambda_0)),
  \]
  and therefore $\mathsf{A}(\mathfrak{T}_{\mathsf{B}^{\Lambda}})\leq\Gamma$ if and only if $\Lambda_0=\{1\}$.
\end{thmB}

Theorem~B is Theorem~\ref{thm:regular-covering-containment}. It determines the entire intersection with the AF alternating group of the covering, rather than only whether the containment holds. When a deck transformation does not commute with an adjoined AF permutation, Lemma~\ref{lem:covering-discrepancy-propagation} shows that the restrictions differ on the corresponding tiles over every tile on sufficiently deep level of $\mathsf{B}$. Scott's lemma then identifies the remaining diagonal blocks with the $\Lambda_0$-orbits on the fiber coordinates. This intersection formula does not use the full group completion theorem of Part I. For every subgroup $L\leq\Lambda$, put
\[
H_L:=\operatorname{C}_{\mathsf{A}(\mathfrak{T}_{\mathsf{B}^{\Lambda}})}(\delta(L)).
\]
Corollary~\ref{cor:deck-subgroup-lattice} shows that every subgroup between $\operatorname{Lift}_\Lambda(\mathsf A(\mathfrak T_{\mathsf B}))$ and $\mathsf A(\mathfrak T_{\mathsf B^\Lambda})$ is $H_L$ for a unique subgroup $L\leq\Lambda$. Each $H_L$ is generated by the canonical lift of the AF alternating group and one AF permutation at a finite level. Moreover, for $L,K\leq\Lambda$,
\[
H_K\leq H_L\ \Longleftrightarrow\ L\leq K,\qquad
H_L\cap H_K=H_{\langle L,K\rangle},\qquad
\langle H_L,H_K\rangle=H_{L\cap K}.
\]
Thus $L\mapsto H_L$ identifies the subgroup lattice of $\Lambda$ with this interval, reversing inclusion and interchanging meets and joins. For cyclic deck groups, Corollary~\ref{cor:finitely-generated-covering-lattice} realizes these inclusions among finitely generated groups. Adjoining a nontrivial $3$-cycle supported in one tile of the covering gives $\Lambda_0=\{1\}$ and hence containment of the AF alternating group (Corollary~\ref{cor:one-cycle-covering-containment}).

Matui \cite[Proposition~3.12]{Matui26Factor} proved that, for degree-two coverings of minimal effective ample groupoids that are almost finite or purely infinite, adjoining an element outside the pulled-back full group to the pulled-back commutator subgroup generates a group containing the ambient commutator subgroup. He also treated prime factor maps with fibers of size at most two and a principal target. For minimal AF groupoids, these commutator subgroups are the AF alternating groups. For every finite deck group,
\[
\operatorname{C}_{\mathsf A(\mathfrak T_{\mathsf B^\Lambda})}(\delta(\Lambda))
=\operatorname{Lift}_\Lambda(\mathsf A(\mathfrak T_{\mathsf B})),
\]
since deck equivariance forces the same even permutation on all lifted copies of each tile of $\mathsf{B}$ at a finite level. Thus Matui's result gives degree-two containment as in Theorem~B and maximality of the canonical lift. Theorem~B computes the exact intersection with the AF alternating group for arbitrary finite deck groups, including intermediate centralizers. Corollary~\ref{cor:deck-subgroup-lattice} also shows that the canonical lift is maximal in the AF alternating group of the covering exactly when the deck group has prime order. We do not address higher-degree factor maps outside the class of finite regular coverings considered here.

We apply Theorem~A to a fragmentation of the modified Fabrykowski--Gupta group. The construction starts with the Fabrykowski--Gupta group generated by the wreath recursions $a=(0\,1\,2)$ and $b=(a,\Id,b)$. Cantu established the existence of suitable periodic fragmentations \cite[Example~5.1.8]{JC19}. The two explicit four-periodic fragments $b_1,b_2$ and their wreath recursions were constructed in \cite[Subsection~5.4]{kua26}. We lift the corresponding fragmented action to the one-sided Markov shift over the alphabet $\widehat{\mathsf{X}}=\{0_0,0_1,1,2\}$ as described in Section~\ref{sec:fragment-FG-modified}, and denote the resulting group by $\widehat{G}_{\mathrm{FG}}$. Put $\mathfrak{G}_{\mathrm{FG}}:=\operatorname{Germ}(\widehat{G}_{\mathrm{FG}}\curvearrowright\Omega(\mathsf{B}))$. The activity subspace of $\langle b_1,b_2\rangle$ is the self-dual ternary $[4,2,3]$ code. Hence the fragments realize every prescribed pair of activity coordinates.

\begin{thmC}[Containment for the modified Fabrykowski--Gupta group]
  Let $\widehat{G}_{\mathrm{FG}}$ be the fragmentation group constructed in Section~\ref{sec:fragment-FG-modified}. Then it admits an alternating propagation system with a finite refinement scheme whose reachable graph has maximal path length one. Consequently,
  \[
    \mathsf{A}(\mathfrak{T}_{\mathsf{B}})\leq\widehat{G}_{\mathrm{FG}},
  \]
  and the absorption delay is at most $2$.
\end{thmC}

Theorem~C is Theorem~\ref{prop:FG-selector-system}. The selector cyclically permutes the connecting points of the current connector and fixes those of the next connector. The explicit fragments supply these actions, and Lemma~\ref{lem:FG-selector-phase-support} verifies the auxiliary fixed points needed to derive the propagation inclusions. The finite refinement scheme then gives absorption delay at most two, as in the earlier two-level construction of Nekrashevych.

We retain the Bratteli diagram, the lifted directed generator, and the four finitary generators of Theorem~C, but vary the period and activity subspace of the fragmentation. For $m\geq2$ and $C\leq\mathbb F_3^m$, let $\widehat G(C)$ be the group defined in Subsection~\ref{FGcontain}. Proposition~\ref{prop:FG-periodic-activity-containment} gives
\[
\bigl\{(c_j,c_{j+1}):c\in C\bigr\}=\mathbb F_3^2
\quad\text{for every }j\in\mathbb Z/m\mathbb Z
\quad\Longrightarrow\quad
\mathsf A(\mathfrak T_{\mathsf B})\leq\widehat G(C),
\]
with absorption delay at most two. The activity subspace in Example~\ref{exmp:FG-five-periodic} projects onto every pair of consecutive coordinates, but $d(C^\perp)=2$. Hence the group contains the AF alternating group although the stronger sufficient condition $d(C^\perp)>2$ fails.

For the original four-periodic group $\widehat{G}_{\mathrm{FG}}$, the calculation of parity and the localization of singular germs in Subsections~\ref{subsec:FG-AF-parity} and~\ref{subsec:FG-singular-full-group}, together with the completion theorem of Part I, give the exact index in the topological full group.

\begin{corD}[Near-full completion]
  For every level $N\geq 2$, any transpositions $\tau_{0_0}\in\mathsf{S}(\mathcal{T}_{0_0,N})$ and $\tau_{0_1}\in\mathsf{S}(\mathcal{T}_{0_1,N})$ satisfy
  \[
    \mathsf{F}(\mathfrak{G}_{\mathrm{FG}})
    =\left\langle\widehat{G}_{\mathrm{FG}},\tau_{0_0},\tau_{0_1}\right\rangle,
    \qquad
    [\mathsf{F}(\mathfrak{G}_{\mathrm{FG}}):\widehat{G}_{\mathrm{FG}}]=4.
  \]
\end{corD}

Corollary~D is Corollary~\ref{cor:FG-near-full-completion}. Its index follows from trivial AF parity (Proposition~\ref{prop:FG-AF-parity}), the four parity classes of the tail groupoid, and localization with simultaneous AF truncation. The calculation concerns the original four-periodic group, not every $\widehat G(C)$.

\section{Preliminaries}
\label{sec:preliminaries}

We use the notation and conventions of \cite{kua26no1}, including finitary transformations and germ-defining singular points. We write $\mathsf{F}(\mathfrak{G})$ for the topological full group of a groupoid $\mathfrak{G}$. A subgroup of $\mathsf{F}(\mathfrak{G})$ is \emph{near full} if it has finite index. We recall the alternating groups at finite levels and the completion theorem of Part I. We use left actions throughout this paper. Group elements are read from right to left, while paths in a Bratteli diagram are read from left to right.

\subsection{Bratteli diagrams, tiles, and AF symmetric and alternating groups}
\label{subsec:prelim-bratteli-tiles}

Recall that a Bratteli diagram is \emph{simple} if, for every level $n$, there is $m>n$ such that every vertex of $V_n$ is joined by a finite path to every vertex of $V_m$.

Let $\mathsf{B}=((V_n)_{n\geq1},(E_n)_{n\geq1},\mathbf{s},\mathbf{r})$ be a simple Bratteli diagram whose space of infinite paths is a Cantor set. We denote this space by $\Omega(\mathsf{B})$ and the set of paths of length $n$ by $\Omega_n(\mathsf{B})$. For $v\in
  V_{n+1}$, put $\mathcal{T}_{v,n}:=\{\gamma\in\Omega_n(\mathsf{B}):\mathbf{r}(\gamma)=v\}$. We regard $\mathcal{T}_{v,n}$ as the set of vertices of the level-$n$ tile corresponding to $v$.

The tail groupoid is denoted by $\mathfrak{T}_{\mathsf{B}}$, and its elementary level-$n$ subgroupoid by $\mathfrak{T}_{\mathsf{B},n}$. We put $\mathsf{S}_n:=\prod_{v\in V_{n+1}}\mathsf{S}(\mathcal{T}_{v,n})$ and $\mathsf{A}_n:=\prod_{v\in V_{n+1}}\mathsf{A}(\mathcal{T}_{v,n})$. The AF symmetric and alternating groups are the
direct limits $\mathsf{S}(\mathfrak{T}_{\mathsf{B}})=\varinjlim\mathsf{S}_n$ and $\mathsf{A}(\mathfrak{T}_{\mathsf{B}})=\varinjlim\mathsf{A}_n$. Since $\mathfrak{T}_{\mathsf{B}}$ is AF, its topological full group is the AF symmetric group, and thus $\mathsf{F}(\mathfrak{T}_{\mathsf{B}})=\mathsf{S}(\mathfrak{T}_{\mathsf{B}})$.

Let $e\in E_{n+1}$ satisfy $\mathbf{s}(e)=v$. Appending $e$ gives a canonical identification $\phi_e:\mathcal{T}_{v,n}\to\mathcal{T}_{v,n}e$, $\phi_e(\gamma)=\gamma e$. The canonical embedding $\iota_{n,n+1}:\mathsf{S}_n\to\mathsf{S}_{n+1}$ copies the same permutation to every descendant support. More precisely, if $\sigma=(\sigma_v)_{v\in V_{n+1}}\in\mathsf{S}_n$ and $w\in V_{n+2}$, then
\begin{equation}
  \left.\iota_{n,n+1}(\sigma)\right|_{\mathcal{T}_{w,n+1}}
  =
  \prod_{\substack{e\in E_{n+1}\\ \mathbf{r}(e)=w}}
  (\sigma_{\mathbf{s}(e)})_e,
  \label{eq:prelim-diagonal-embedding}
\end{equation}
where $(\sigma_{\mathbf{s}(e)})_e=\phi_e\sigma_{\mathbf{s}(e)}\phi_e^{-1}$. These embeddings restrict to $\mathsf{A}_n\to\mathsf{A}_{n+1}$. Thus every symmetric and alternating group at a finite level is inherited at the next level as a diagonal group on its descendant copies.

\subsection{Tile inflations and connectors}
\label{subsec:prelim-tile-inflations}

We briefly recall the construction of tile inflation from \cite[Subsection~2.4]{kua26no1}. A finite tile is a well-labeled
graph with boundary. A boundary edge has exactly one endpoint in the tile, and its defined endpoint is called a
boundary point. Suppose that the level-$n$ tiles have been constructed, and fix $v\in V_{n+2}$. Before adding the new connector edges, the vertex set of the next tile is the disjoint union
\begin{equation}
  \mathcal{T}_{v,n+1}
  =
  \bigsqcup_{\substack{e\in E_{n+1}\\ \mathbf{r}(e)=v}}
  \mathcal{T}_{\mathbf{s}(e),n}e.
  \label{eq:prelim-tile-inflation}
\end{equation}
Each set $\mathcal{T}_{\mathbf{s}(e),n}e$ is a copy of a level-$n$ tile inside $\mathcal{T}_{v,n+1}$.

\begin{defn}[Connector]
  \label{def:prelim-connector}
  Let $e_1,e_2\in E_{n+1}$ have the same range. Suppose that $\gamma_1e_1$ carries an outgoing boundary edge labeled by $F$ and $\gamma_2e_2$ carries the compatible incoming boundary edge labeled by $F^{-1}$. The triple
  $c=(\gamma_1e_1,\gamma_2e_2,F)$ is a level-$(n+1)$ connector. In the inflation, we replace the two boundary edges by $\gamma_1e_1\xrightarrow{F}\gamma_2e_2$ and $\gamma_2e_2\xrightarrow{F^{-1}}\gamma_1e_1$. The vertices $\gamma_1e_1$ and $\gamma_2e_2$ are its \emph{connecting points}.
\end{defn}

We use the following convention. The \emph{connector graph} used in the inflation from level $n$ to level $n+1$ is
denoted by $\mathsf{e}_n$. It is the union of the connectors occurring in Definition~\ref{def:prelim-connector}. Thus $\mathsf{e}_n$ consists of the level-$(n+1)$ connectors added in this inflation.

The graph $\mathsf{e}_n$ may have several connected components. For example, in Section~\ref{sec:fragment-FG-modified}, the components are triangular connector graphs.

Let $V(\mathsf{e}_n)$ be the set of all connecting points of $\mathsf{e}_n$. To describe the restrictions of the transformations used at finitely many connectors, we may extend the finite paths representing their connecting points and continue to denote the extended prefixes by the same letters. We choose the extensions so that the corresponding cylinders are pairwise disjoint and the relevant germs are represented by prefix replacements. For a connector $c=(p,q,F)$, put $U_c=[p]\cup[q]$, and more generally put $U_{\mathsf{e}_n}:=\bigsqcup_{p\in V(\mathsf{e}_n)}[p]$. We call $U_{\mathsf{e}_n}$ the \emph{connector neighborhood} of $\mathsf{e}_n$. If $\mathsf{a}\subseteq\mathsf{e}_n$ is a union of connectors, we define $U_{\mathsf{a}}$ in the same way. An element $s$ induces a permutation $\pi$ of the connecting points $V(\mathsf{e}_n)$ if $s(pw)=\pi(p)w$ for every $p\in V(\mathsf{e}_n)$ and every compatible tail $w$.

The labels of the boundary and connector edges define partial homeomorphisms of $\Omega(\mathsf{B})$. After the
required extensions at boundary points, they generate the group or inverse semigroup associated with the tile
inflation. The containment arguments use the decomposition~\eqref{eq:prelim-tile-inflation}, the connector graphs, and the diagonal embedding~\eqref{eq:prelim-diagonal-embedding}.

\subsection{The completion theorem of Part I}
\label{subsec:prelim-full-group-completion}

We recall the hypotheses and conclusion of the main theorem of \cite{kua26no1}. Let $\widehat{G}=\langle S\rangle$ be a group of bounded type arising from a tile inflation over $\mathsf{B}$, put $\mathfrak{G}=\operatorname{Germ}(\widehat{G}\curvearrowright\Omega(\mathsf{B}))$, and let $\Xi=\{\xi_1,\ldots,\xi_r\}$ be the set of boundary points of the infinite tiles.

\begin{defn}[Finite singular germ condition]
  \label{def:prelim-finite-singular-germ-condition}
  We say that $\widehat{G}$ satisfies the \emph{finite singular germ condition} if $S=S_{\mathrm{fin}}\sqcup S_{\mathrm{sing}}$, every element of $S_{\mathrm{fin}}$ is finitary, every $\xi_i$ is germ-defining, every $s\in S_{\mathrm{sing}}$ fixes one $\xi_i$ and has its unique non-AF germ there, and every group of germs $H_i=\widehat{G}_{\xi_i}/\widehat{G}_{(\xi_i)}$ is finite. Here $\widehat{G}_{\xi_i}$ is the stabilizer of $\xi_i$, and $\widehat{G}_{(\xi_i)}$ is the subgroup of elements with trivial germ at $\xi_i$.
\end{defn}

In applying the germ-defining condition of \cite[Definition~3.4]{kua26no1}, we use \emph{nonloop cycle} for a closed path of length at least two labeled by the generators, whose vertices before the terminal return are pairwise distinct. In particular, a single loop edge at a singular point is excluded.

Put $\mathsf{P}_{\mathsf{B}}:=H_0(\mathfrak{T}_{\mathsf{B}};\mathbb{Z}/2\mathbb{Z})$. The homomorphism $\varepsilon_{\mathsf{B}}$ for AF parity gives an
exact sequence
\[ 1\longrightarrow\mathsf{A}(\mathfrak{T}_{\mathsf{B}})\longrightarrow
  \mathsf{S}(\mathfrak{T}_{\mathsf{B}})\xrightarrow{\ \varepsilon_{\mathsf{B}}\ }
  \mathsf{P}_{\mathsf{B}}\longrightarrow1.\]
Put $\mathsf{P}_{\widehat{G}}:=\varepsilon_{\mathsf{B}}(\widehat{G}\cap\mathsf{S}(\mathfrak{T}_{\mathsf{B}}))$. In the setting of bounded rank considered here, $\mathsf{P}_{\mathsf{B}}$ is finite, and thus the quotient $\mathsf{P}_{\mathsf{B}}/\mathsf{P}_{\widehat{G}}$ is a finite-dimensional vector space over $\mathbb{Z}/2\mathbb{Z}$. Let $p_1,\ldots,p_d\in\mathsf{P}_{\mathsf{B}}$ project
to a basis of this quotient. By surjectivity of $\varepsilon_{\mathsf{B}}$, we may choose $\tau_1,\ldots,\tau_d\in
  \mathsf{S}(\mathfrak{T}_{\mathsf{B}})$ such that $\varepsilon_{\mathsf{B}}(\tau_j)=p_j$.

\begin{defn}[Parity completion]
  \label{def:parity-completion}
  We define the \emph{parity completion} of $\widehat{G}$ by
  $\widehat{G}^{\mathrm{par}}:=\langle\widehat{G},\tau_1,\ldots,\tau_d\rangle$.
\end{defn}

When $\mathsf{A}(\mathfrak{T}_{\mathsf{B}})\leq\widehat{G}$, the exact sequence for AF parity gives $\mathsf{S}(\mathfrak{T}_{\mathsf{B}})\leq\widehat{G}^{\mathrm{par}}$. Hence $\widehat{G}^{\mathrm{par}}=\langle\widehat{G},\mathsf{S}(\mathfrak{T}_{\mathsf{B}})\rangle$, independently of the chosen representatives.

If $\mathsf{P}_{\widehat{G}}=\mathsf{P}_{\mathsf{B}}$, then $d=0$ and $\widehat{G}^{\mathrm{par}}=\widehat{G}$.

For a finite path $\gamma$, let $\kappa_\gamma(w)=\gamma w$ be the prefixing homeomorphism from the space of infinite paths starting at the terminal vertex of $\gamma$ onto the cylinder $[\gamma]$. We use the pathwise sections and shifted groups of germs from \cite[Subsection~3.3]{kua26no1}. The singular germs
satisfy the \emph{localization condition} if every shifted singular germ has, after passing to a deeper level, a
representative $\widehat{h}\in\widehat{G}$ preserving a cylinder determined by a prefix of the corresponding boundary point on an infinite tile, representing the shifted germ there, and having only AF germs outside that cylinder. The AF truncation of $\widehat{h}$ acts identically on that cylinder and agrees with $\widehat{h}$ outside it. The AF truncations satisfy the \emph{parity condition} if their parity classes belong to $\mathsf{P}_{\widehat{G}}$.

The following is \cite[Theorem~4.10]{kua26no1}.

\begin{thm}
  \label{thm:prelim-full-group-completion}
  Let $\widehat{G}$ be a finitely generated group of bounded type over a simple Bratteli diagram satisfying the finite singular germ condition. Suppose that $\mathsf{A}(\mathfrak{T}_{\mathsf{B}})\leq\widehat{G}$ and that the singular germs satisfy the localization condition. Then $\widehat{G}^{\mathrm{par}}=\mathsf{F}(\mathfrak{G})$. In particular, $\widehat{G}=\mathsf{F}(\mathfrak{G})$ if and only if $\mathsf{P}_{\widehat{G}}=\mathsf{P}_{\mathsf{B}}$.

  If the AF truncations also satisfy the parity condition, then $\mathsf{F}(\mathfrak{G})=\mathsf{S}(\mathfrak{T}_{\mathsf{B}})\widehat{G}$ and
  \[[\mathsf{F}(\mathfrak{G}):\widehat{G}]=
    [\mathsf{P}_{\mathsf{B}}:\mathsf{P}_{\widehat{G}}]
    <\infty.\]
\end{thm}

\section{Finite alternating groups}\label{sec:alternating-gluing}
We collect the elementary calculations with finite alternating groups used in the propagation arguments. All finite supports below have at least three points, and every permutation is extended identically outside its support. See, for instance, \cite[Section~6 of Chapter~I]{hungerford2003} for the standard facts about finite alternating groups.

\subsection{Gluing alternating groups}\label{subsec:gluing-alternating-supports}
Let $X$ be a finite set and let $Y\subseteq X$. We denote the symmetric and alternating groups on $Y$ by $\mathsf{S}(Y)$ and $\mathsf{A}(Y)$, respectively, and regard them as subgroups of $\mathsf{S}(X)$.

\begin{lemma}
  \label{lem:fixed-basepoint-generation}
  Let $p\in Y$.  Then $\mathsf{A}(Y)$ is generated by the cycles $(p\,u\,v)$, where $u,v\in Y\setminus\{p\}$ are distinct.
\end{lemma}

\begin{proof}
  The group $\mathsf{A}(Y)$ is generated by $3$-cycles. If $u,v,w\in Y\setminus\{p\}$ are distinct, then $(u\,v\,w)=(p\,u\,v)(p\,v\,w)$.
\end{proof}

\begin{lemma}
  \label{lem:overlap-gluing}
  Let $U,V\subseteq X$ satisfy $U\cap V\neq\varnothing$. Then $\langle\mathsf{A}(U),\mathsf{A}(V)\rangle=\mathsf{A}(U\cup V)$.
\end{lemma}

\begin{proof}
  The assertion is immediate if one support is contained in the other. Otherwise, choose $p\in U\cap V$, $a\in U\setminus V$ and $b\in V\setminus U$.  If $q\in(U\cap V)\setminus\{p\}$, then $(p\,q\,b)(p\,a\,q)=(p\,a\,b)$.
  If $U\cap V=\{p\}$, choose $c\in U\setminus\{p,a\}$ and $d\in V\setminus\{p,b\}$. Put $\alpha=(p\,a\,c)$ and $\beta=(p\,d\,b)$. We have $\alpha\beta\alpha^{-1}=(a\,d\,b)$, and hence $(a\,d\,b)\beta^{-1}=(p\,a\,d)$. Multiplying on the left by $(p\,d\,b)$ gives $(p\,a\,b)$.

  Conjugating the cycles $(p\,a\,b)$ by elements of $\mathsf{A}(U)$ and $\mathsf{A}(V)$ and applying Lemma~\ref{lem:fixed-basepoint-generation}, we obtain all $3$-cycles on $U\cup V$.
\end{proof}

\begin{cor}
  \label{cor:common-anchor-gluing}
  Let $U,V\subseteq X$ be disjoint and let $p\notin U\cup V$.  Then
  \[\langle\mathsf{A}(U\cup\{p\}),\mathsf{A}(V\cup\{p\})\rangle=\mathsf{A}(U\cup V\cup\{p\}).\]
\end{cor}

\begin{proof}
  The supports meet at $p$, and thus we apply Lemma~\ref{lem:overlap-gluing}.
\end{proof}

\begin{cor}
  \label{cor:connected-cover-gluing}
  Let $Y_1,\ldots,Y_r\subseteq X$ cover $Y$, and suppose that the graph joining $i$ to $j$ whenever $Y_i\cap Y_j\neq\varnothing$ is connected. Then $\langle\mathsf{A}(Y_1),\ldots,\mathsf{A}(Y_r)\rangle=\mathsf{A}(Y)$.
\end{cor}

\begin{proof}
  Choose a spanning tree of the intersection graph and apply Lemma~\ref{lem:overlap-gluing} successively along the tree.
\end{proof}

\begin{cor}
  \label{cor:adjoining-points-common-support}
  Let $U,V\subseteq X$ be disjoint. Suppose that a subgroup $H\leq\mathsf{A}(U\cup V)$ contains $\mathsf{A}(U\cup\{v\})$ for every $v\in V$. Then $\mathsf{A}(U\cup V)\leq H$.
\end{cor}

\begin{proof}
  The supports $U\cup\{v\}$ have the common intersection $U$. The result follows from Corollary~\ref{cor:connected-cover-gluing}.
\end{proof}

We next consider two disjoint supports joined by a $3$-cycle.

\begin{lemma}
  \label{lem:one-bridge-gluing}
  Let $U,V\subseteq X$ be disjoint. Suppose that a subgroup $H\leq\mathsf{A}(U\sqcup V)$ contains $\mathsf{A}(U)$, $\mathsf{A}(V)$, and a $3$-cycle $c$ whose support meets both $U$ and $V$. Then
  $\mathsf{A}(U\sqcup V)\leq H$.
\end{lemma}

\begin{proof}
  Put $Z=\operatorname{supp}(c)$. Since $c$ is a $3$-cycle, $\mathsf{A}(Z)=\langle c\rangle\leq H$. The supports $U,Z,V$ form a connected cover of $U\sqcup V$, and thus the result follows from
  Corollary~\ref{cor:connected-cover-gluing}.
\end{proof}

\subsection{Diagonal gluing}\label{subsec:diagonal-glue}
Let $I$ be a finite index set, let $Z$ be a finite model set, and let $\phi_i:Z\to Z_i$, $i\in I$, be bijections onto
pairwise disjoint copies. For $g\in\mathsf{S}(Z)$, put $\Delta_I(g):=\prod_{i\in I}\phi_i g\phi_i^{-1}$, and for
$Y\subseteq Z$, put $\Delta_I\mathsf{A}(Y):=\{\Delta_I(g):g\in\mathsf{A}(Y)\}$. Whenever pairwise disjoint supports $Z_i$, $i\in I$, are canonically identified with the same model set, we also write $\Delta(\mathsf{A}(Z_i):i\in I)$ for the corresponding diagonal alternating group.

\begin{lemma}
  \label{lem:diagonal-lifting}
  Let $K_1,\ldots,K_r\leq\mathsf{S}(Z)$.  Then
  \[
    \left\langle
    \Delta_I(K_1),\ldots,\Delta_I(K_r)
    \right\rangle
    =
    \Delta_I\left(\left\langle K_1,\ldots,K_r\right\rangle\right).
  \]
\end{lemma}

\begin{proof}
  The map $\Delta_I:\mathsf{S}(Z)\to\mathsf{S}(\bigsqcup_{i\in I}Z_i)$ is a homomorphism.
\end{proof}

\begin{cor}
  \label{cor:diagonal-connected-cover-gluing}
  Let $Y_1,\ldots,Y_r\subseteq Z$ cover $Y$ and have connected intersection graph. Then $\langle\Delta_I\mathsf{A}(Y_1),\ldots,\Delta_I\mathsf{A}(Y_r)\rangle=\Delta_I\mathsf{A}(Y)$.
\end{cor}

\begin{proof}
  Apply Corollary~\ref{cor:connected-cover-gluing} in the model copy and then Lemma~\ref{lem:diagonal-lifting}.
\end{proof}

\begin{cor}
  \label{cor:diagonal-common-anchor-gluing}
  Let $U,V\subseteq Z$ be disjoint and let $p\notin U\cup V$.  Then
  \[
    \left\langle
    \Delta_I\mathsf{A}(U\cup\{p\}),
    \Delta_I\mathsf{A}(V\cup\{p\})
    \right\rangle
    =
    \Delta_I\mathsf{A}(U\cup V\cup\{p\}).
  \]
\end{cor}

\begin{proof}
  Apply Corollary~\ref{cor:common-anchor-gluing} in the model copy and then Lemma~\ref{lem:diagonal-lifting}.
\end{proof}

\begin{cor}
  \label{cor:diagonal-one-bridge-gluing}
  Let $U,V\subseteq Z$ be disjoint. Suppose that a subgroup $H$ contains
  $\Delta_I\mathsf{A}(U)$, $\Delta_I\mathsf{A}(V)$, and $\Delta_I(c)$, where $c$ is a $3$-cycle whose support meets both $U$ and $V$. Then
  $\Delta_I\mathsf{A}(U\sqcup V)\leq H$.
\end{cor}

\begin{proof}
  The sets $U,\operatorname{supp}(c),V$ have connected intersection graph. Since $\langle c\rangle=\mathsf{A}(\operatorname{supp}(c))$, Corollary~\ref{cor:connected-cover-gluing} gives $\mathsf{A}(U\sqcup V)\leq\langle\mathsf{A}(U),c,\mathsf{A}(V)\rangle$. Apply Lemma~\ref{lem:diagonal-lifting}.
\end{proof}

\subsection{Subdirect products and diagonal subgroups}
\label{subsec:canonical-synchronization-blocks}

\begin{defn}\label{def:subdirect}
  Let $T_1,\ldots,T_k$ be groups. A subgroup
  $H\leq T_1\times\cdots\times T_k$ is a \emph{subdirect product} if every coordinate projection $H\to T_i$ is surjective. It is \emph{proper} if $H$ is a proper subgroup of $T_1\times\cdots\times T_k$. 

  Let $I\subseteq\{1,\ldots,k\}$ be nonempty. A subgroup $D\leq\prod_{i\in I}T_i$ is a \emph{full diagonal subgroup} if there are a group $T$ and isomorphisms $\alpha_i:T\to T_i$, $i\in I$, such that
  \[
  D=\{(\alpha_i(t))_{i\in I}:t\in T\}.
  \]
\end{defn}

The decomposition of subdirect products of nonabelian simple groups is Scott's lemma \cite[p.~328]{Scott80} (see also \cite[Section~4.6]{PraegerSchneider18}). We state it in the form used below.

\begin{lemma}
  \label{lem:subdirect-simple-decomposition}
  Let $T_1,\ldots,T_k$ be nonabelian simple groups and let $H\leq T_1\times\cdots\times T_k$ be a subdirect product. Then there is a partition $\{1,\ldots,k\}=I_1\sqcup\cdots\sqcup I_r$ such that $H=H_1\times\cdots\times H_r$, where each $H_j$ is a full diagonal subgroup of $\prod_{i\in I_j}T_i$. The partition is uniquely determined by $H$.
\end{lemma}

The following consequence of Scott's lemma determines the partition in Lemma~\ref{lem:subdirect-simple-decomposition} for groups generated by diagonal subgroups.

\begin{prop}
  \label{prop:diagonal-coordinate-separation}
  Let $T$ be a nonabelian simple group, let $I$ be finite, and for $a=1,\ldots,m$ choose automorphisms
  $\alpha_i^{(a)}\in\operatorname{Aut}(T)$, $i\in I$. Put
  \[
  D_a:=\left\{\bigl(\alpha_i^{(a)}(t)\bigr)_{i\in I}:t\in T\right\}\leq T^I,
  \qquad
  H:=\langle D_1,\ldots,D_m\rangle.
  \]
  Define a relation on $I$ by $i\sim j$ if the automorphism $\alpha_j^{(a)}\bigl(\alpha_i^{(a)}\bigr)^{-1}$ is independent of $a$. Then $\sim$ is an equivalence relation, and its classes are exactly the diagonal blocks in the decomposition of $H$ from Lemma~\ref{lem:subdirect-simple-decomposition}. In particular, $H=T^I$ if and only if, for every distinct $i,j\in I$, there are $a,b$ such that
  \[
  \alpha_j^{(a)}\bigl(\alpha_i^{(a)}\bigr)^{-1}
  \neq
  \alpha_j^{(b)}\bigl(\alpha_i^{(b)}\bigr)^{-1}.
  \]
\end{prop}

\begin{proof}
  The relation is reflexive and symmetric. If $i\sim j$ and $j\sim\ell$, then
  \[
  \alpha_\ell^{(a)}(\alpha_i^{(a)})^{-1}
  =\alpha_\ell^{(a)}(\alpha_j^{(a)})^{-1}
   \alpha_j^{(a)}(\alpha_i^{(a)})^{-1}
  \]
  is independent of $a$, and thus the relation is transitive.

  The group $H\leq T^I$ is a subdirect product. By Lemma~\ref{lem:subdirect-simple-decomposition}, its projection to the $i$- and $j$-coordinates is either $T\times T$ or the graph of an isomorphism $T\to T$. The projection of $D_a$ to these two coordinates is the graph of $\alpha_j^{(a)}(\alpha_i^{(a)})^{-1}$. Hence the two-coordinate projection of $H$ is a proper diagonal subgroup exactly when these graph isomorphisms are the same for all $a$. This proves the description of the blocks and the final assertion.
\end{proof}

\begin{cor}
  \label{cor:breaking-synchronization}
  Let $Z$ be a finite set with $|Z|\geq5$, let $Z_i$, $i\in I$, be canonically identified copies of $Z$, and let
  $D=\Delta_{i\in I}\mathsf{A}(Z_i)$.
  Suppose that $s_1,\ldots,s_m$ preserve every $Z_i$ setwise, and put
  \[
  H=\langle D,s_1Ds_1^{-1},\ldots,s_mDs_m^{-1}\rangle.
  \]
  Two coordinates $i,j$ lie in the same diagonal block of $H$ if and only if every $s_a$ has the same restriction on $Z_i$ and $Z_j$ under the canonical identifications. Consequently, if one $s_a$ acts differently on two synchronized coordinates, then those coordinates lie in different diagonal blocks of $H$.
\end{cor}

\begin{proof}
  Apply Proposition~\ref{prop:diagonal-coordinate-separation} to $T=\mathsf{A}(Z)$, including the original diagonal group among the diagonal generators. Conjugation by the restriction of $s_a$ gives the corresponding coordinate automorphism. Since the centralizer of $\mathsf{A}(Z)$ in $\mathsf{S}(Z)$ is trivial for $|Z|\geq5$, equality of the induced conjugation automorphisms is equivalent to equality of the restrictions under the canonical identifications.
\end{proof}

\section{Containment of the AF alternating group under refinement}
\label{sec:selector-absorbing-alternating-systems}

Let $\mathsf{B}$ be a simple Bratteli diagram, $\mathcal{T}_{v,n}$ be a level-$n$ tile, $\mathsf{A}_n=\prod_{v\in
    V_{n+1}}\mathsf{A}(\mathcal{T}_{v,n})$, and let $\iota_{n,m}$ denote the canonical diagonal embedding. Let $\widehat{G}$ be a group acting on $\Omega(\mathsf{B})$ defined by the tile inflation process under consideration.

We use diagonal embeddings, prescribed permutations of connecting points, and the gluing lemmas of Section~\ref{sec:alternating-gluing} to construct alternating groups at successive levels.

We increase the initial level so that every finite set used in the arguments of Subsection~\ref{subsec:canonical-synchronization-blocks} has at least five points. The gluing lemmas require only three points. We treat the preceding levels separately.

\subsection{Alternating data under refinement}
\label{subsec:alternating-data-diagonal-inheritance}

By~\eqref{eq:prelim-diagonal-embedding}, one alternating permutation at level $n$ acts in the same way on every descendant copy at level $n+1$. The following data record the supports and their identifications.

For $m\leq n$, $v\in V_{m+1}$, and $w\in V_{n+1}$, let $\Omega_{m,n}(v,w)$ be the set of paths $\eta=e_{m+1}\cdots e_n$ such that $\mathbf{s}(e_{m+1})=v$ and $\mathbf{r}(e_n)=w$. For $m=n$, the set $\Omega_{n,n}(v,v)$ consists of the empty path, and $\Omega_{n,n}(v,w)$ is empty for $v\neq w$. If $Y\subseteq\mathcal{T}_{v,m}$ and $\eta\in\Omega_{m,n}(v,w)$, put $Y\eta:=\{\gamma\eta:\gamma\in Y\}$.

\begin{defn}[Descendant embedding]
  \label{def:descendant-embedding}
  Let $m\leq n$, and let $Y=Y_1\sqcup\cdots\sqcup Y_r$, where each $Y_j$ is a finite subset of a tile $\mathcal{T}_{v_j,m}$. A \emph{descendant embedding of $Y$ at level $n$} is a map $\phi:Y\to\Omega_n(\mathsf{B})$ for which there are a vertex $w\in V_{n+1}$ and paths $\eta_j\in\Omega_{m,n}(v_j,w)$, $j=1,\ldots,r$, such that $\phi(\gamma)=\gamma\eta_j$ for every $\gamma\in Y_j$.
\end{defn}

The map $\phi$ is injective, and its image is $\phi(Y)=\bigsqcup_{j=1}^rY_j\eta_j \subseteq\mathcal{T}_{w,n}$. Thus $\phi$ preserves each level-$m$ prefix and appends the path $\eta_j$ on $Y_j$. These paths determine the identification with its image.

\begin{defn}[Alternating datum]
  \label{def:alternating-datum}
  An \emph{alternating datum at level $n$} is a tuple $\mathcal{D}=(Y;\phi_1,\ldots,\phi_k)$, where:

  \begin{enumerate}
    \item
          $Y=Y_1\sqcup\cdots\sqcup Y_r$ is a finite set with $|Y|\geq3$, and each $Y_j$ is contained in a level-$m$ tile for some fixed $m\leq n$;

    \item
          $\phi_1,\ldots,\phi_k$ are descendant embeddings of the same set $Y$ at level $n$;

    \item
          the coordinate supports $Y^{(i)}:=\phi_i(Y)$, $i=1,\ldots,k$, are pairwise disjoint.
  \end{enumerate}

  We associate with $\mathcal{D}$ the subgroup
  \[\mathsf{A}(\mathcal{D}):=
    \left\{\prod_{i=1}^k\phi_i\sigma\phi_i^{-1}:\sigma\in\mathsf{A}(Y)\right\}\leq\mathsf{A}_n,\]
  where every permutation $\phi_i\sigma\phi_i^{-1}$ is extended identically outside $Y^{(i)}$.

  We call $\mathcal{D}$ \emph{independent} if $k=1$ and \emph{synchronized} if $k\geq2$. Here independent refers only to the datum having one coordinate support. We define its support by $\operatorname{supp}(\mathcal{D})=\bigsqcup_{i=1}^kY^{(i)}$.
\end{defn}

In a synchronized datum, one permutation acts on all coordinate supports. The decomposition $Y=Y_1\sqcup\cdots\sqcup Y_r$ permits each support to combine copies of several tiles under inflation.

\begin{exmp}
  \label{exmp:descendant-embedding-FG}
  Consider the vertex $0_1$ in the Bratteli diagram (Figure~\ref{fig:FG-modified-bratteli}) of the modified
  Fabrykowski--Gupta group. Let $e_{0_0},e_1,e_2$ be the edges from $0_1$ to $0_0,1,2$, respectively. Take
  $Y=\mathcal{T}_{0_1,n}$ and define
  \[\phi_x:Y\longrightarrow\Omega_{n+1}(\mathsf{B}),
  \qquad
    \phi_x(\gamma)=\gamma e_x,
    \qquad
    x\in\{0_0,1,2\}.
  \]
  Then $\mathcal{D} =(Y;\phi_{0_0},\phi_1,\phi_2)$ is a synchronized alternating datum at level $n+1$, and
  \[
    \begin{aligned}
      \mathsf{A}(\mathcal{D})
       & =
      \Delta\bigl(
      \mathsf{A}(\mathcal{T}_{0_1,n}e_{0_0}),
      \mathsf{A}(\mathcal{T}_{0_1,n}e_1),
      \mathsf{A}(\mathcal{T}_{0_1,n}e_2)
      \bigr) \\
       & =
      \iota_{n,n+1}
      \bigl(
      \mathsf{A}(\mathcal{T}_{0_1,n})
      \bigr).
    \end{aligned}
  \]
  Figure~\ref{fig:FG-descendant-datum} shows these descendant copies.
\end{exmp}

\begin{figure}[ht]
  \centering
  \begin{tikzpicture}[
      x=1cm,
      y=1cm,
      vertex/.style={
          circle,
          fill=black,
          inner sep=1.9pt
        },
      tile/.style={
          draw,
          rounded corners=2pt,
          minimum width=2.75cm,
          minimum height=1.05cm,
          align=center
        },
      copy/.style={
          draw,
          rounded corners=1.5pt,
          minimum width=1.55cm,
          minimum height=0.5cm,
          align=center
        },
      bratteli/.style={
          ->,
          line width=0.75pt
        },
      address/.style={
          ->,
          dashed,
          line width=0.7pt
        }
    ]


    \node[vertex] (top) at (0,6.1) {};
    \node[above=5pt] at (top) {$0_1$};

    \node[vertex] (bottom00) at (-3.3,4.7) {};
    \node[vertex] (bottom1)  at (0,4.7) {};
    \node[vertex] (bottom2)  at (3.3,4.7) {};

    \node[below=5pt] at (bottom00) {$0_0$};
    \node[below=5pt] at (bottom1)  {$1$};
    \node[below=5pt] at (bottom2)  {$2$};

    \node[left=1.2cm] at (top) {$V_{n+1}$};
    \node[left=1.2cm] at (bottom00) {$V_{n+2}$};

    \draw[bratteli]
    (top) --
    node[pos=0.55,above left=1pt] {$e_{0_0}$}
    (bottom00);

    \draw[bratteli]
    (top) --
    node[pos=0.52,right=3pt] {$e_1$}
    (bottom1);

    \draw[bratteli]
    (top) --
    node[pos=0.55,above right=1pt] {$e_2$}
    (bottom2);


    \node[tile] (model) at (0,2.9)
    {$Y=\mathcal{T}_{0_1,n}$};

    \node[above=3pt] at (model.north)
    {$Y$ at level $n$};


    \node[tile] (tile00) at (-3.3,0.4) {};
    \node[tile] (tile1)  at (0,0.4) {};
    \node[tile] (tile2)  at (3.3,0.4) {};

    \node[below=3pt] at (tile00.south)
    {$\mathcal{T}_{0_0,n+1}$};

    \node[below=3pt] at (tile1.south)
    {$\mathcal{T}_{1,n+1}$};

    \node[below=3pt] at (tile2.south)
    {$\mathcal{T}_{2,n+1}$};

    \node[copy] (copy00) at (-3.3,0.4)
    {$Ye_{0_0}$};

    \node[copy] (copy1) at (0,0.4)
    {$Ye_1$};

    \node[copy] (copy2) at (3.3,0.4)
    {$Ye_2$};


    \draw[address]
    (model.south west)
    to[out=-125,in=90]
    node[pos=0.1,left=4pt] {$\phi_{0_0}$}
    (copy00.north);

    \draw[address]
    (model.south)
    --
    node[pos=0.32,right=4pt] {$\phi_1$}
    (copy1.north);

    \draw[address]
    (model.south east)
    to[out=-55,in=90]
    node[pos=0.1,right=4pt] {$\phi_2$}
    (copy2.north);

  \end{tikzpicture}
  \caption{A synchronized alternating datum arising from the diagonal embedding of
    $Y=\mathcal{T}_{0_1,n}$.}
  \label{fig:FG-descendant-datum}
\end{figure}
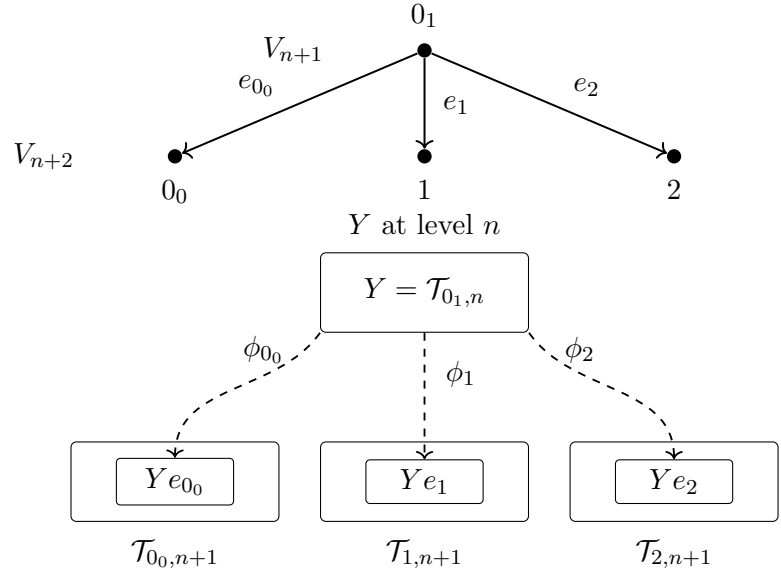

\begin{prop}
  \label{prop:diagonal-inheritance}
  Let $\mathcal{D}=(Y;\phi_1,\ldots,\phi_k)$ be an alternating datum at level $n$. Suppose
  $\phi_i(Y)\subseteq\mathcal{T}_{v_i,n}$ for $i=1,\ldots,k$. For every $e\in E_{n+1}$ with $\mathbf{s}(e)=v_i$, define
  \[\phi_{i,e}:Y\longrightarrow\Omega_{n+1}(\mathsf{B}),
    \qquad
    \phi_{i,e}(\gamma)=\phi_i(\gamma)e.\]
  Then
  \[
    \iota_{n,n+1}\bigl(\mathsf{A}(\mathcal{D})\bigr)=
    \left\{\prod_{i=1}^k\prod_{\substack{e\in E_{n+1}\\\mathbf{s}(e)=v_i}}
    \phi_{i,e}\sigma\phi_{i,e}^{-1}:\sigma\in\mathsf{A}(Y)\right\}.
  \]
  In particular, $\iota_{n,n+1}(\mathsf{A}(\mathcal{D}))$ is the diagonal alternating group on all descendant copies of the coordinate supports of $\mathcal{D}$.
\end{prop}

\begin{proof}
  Let $\sigma\in\mathsf{A}(Y)$.  The corresponding element of $\mathsf{A}(\mathcal{D})$ acts on the coordinate support $\phi_i(Y)$ as $\phi_i\sigma\phi_i^{-1}$. Under the diagonal embedding $\iota_{n,n+1}$, this permutation is copied to $\phi_i(Y)e$ for every edge $e$ with $\mathbf{s}(e)=v_i$. Under the identification $\phi_{i,e}$, each such copy is again the permutation $\sigma$. Since these descendant supports are pairwise disjoint, the displayed formula follows.
\end{proof}

The formula specifies all copies of the permutation under refinement. At the next level, we generate alternating groups on unions of these supports, with the diagonal identifications specified by the target datum.

\subsection{Connectors, selectors and propagation}
\label{subsec:connectors-selectors-propagation}

For each $n$, choose a finite family $\mathscr{D}_n$ of alternating data and a subgroup $\mathsf{C}_n\leq\mathsf{A}_n$ containing their alternating groups and the auxiliary alternating groups used to move connecting points at the next refinement.

For a target datum $\mathcal{D}'=(Y';\psi_1,\ldots,\psi_\ell)\in\mathscr{D}_{n+1}$, choose data $\mathcal{D}_1,\ldots,\mathcal{D}_r\in\mathscr{D}_n$ whose descendant supports construct $\mathcal{D}'$. By Proposition~\ref{prop:diagonal-inheritance}, the groups $\iota_{n,n+1}(\mathsf{A}(\mathcal{D}_j))$, $j=1,\ldots,r$, act diagonally on the corresponding descendant copies.

For each coordinate support $\psi_i(Y')$, we obtain finite subsets $Z_{i,1},\ldots,Z_{i,r_i} \subseteq \psi_i(Y')$ from these descendant copies. Their union is $\psi_i(Y')$. If $\mathcal{D}'$ is synchronized, the decompositions of the different coordinate supports correspond under the descendant embeddings $\psi_1,\ldots,\psi_\ell$.

An inherited alternating group may be supported on a single set $Z_{i,j}$ or may act diagonally on several such sets. We use the connectors to produce $3$-cycles whose supports meet different descendant copies. Once these cycles and the alternating groups supported on the individual sets have been constructed, Corollary~\ref{cor:connected-cover-gluing} applies when their supports cover $\psi_i(Y')$ and have connected intersection graph. For a synchronized target, we perform the same operations in every coordinate and apply the diagonal gluing lemmas.

\begin{defn}[Active and protected connectors]\label{def:active-protect}

  Fix a target datum
$\mathcal{D}'\in\mathscr{D}_{n+1}$ and the descendant supports used to
  construct it.  Let $\Sigma\subseteq\Omega_{n+1}(\mathsf{B})$ be the
  union of these supports, and put
  $[\Sigma]:=\bigsqcup_{\gamma\in\Sigma}[\gamma]$.

  Let $\mathsf{a}\subseteq\mathsf{e}_n$ be the union of the connectors used to glue these supports. We call the
  connectors of $\mathsf{a}$ \emph{active}. Let $\pi$ be the prescribed permutation of the connecting points
  $V(\mathsf{a})$.

  We also choose a union $\mathsf{p}$ of connectors, called the \emph{protected connectors}. These may belong to $\mathsf{e}_n$, $\mathsf{e}_{n-1}$, or $\mathsf{e}_{n+1}$. Their neighborhoods must be fixed pointwise during the present conjugation. They include:

  \begin{enumerate}
    \item
          connectors used to produce a different target datum at the same
          level of $\mathsf{B}$ which must remain unchanged;

    \item
          connectors from an adjacent level whose neighborhoods must remain fixed in the present induction step.
  \end{enumerate}

  We only include such connectors when their connector neighborhoods meet the supports involved in the present gluing
  argument. We require $U_{\mathsf{a}}\cap U_{\mathsf{p}}=\varnothing$. Connector neighborhoods disjoint from $[\Sigma]$
  impose no condition.
\end{defn}
\begin{defn}[Selector]\label{def:selector}
  A \emph{selector} for a gluing argument is an element $s\in\widehat{G}$ satisfying the following conditions.

  \begin{enumerate}[label=\textup{(E\arabic*)},ref=\textup{(E\arabic*)}]
    \item\label{cond:selector:E1}
          For every $p\in V(\mathsf{a})$ and every compatible infinite tail $w$,
          we have
          $s(pw)=\pi(p)w$.
          If the active connector graph consists of several synchronized copies, then $\pi$ is the same permutation in every copy
          under the identifications given by the descendant embeddings.

    \item\label{cond:selector:E2}
          The element $s$ acts identically on $U_{\mathsf{p}}$.
  \end{enumerate}
\end{defn}

No condition is imposed on the action of a selector away from the active and protected connector neighborhoods. In an argument using conjugation by a selector, any additional points that must remain fixed are specified and checked separately. Proposition~\ref{prop:selector-bridge-cycles} is the basic instance of this requirement.

The following calculation joins two supports in each coordinate by conjugating a product of $3$-cycles.

\begin{prop}
  \label{prop:selector-bridge-cycles}
  Let $I$ be a finite set.  For every $i\in I$, let $U_i$ and $V_i$ be
  disjoint finite sets that are identified across the indices $i$.  Choose
  corresponding points
  $p_i,x_i,y_i\in U_i$ and $q_i\in V_i$, where
  $p_i,x_i,y_i$ are distinct.

  Suppose that a subgroup $H$ contains the synchronized groups
  \[
    \Delta_{i\in I}\mathsf{A}(U_i),
    \qquad
    \Delta_{i\in I}\mathsf{A}(V_i),
  \]
  and let $s$ be a selector satisfying $s(p_i)=q_i$, $s(x_i)=x_i$ and $s(y_i)=y_i$. Then
  \[
    \Delta_{i\in I}
    \mathsf{A}(U_i\sqcup V_i)
    \leq
    \langle H,s\rangle.
  \]
\end{prop}

\begin{proof}
  The first synchronized alternating group contains
  $\alpha:=\prod_{i\in I}(p_i\,x_i\,y_i)$.
  Since $s(p_i)=q_i$ and $s$ fixes $x_i,y_i$ for every $i\in I$, we have
  \[s\alpha s^{-1}=\prod_{i\in I}(q_i\,x_i\,y_i).\]
  This is a product of corresponding $3$-cycles, each meeting both $U_i$ and $V_i$. The result follows from
  Corollary~\ref{cor:diagonal-one-bridge-gluing}.
\end{proof}

Repeated applications recover alternating groups on larger unions once the supports form a connected cover. In applications, selectors conjugate inherited $3$-cycles across descendant copies. Conjugation by elements of $\iota_{n,n+1}(\mathsf{C}_n)$ varies the connecting points. We cancel identical factors when necessary, and the gluing lemmas give the inclusion in Condition~\ref{cond:selector-system:A2}. This is the \emph{propagation} of the alternating data.

\subsection{Eventual absorption and refinement}
\label{subsec:bounded-absorption}

Some alternating groups remain synchronized across descendant supports after propagation. To deduce containment from this propagation, we require their eventual recovery at each finite level; a uniform absorption delay is a stronger condition.


\begin{defn}[Alternating propagation system]
  \label{def:alternating-propagation-system}
  An \emph{alternating propagation system} for $\widehat{G}$ consists of an integer $n_0$, finite families of alternating data $\mathscr{D}_n$, and groups $\mathsf{C}_n\leq\mathsf{A}_n$ for $n\geq n_0$. We require that $\mathsf{C}_n$ contain $\mathsf{A}(\mathcal{D})$ for every $\mathcal{D}\in\mathscr{D}_n$, together with every auxiliary subgroup at a finite level used at level $n$, and that the following conditions hold.

  \begin{enumerate}[label=\textup{(A\arabic*)},ref=\textup{(A\arabic*)}]
    \item\label{cond:selector-system:A1}
    \emph{Finite seed.}
          The initial group is contained in $\widehat{G}$: $\mathsf{C}_{n_0}\leq\widehat{G}$.

    \item\label{cond:selector-system:A2}
    \emph{Propagation.}
          For every $n\geq n_0$, there are elements $s_{n,1},\ldots,s_{n,k_n}\in\widehat{G}$ such that
          \[
          \mathsf{C}_{n+1}\leq
          \left\langle
          \iota_{n,n+1}(\mathsf{C}_n),
          s_{n,1},\ldots,s_{n,k_n}
          \right\rangle.
          \]
  \end{enumerate}
\end{defn}

\begin{prop}
  \label{prop:propagation-contained-in-G}
  If $\{(\mathscr{D}_n,\mathsf{C}_n)\}_{n\geq n_0}$ is an alternating propagation system, then $\mathsf{C}_n\leq\widehat{G}$ for every $n\geq n_0$.
\end{prop}

\begin{proof}
  By Condition~\ref{cond:selector-system:A1}, we have $\mathsf{C}_{n_0}\leq\widehat{G}$. Suppose that $\mathsf{C}_n\leq\widehat{G}$. The group $\iota_{n,n+1}(\mathsf{C}_n)$ consists of the same homeomorphisms written at level $n+1$, and the elements $s_{n,1},\ldots,s_{n,k_n}$ in Condition~\ref{cond:selector-system:A2} belong to $\widehat{G}$. Hence $\mathsf{C}_{n+1}\leq\widehat{G}$, and the result follows by induction.
\end{proof}

For $n\geq n_0$, put
\[
  \mathsf{C}_n^{\uparrow}:=
  \left\langle
  \iota_{m,n}(\mathsf{C}_m):n_0\leq m\leq n
  \right\rangle
  \leq\mathsf{A}_n.
\]
By Proposition~\ref{prop:propagation-contained-in-G}, every $\mathsf{C}_n^{\uparrow}$ is contained in $\widehat{G}$. Moreover, $\iota_{n,n+1}(\mathsf{C}_n^{\uparrow})\leq\mathsf{C}_{n+1}^{\uparrow}$.
Thus $\mathsf{C}_n^{\uparrow}$ contains the images of all groups $\mathsf{C}_m$ constructed at levels $m\leq n$.

\begin{defn}[Eventual and bounded absorption]
  \label{def:eventual-bounded-absorption}
  An alternating propagation system is \emph{eventually absorbing} if, for every $n\geq n_0$, there is an integer $d_n\geq0$ such that
  \[
  \iota_{n,n+d_n}(\mathsf{A}_n)\leq\mathsf{C}_{n+d_n}^{\uparrow}.
  \]
  It is \emph{boundedly absorbing} if the integers $d_n$ can be chosen with $d_n\leq D$ for one constant $D$ independent of $n$. The smallest such $D$ is the \emph{absorption delay}. We use the convention that $\iota_{n,n}=\Id$.
\end{defn}

\begin{prop}\label{thm:selector-bootstrap}
  Suppose that $\widehat{G}$ admits an eventually absorbing alternating propagation system. Then
  \[
  \mathsf{A}(\mathfrak{T}_{\mathsf{B}})\leq\widehat{G}.
  \]
\end{prop}

\begin{proof}
  Let $g\in\mathsf{A}(\mathfrak{T}_{\mathsf{B}})$. After passing to a deeper level if necessary, we may assume that $g\in\mathsf{A}_n$ for some $n\geq n_0$. Choose $d_n$ as in Definition~\ref{def:eventual-bounded-absorption}. Then $\iota_{n,n+d_n}(g)\in\mathsf{C}_{n+d_n}^{\uparrow}\leq\widehat{G}$.
  The elements $g$ and $\iota_{n,n+d_n}(g)$ define the same homeomorphism of $\Omega(\mathsf{B})$. Hence $g\in\widehat{G}$.
\end{proof}

To bound the absorption delay, we record inclusions after one refinement among groups associated with alternating data. The collections below may include independent or synchronized data, and need not be the canonical blocks of a subdirect product.

For $v\in V_{n+1}$, let $\mathcal{U}_{v,n}:=(\mathcal{T}_{v,n};\Id)$ be the canonical independent alternating datum associated with the level-$n$ tile $\mathcal{T}_{v,n}$.

\begin{defn}[Finite refinement scheme]
  \label{def:finite-refinement-scheme}
  Let $\{(\mathscr{D}_n,\mathsf{C}_n)\}_{n\geq n_0}$ be an alternating propagation system. A \emph{finite refinement scheme} consists of a finite directed graph $\mathcal{Q}$ and, for every vertex $q\in V(\mathcal{Q})$ and every $n\geq n_0$, a collection $\mathscr{R}_n(q)$ of alternating data at level $n$ satisfying the following conditions.

  \begin{enumerate}[label=\textup{(R\arabic*)},ref=\textup{(R\arabic*)}]
    \item\label{cond:refinement:R1}
          For every $n\geq n_0$ and every $v\in V_{n+1}$, a vertex $q(v,n)\in V(\mathcal{Q})$ is specified such that the canonical datum $\mathcal{U}_{v,n}$ belongs to $\mathscr{R}_n(q(v,n))$.

    \item\label{cond:refinement:R2}
          Let $\mathcal{R}\in\mathscr{R}_n(q)$. There are vertices $q_1,\ldots,q_s$ with directed edges $q\to q_j$ in $\mathcal{Q}$ and data $\mathcal{R}_j\in\mathscr{R}_{n+1}(q_j)$ such that
          \[
          \iota_{n,n+1}\bigl(\mathsf{A}(\mathcal{R})\bigr)
          \leq
          \left\langle
          \mathsf{C}_{n+1}^{\uparrow},
          \mathsf{A}(\mathcal{R}_1),\ldots,
          \mathsf{A}(\mathcal{R}_s)
          \right\rangle.
          \]
          The case $s=0$ means that
          $\iota_{n,n+1}(\mathsf{A}(\mathcal{R}))\leq\mathsf{C}_{n+1}^{\uparrow}$.
  \end{enumerate}

  A vertex of $\mathcal{Q}$ is \emph{initial} if it occurs among the specified vertices $q(v,n)$ in Condition~\ref{cond:refinement:R1}. We call the subgraph reachable from the initial vertices the \emph{reachable graph}.
\end{defn}

The data in $\mathscr{R}_n(q)$ need not belong to $\mathscr{D}_n$. Their associated groups may already be contained in $\widehat{G}$. Condition~\ref{cond:refinement:R2} requires only an inclusion: the image of $\mathsf{A}(\mathcal{R})$ under one refinement is contained in the group generated by $\mathsf{C}_{n+1}^{\uparrow}$ and the groups $\mathsf{A}(\mathcal{R}_j)$ assigned to successor vertices.

The following theorem restates Theorem~A.

\begin{thm}
  \label{thm:finite-refinement-containment}
  Let $\widehat{G}$ be a group of bounded type over a simple Bratteli diagram $\mathsf{B}$. Suppose that $\widehat{G}$ admits an alternating propagation system with a finite refinement scheme whose reachable graph is acyclic. Then
  \[
    \mathsf{A}(\mathfrak{T}_{\mathsf{B}})\leq\widehat{G}.
  \]
  Let $L$ be the maximal length of a directed path in the reachable graph. Then the absorption delay is at most $L+1$.
\end{thm}

\begin{proof}
  For a vertex $q$ in the reachable graph, let $h(q)$ be the maximal number of directed edges in a path starting at $q$. Since the reachable graph is finite and acyclic, $h(q)$ is finite.

  We claim that, for every $n\geq n_0$, every reachable vertex $q$, and every $\mathcal{R}\in\mathscr{R}_n(q)$,
  \begin{equation}
  \iota_{n,n+h(q)+1}\bigl(\mathsf{A}(\mathcal{R})\bigr)
  \leq
  \mathsf{C}_{n+h(q)+1}^{\uparrow}.
  \label{eq:refinement-height}
  \end{equation}
  We argue by induction on $h(q)$. If $h(q)=0$, then $q$ has no outgoing edge in the reachable graph, and Condition~\ref{cond:refinement:R2} gives
  $\iota_{n,n+1}(\mathsf{A}(\mathcal{R}))\leq\mathsf{C}_{n+1}^{\uparrow}$.

  Suppose that $h(q)>0$. Apply Condition~\ref{cond:refinement:R2}. Every successor $q_j$ satisfies $h(q_j)<h(q)$. By induction, each group $\mathsf{A}(\mathcal{R}_j)$ is absorbed after at most $h(q_j)+1\leq h(q)$ further refinements. The diagonal embeddings send $\mathsf{C}_{n+1}^{\uparrow}$ into $\mathsf{C}_m^{\uparrow}$ for every $m\geq n+1$. Refining the containment in Condition~\ref{cond:refinement:R2} to level $n+h(q)+1$ therefore gives~\eqref{eq:refinement-height}.

  Let $L$ be the maximal length of a directed path in the reachable graph. By Condition~\ref{cond:refinement:R1}, every factor $\mathsf{A}(\mathcal{T}_{v,n})$ of $\mathsf{A}_n$ is represented by a datum assigned to an initial vertex, and every such vertex has height at most $L$. Hence $\iota_{n,n+L+1}(\mathsf{A}_n)\leq\mathsf{C}_{n+L+1}^{\uparrow}$ for every $n\geq n_0$. Thus the absorption delay is at most $L+1$. The containment follows from Proposition~\ref{thm:selector-bootstrap}.
\end{proof}

\begin{rmk}
  A finite refinement scheme gives a sufficient bound, not a characterization of containment. Condition~\ref{cond:refinement:R2} is one-sided, and the graph may have unused edges. Thus a directed cycle does not preclude containment. Every required inclusion must hold at all levels to which its rule applies; checking finitely many initial levels is insufficient. Once these inclusions hold, the graph gives the bound in Theorem~A. By contrast, Theorem~\ref{thm:regular-covering-containment} computes the exact intersection from the adjoined AF permutations and the finite deck group.
\end{rmk}

\begin{exmp}
  If $\mathsf{C}_n^{\uparrow}=\mathsf{A}_n$ for every $n\geq n_0$, then the absorption delay is $0$. This is the case for the fragmented golden mean dihedral group in \cite[Section~8]{nekrash18} and the fragmented modified LMS-group in \cite[Section~5]{kua26no1}.

  For the modified Fabrykowski--Gupta group, the graph in Lemma~\ref{lem:FG-refinement-graph} has maximal path length one. Theorem~A gives absorption delay at most $2$.
\end{exmp}

\subsection{An intrinsic obstruction from synchronization}
\label{subsec:intrinsic-synchronization-obstruction}

Unlike the chosen refinement scheme, the following obstruction is defined from $\widehat{G}\cap\mathsf{A}_n$ itself and concerns all its elements preserving the chosen supports.

Let $Y_1,\ldots,Y_k$ be pairwise disjoint subsets of level-$n$ tiles, canonically identified with the same finite set $Y$, where $|Y|\geq5$. Put
\[
  \mathsf{K}^{+}(Y_1,\ldots,Y_k):=
  \{g\in \widehat{G}\cap\mathsf{A}_n:g(Y_i)=Y_i\text{ and }g|_{Y_i}\in\mathsf{A}(Y_i)\text{ for every }i\}.
\]
Restriction to the $Y_i$ gives a homomorphism
\[
  \rho_Y:\mathsf{K}^{+}(Y_1,\ldots,Y_k)
  \longrightarrow
  \prod_{i=1}^k\mathsf{A}(Y_i).
\]
\begin{rmk}
  By Lemma~\ref{lem:subdirect-simple-decomposition}, if the image of $\rho_Y$ is a proper subdirect product, then at least two of the sets $Y_i$ belong to the same canonical diagonal block of this subdirect product.
\end{rmk}

\begin{prop}
  \label{prop:intrinsic-synchronization-obstruction}
  If the image of $\rho_Y$ is a proper subdirect product of $\prod_i\mathsf{A}(Y_i)$, then
  \[
    \mathsf{A}(\mathfrak{T}_{\mathsf{B}})\nleq \widehat{G}.
  \]
\end{prop}

\begin{proof}
  Suppose instead that $\mathsf{A}(\mathfrak{T}_{\mathsf{B}})\leq\widehat{G}$. After passing to the level on which the sets $Y_i$ are defined, every even permutation supported on one $Y_i$ belongs to $\widehat{G}\cap\mathsf{A}_n$. Hence
  \[
    \prod_{i=1}^k\mathsf{A}(Y_i)
    \leq
    \rho_Y\bigl(\mathsf{K}^{+}(Y_1,\ldots,Y_k)\bigr),
  \]
  contradicting the assumption that this image is a proper subdirect product.
\end{proof}

Proposition~\ref{prop:intrinsic-synchronization-obstruction} tests $\widehat{G}\cap\mathsf{A}_n$ itself. A proper subgroup obtained from a chosen propagation does not give the same obstruction.

\subsection{Finite regular coverings and an exact containment criterion}
\label{subsec:regular-covering-containment}

Group-valued edge labels define skew-product Bratteli diagrams \cite[Subsection~2.2]{ElKacimiParthasarathy09}. We use an equivalent convention for source and range.

\begin{defn}[Finite regular covering]
  \label{def:finite-regular-covering}
  Let $\mathsf{B}$ be the \emph{base Bratteli diagram}. Let $\Lambda$ be a finite group and assign a label $\upsilon(e)\in\Lambda$ to every edge $e$ of $\mathsf{B}$. Define a Bratteli diagram $\mathsf{B}^{\Lambda}$ with vertices $(v,u)\in V_n\times\Lambda$ at level $n$ and edges $(e,u)$ with
  \[
    \mathbf{s}(e,u)=(\mathbf{s}(e),u),\qquad
    \mathbf{r}(e,u)=(\mathbf{r}(e),u\upsilon(e)).
  \]
  The maps $p(v,u)=v$ and $p(e,u)=e$ define a map $p:\mathsf{B}^{\Lambda}\to\mathsf{B}$. We call $p$ a \emph{finite regular covering} of $\mathsf{B}$ with \emph{deck group} $\Lambda$. A \emph{lift} of a finite or infinite path $\gamma$ in $\mathsf{B}$ is a path $\widetilde{\gamma}$ in $\mathsf{B}^{\Lambda}$ such that $p(\widetilde{\gamma})=\gamma$. Once the fiber coordinate of its initial vertex is specified, the lift is unique. For $h\in\Lambda$, the \emph{deck transformation} $\delta_h$ is given by
  \[
    \delta_h(v,u)=(v,hu),\qquad \delta_h(e,u)=(e,hu).
  \]
  The maps $\delta_h$, $h\in\Lambda$, form the group $\delta(\Lambda)$ of deck transformations of the covering.
\end{defn}

Thus $\Lambda$ denotes the abstract deck group, while $\delta(\Lambda)$ denotes the corresponding group of deck transformations acting on the space of infinite paths of $\mathsf{B}^{\Lambda}$. The right multiplication in the range formula commutes with this left action. The deck transformations therefore preserve path length and act freely on $\Omega(\mathsf{B}^{\Lambda})$. For a finite path $\eta=e_1\cdots e_k$, put $\upsilon(\eta)=\upsilon(e_1)\cdots\upsilon(e_k)$. Its lift starting in fiber $u$ ends in fiber $u\upsilon(\eta)$.

Write $\mathcal{T}^{\Lambda}_{v,u,n}$ for the set of level-$n$ paths of $\mathsf{B}^{\Lambda}$ ending at $(v,u)$. A base path has a unique lift with a prescribed terminal fiber. Denote this lifting bijection by $\Phi_{v,u,n}:\mathcal{T}_{v,n}\longrightarrow\mathcal{T}^{\Lambda}_{v,u,n}$. Thus $\Phi_{v,u,n}$ is the inverse of the forgetting map on this tile. Put
\[
  \mathsf{A}^{\Lambda}_n=\prod_{v\in V_{n+1},\,u\in\Lambda}
  \mathsf{A}(\mathcal{T}^{\Lambda}_{v,u,n}),\qquad
  \mathsf{S}^{\Lambda}_n=\prod_{v\in V_{n+1},\,u\in\Lambda}
  \mathsf{S}(\mathcal{T}^{\Lambda}_{v,u,n}).
\]
The \emph{canonical lift} of $(\sigma_v)_v\in\mathsf{A}_n$ is the element of $\mathsf{A}^{\Lambda}_n$ whose restriction to $\mathcal{T}^{\Lambda}_{v,u,n}$ is $\Phi_{v,u,n}\sigma_v\Phi_{v,u,n}^{-1}$ for every $v$ and $u$. These lifts are compatible with the Bratteli embeddings: appending an edge changes the terminal fiber by right multiplication but copies the same permutation in every fiber. They define an injective homomorphism
\[
  \operatorname{Lift}_{\Lambda}:\mathsf{A}(\mathfrak{T}_{\mathsf{B}})
  \longrightarrow\mathsf{A}(\mathfrak{T}_{\mathsf{B}^{\Lambda}}).
\]
Its image commutes with every $\delta_h$.

For a subgroup $J\leq\operatorname{Homeo}(\Omega(\mathsf{B}^{\Lambda}))$ and $L\leq\Lambda$, write
$\operatorname{C}_J(\delta(L))=\{g\in J:g\delta_h=\delta_hg\text{ for every }h\in L\}$.

The following restates Theorem~B in the notation of this subsection.

\begin{thm}
  \label{thm:regular-covering-containment}
  Assume that $\mathsf{B}^{\Lambda}$ is simple and its space of infinite paths is a Cantor set. Let $\Gamma_0$ be a group of homeomorphisms of this space such that
  \[
    \operatorname{Lift}_{\Lambda}\bigl(\mathsf{A}(\mathfrak{T}_{\mathsf{B}})\bigr)
    \leq\Gamma_0
    \leq\operatorname{C}_{\operatorname{Homeo}(\Omega(\mathsf{B}^{\Lambda}))}(\delta(\Lambda)).
  \]
  Let $s_1,\ldots,s_r\in\mathsf{S}(\mathfrak{T}_{\mathsf{B}^{\Lambda}})$, put $\Gamma=\langle\Gamma_0,s_1,\ldots,s_r\rangle$, and define
  \begin{equation}
    \Lambda_0=\{h\in\Lambda:s_j\delta_h=\delta_hs_j\text{ for }1\leq j\leq r\}.
    \label{eq:surviving-deck-group}
  \end{equation}
  Then
  \begin{equation}
    \Gamma\cap\mathsf{A}(\mathfrak{T}_{\mathsf{B}^{\Lambda}})
    =\operatorname{C}_{\mathsf{A}(\mathfrak{T}_{\mathsf{B}^{\Lambda}})}(\delta(\Lambda_0)).
    \label{eq:cover-AF-intersection}
  \end{equation}
  Consequently,
  \begin{equation}
    \mathsf{A}(\mathfrak{T}_{\mathsf{B}^{\Lambda}})\leq\Gamma
    \quad\Longleftrightarrow\quad \Lambda_0=\{1\}.
    \label{eq:cover-containment-iff}
  \end{equation}
  The subgroup $\Lambda_0$ is determined by the finite permutations of the $s_j$ at any common finite level.
\end{thm}

The next three lemmas compute $\Lambda_0$, compare the restrictions of the adjoined permutations on later tiles, and determine the resulting diagonal groups.

\begin{lemma}
  \label{lem:finite-deck-test}
  Choose a common finite level $N$ for $s_1,\ldots,s_r$. For $v\in V_{N+1}$ and $u\in\Lambda$, write
  \[
    \sigma_{j,v,u}=\Phi_{v,u,N}^{-1}
    \bigl(s_j|_{\mathcal{T}^{\Lambda}_{v,u,N}}\bigr)
    \Phi_{v,u,N}\in\mathsf{S}(\mathcal{T}_{v,N}).
  \]
  Then, for $h\in\Lambda$,
  \begin{equation}
    h\in\Lambda_0
    \quad\Longleftrightarrow\quad
    \sigma_{j,v,hu}=\sigma_{j,v,u}
    \text{ for every }j,v,u.
    \label{eq:finite-deck-test}
  \end{equation}
\end{lemma}

\begin{proof}
  The deck transformation $\delta_h$ sends
  $\Phi_{v,u,N}(\gamma)$ to $\Phi_{v,hu,N}(\gamma)$ for every
  $\gamma\in\mathcal{T}_{v,N}$. Since each $s_j$ is represented at level $N$, it preserves every tile
  $\mathcal{T}^{\Lambda}_{v,u,N}$. Hence the equality
  $s_j\delta_h=\delta_hs_j$ on the union of the level-$N$ cylinders ending at
  $(v,u)$ is equivalent to
  $\sigma_{j,v,hu}=\sigma_{j,v,u}$. These cylinders partition the unit space,
  which proves the equivalence.
\end{proof}

\begin{lemma}
  \label{lem:covering-discrepancy-propagation}
  There is $M>N$ such that every tile of the base at every level $m\geq M$ has at least five vertices and, for every such $m$, every $w\in V_{m+1}$ and every $c,h\in\Lambda$, the following are equivalent.
  \begin{enumerate}
    \item $h\in\Lambda_0$.
    \item For every $j$, the restrictions of $s_j$ to
          $\mathcal{T}^{\Lambda}_{w,c,m}$ and
          $\mathcal{T}^{\Lambda}_{w,hc,m}$ agree under the canonical
          identifications with $\mathcal{T}_{w,m}$.
  \end{enumerate}
\end{lemma}

Thus, if $h\notin\Lambda_0$, every sufficiently deep tile of the base has lifts on which one of the $s_j$ has different restrictions under $\delta_h$.

\begin{proof}
  The space $\Omega(\mathsf{B})$ of infinite paths is infinite, since the forgetting map from $\Omega(\mathsf{B}^{\Lambda})$ is a $|\Lambda|$-to-one surjection. Hence there is a level $n\geq N$ with at least five paths of length $n$ in $\mathsf{B}$. By simplicity of $\mathsf{B}^{\Lambda}$, increase to $M>n$ so that every vertex of the covering at level $n+1$ is joined to every vertex of the covering at level $M+1$. Forgetting the fiber coordinate shows that every vertex of the base at level $M+1$ receives a path from every vertex of the base at level $n+1$. Hence every tile of the base at level $M$ contains at least the total number of paths of length $n$ in $\mathsf{B}$ and therefore has at least five vertices. Every later tile contains a copy of a level-$M$ tile, and thus has at least five vertices.

  After increasing $M$ once more if necessary, every vertex of the covering at level $N+1$ is joined to every vertex of the covering at level $M+1$. This remains true for every $m\geq M$: every vertex at level $m+1$ has a predecessor at level $M+1$, and every vertex at level $N+1$ reaches that predecessor.

  If $h\in\Lambda_0$, the equality of the restrictions follows from
  $s_j\delta_h=\delta_hs_j$. Conversely, suppose $h\notin\Lambda_0$.
  Lemma~\ref{lem:finite-deck-test} gives $j,v,u$ and
  $\gamma\in\mathcal{T}_{v,N}$ such that
  $\sigma_{j,v,hu}(\gamma)\neq\sigma_{j,v,u}(\gamma)$.
  Fix $m\geq M$, $w\in V_{m+1}$ and $c\in\Lambda$. Choose a covering path
  from $(v,u)$ to $(w,c)$, and let $\eta$ be its underlying base path. Its image under $\delta_h$ is the lift of the same path $\eta$ from $(v,hu)$ to $(w,hc)$. Under refinement to level $m$, the two restrictions of
  $s_j$ send the base path $\gamma\eta$ to $\sigma_{j,v,u}(\gamma)\eta$ and $\sigma_{j,v,hu}(\gamma)\eta$, respectively. These paths are distinct. Thus the two restrictions do not
  agree. Since $w$ and $c$ were arbitrary, the asserted difference occurs over every sufficiently deep tile of the base.
\end{proof}

\begin{lemma}
  \label{lem:finite-level-cover-centralizer}
  With $M$ as in Lemma~\ref{lem:covering-discrepancy-propagation}, for every
  $m\geq M$,
  \begin{equation}
    \operatorname{C}_{\mathsf{A}^{\Lambda}_m}(\delta(\Lambda_0))
    \leq \Gamma.
    \label{eq:finite-cover-centralizer}
  \end{equation}
  More precisely, for $w\in V_{m+1}$ let $D_{w,m}$ be the canonical diagonal
  copy of $\mathsf{A}(\mathcal{T}_{w,m})$ on the tiles
  $\mathcal{T}^{\Lambda}_{w,c,m}$, $c\in\Lambda$, and put
  \[
    J_{w,m}=\langle D_{w,m},
    s_1D_{w,m}s_1^{-1},\ldots,s_rD_{w,m}s_r^{-1}\rangle.
  \]
  Then
  \begin{equation}
    \prod_{w\in V_{m+1}}J_{w,m}
    =\operatorname{C}_{\mathsf{A}^{\Lambda}_m}(\delta(\Lambda_0)).
    \label{eq:finite-cover-centralizer-exact}
  \end{equation}
\end{lemma}

\begin{proof}
  The group $D_{w,m}$ is contained in $\Gamma_0$, since it is the canonical lift of
  the factor $\mathsf{A}(\mathcal{T}_{w,m})$ of the AF alternating group of the base.
  Each $s_j$, represented at level $N$ and then refined to level $m$, preserves
  every tile of the covering at level $m$. Hence every conjugate
  $s_jD_{w,m}s_j^{-1}$ is contained in $\Gamma\cap\mathsf{A}^{\Lambda}_m$.

  Identify all tiles over $w$ with the common set $\mathcal{T}_{w,m}$. By
  Lemma~\ref{lem:covering-discrepancy-propagation}, the restrictions of $s_j$ to the tiles
  indexed by $c$ and $c'$ agree for every $j$ if and only if $c'=hc$ for some
  $h\in\Lambda_0$. Corollary~\ref{cor:breaking-synchronization} therefore shows
  that the canonical diagonal blocks of $J_{w,m}$ are precisely the left
  $\Lambda_0$-orbits $\Lambda_0c$. Since the original diagonal group
  $D_{w,m}$ is one of the generators, the identifications inside each block are
  the canonical ones.

  An element of $\mathsf{A}^{\Lambda}_m$ commutes with $\delta(\Lambda_0)$ exactly when, for every base vertex $w$, its permutations on the tiles of the covering agree under the canonical identifications along each orbit $\Lambda_0c$. Hence the subgroup induced by the centralizer on the tiles over $w$ is the product of the diagonal alternating groups on these orbits, which is $J_{w,m}$.
  The groups associated with different base vertices have disjoint supports, which proves
  \eqref{eq:finite-cover-centralizer-exact} and hence
  \eqref{eq:finite-cover-centralizer}.
\end{proof}

\begin{proof}[Proof of Theorem~\ref{thm:regular-covering-containment}]
  Choose a common finite level $N$ for the adjoined permutations. By
  Lemma~\ref{lem:finite-deck-test}, the subgroup $\Lambda_0$ is read directly
  from their finite permutation tables at this level.

  Every element of $\Gamma_0$ commutes with $\delta(\Lambda)$, and every
  $s_j$ commutes with $\delta(\Lambda_0)$ by definition. Hence
  \[
    \Gamma\cap\mathsf{A}(\mathfrak{T}_{\mathsf{B}^{\Lambda}})
    \leq
    \operatorname{C}_{\mathsf{A}(\mathfrak{T}_{\mathsf{B}^{\Lambda}})}
    (\delta(\Lambda_0)).
  \]

  Conversely, let $a$ belong to the centralizer on the right. Represent $a$ at
  a finite level and refine this representation to some level $m\geq M$.
  Then
  $a\in\operatorname{C}_{\mathsf{A}^{\Lambda}_m}(\delta(\Lambda_0))$. Hence
  Lemma~\ref{lem:finite-level-cover-centralizer} gives $a\in\Gamma$. This proves
  \eqref{eq:cover-AF-intersection}.

  If $\Lambda_0=\{1\}$, its centralizer is the whole AF alternating group, and containment follows. If $h\in\Lambda_0\setminus\{1\}$, take a nontrivial
  $3$-cycle supported in one tile of the covering at a sufficiently deep level. The
  deck action is free on the set of tiles over each base vertex, and thus $\delta_h$ moves its support
  to a disjoint tile. The $3$-cycle therefore does not commute with $\delta_h$
  and cannot belong to $\Gamma$. This proves
  \eqref{eq:cover-containment-iff}.
\end{proof}

For the relation of the degree-two case to \cite[Proposition~3.12]{Matui26Factor}, see the introduction.

The hypothesis concerns the canonical lift of the AF alternating group of the base. A separate application of Theorem~A may supply this base group. Equation~\eqref{eq:cover-AF-intersection} then determines exactly which permutations in the AF alternating group of the cover belong to $\Gamma$.

The calculation is finite once the common finite level $N$ and its permutation tables are given. Test~\eqref{eq:finite-deck-test} for the finitely many elements of $\Lambda$. A finite product of incidence matrices whose entries are all positive supplies the level $M$ used in the proof. If the simple diagram of the covering is stationary or periodic, such a product is found after telescoping to the period endpoints. This does not claim an algorithm for recognizing arbitrary AF elements or for proving containment in the base from unspecified generators.

\begin{cor}
  \label{cor:one-cycle-covering-containment}
  Under the hypotheses on $\mathsf{B}^{\Lambda}$ and $\Gamma_0$ in Theorem~\ref{thm:regular-covering-containment}, let $s$ be a nontrivial $3$-cycle supported in one tile of the covering at a finite level. Then
  $\mathsf{A}(\mathfrak{T}_{\mathsf{B}^{\Lambda}})\leq\langle\Gamma_0,s\rangle$.
  If $\Gamma_0$ is finitely generated, so is $\langle\Gamma_0,s\rangle$.
\end{cor}

\begin{proof}
  Every nonidentity deck transformation sends the support of $s$ to a disjoint tile of the covering. Thus no such transformation commutes with $s$, and $\Lambda_0=\{1\}$. Apply Theorem~\ref{thm:regular-covering-containment}. The final assertion follows by adjoining $s$ to a finite generating set.
\end{proof}

\begin{cor}
  \label{cor:deck-subgroup-lattice}
  Under the hypotheses on the covering in Theorem~\ref{thm:regular-covering-containment}, put
  \[
    H_L:=\operatorname{C}_{\mathsf A(\mathfrak T_{\mathsf B^\Lambda})}(\delta(L))
    \qquad(L\leq\Lambda).
  \]
  Every subgroup between $\operatorname{Lift}_\Lambda(\mathsf A(\mathfrak T_{\mathsf B}))$ and $\mathsf A(\mathfrak T_{\mathsf B^\Lambda})$ is $H_L$ for a unique subgroup $L\leq\Lambda$.
  For every subgroup $L\leq\Lambda$, there is a product $s_L$ of disjoint $3$-cycles supported in tiles at a finite level such that
  \[
    H_L=\left\langle\operatorname{Lift}_{\Lambda}
      \bigl(\mathsf A(\mathfrak T_{\mathsf B})\bigr),s_L\right\rangle.
  \]
  For $L,K\leq\Lambda$, one has
  \[
    H_K\leq H_L\quad\Longleftrightarrow\quad L\leq K,
    \qquad
    H_L\cap H_K=H_{\langle L,K\rangle},
    \qquad
    \langle H_L,H_K\rangle=H_{L\cap K}.
  \]
  All the groups $H_L$ have groupoid of germs $\mathfrak T_{\mathsf B^\Lambda}$.
  The subgroup $\operatorname{Lift}_\Lambda(\mathsf A(\mathfrak T_{\mathsf B}))$ is maximal in $\mathsf A(\mathfrak T_{\mathsf B^\Lambda})$ if and only if $|\Lambda|$ is prime.
\end{cor}

\begin{proof}
  Choose a tile of $\mathsf{B}$ $\mathcal T_{v,N}$ with at least three vertices and a nontrivial $3$-cycle $\sigma\in\mathsf A(\mathcal T_{v,N})$. Let $s_L$ restrict to the canonical lift of $\sigma$ on the tiles $\mathcal T^\Lambda_{v,u,N}$, $u\in L$, and act identically on every other tile. Conjugation by $\delta_h$ changes this set of fiber indices from $L$ to $hL$. Consequently, $s_L$ commutes with $\delta_h$ exactly when $hL=L$, or equivalently $h\in L$. Theorem~\ref{thm:regular-covering-containment}, applied with $\Gamma_0=\operatorname{Lift}_\Lambda(\mathsf A(\mathfrak T_{\mathsf B}))$, gives the first equality.

  If $L\leq K$, then $H_K\leq H_L$. Conversely, if $h\in L\setminus K$, the permutation $s_K$ belongs to $H_K$ but does not commute with $\delta_h$, and hence does not belong to $H_L$. The equality $H_L\cap H_K=H_{\langle L,K\rangle}$ follows immediately from the definition.

  For the generated subgroup, pass to a level at which every tile of $\mathsf{B}$ has at least five vertices. On the lifted tiles over each base vertex, $H_L$ and $H_K$ are products of full diagonal alternating groups, with blocks indexed respectively by the $L$-orbits $Lc$ and the $K$-orbits $Kc$ of the left deck action. Their generated subgroup is subdirect. Its projection to two fiber coordinates is diagonal precisely when the coordinates lie in the same $L$-orbit and the same $K$-orbit, hence in the same $(L\cap K)$-orbit. Scott's lemma, as stated in Lemma~\ref{lem:subdirect-simple-decomposition}, shows that the generated subgroup is the product of the diagonal alternating groups on these intersections. Passing to the direct limit gives $\langle H_L,H_K\rangle=H_{L\cap K}$.

  Let $J$ be a subgroup between $\operatorname{Lift}_\Lambda(\mathsf A(\mathfrak T_{\mathsf B}))$ and $\mathsf A(\mathfrak T_{\mathsf B^\Lambda})$, and put
  \[
    L=\{h\in\Lambda:\delta_hg=g\delta_h\text{ for every }g\in J\}.
  \]
  Since $\Lambda$ is finite, there are $s_1,\ldots,s_r\in J$ such that
  \[
    L=\{h\in\Lambda:\delta_hs_j=s_j\delta_h\text{ for }1\leq j\leq r\}.
  \]
  Indeed, for each $h\notin L$, choose an element of $J$ that does not commute with $\delta_h$. Theorem~\ref{thm:regular-covering-containment}, applied to these elements and the canonical lift of the base AF alternating group, gives
  \[
    H_L=\left\langle\operatorname{Lift}_\Lambda
      \bigl(\mathsf A(\mathfrak T_{\mathsf B})\bigr),s_1,\ldots,s_r\right\rangle
    \leq J,
  \]
  since the generated group consists of AF alternating permutations. Conversely, every element of $J$ commutes with $\delta(L)$, and thus $J\leq H_L$. This proves $J=H_L$. Uniqueness follows from the inclusion equivalence above.

  The endpoints are $H_\Lambda=\operatorname{Lift}_\Lambda(\mathsf A(\mathfrak T_{\mathsf B}))$ and $H_{\{1\}}=\mathsf A(\mathfrak T_{\mathsf B^\Lambda})$. If $\Lambda\neq\{1\}$, the lifted group is maximal exactly when $\Lambda$ has no nontrivial proper subgroup, which is equivalent to $|\Lambda|$ being prime. If $\Lambda=\{1\}$, the two endpoints coincide and the lift is not a maximal proper subgroup.

  Finally, the canonical lift of the AF alternating group of the base realizes every germ of the tail groupoid of the covering. Indeed, two tail-equivalent lifted paths have, at a sufficiently deep level, finite prefixes in the same tile of the cover. An even permutation of the corresponding tile of $\mathsf{B}$ sends one projected prefix to the other, and its canonical lift realizes the required germ. Every $H_L$ consists of AF homeomorphisms and contains these lifts. Hence its groupoid of germs is exactly $\mathfrak T_{\mathsf B^\Lambda}$.
\end{proof}

\begin{exmp}
  \label{exmp:persistent-diagonal-obstruction}
  Take the one-vertex diagram with five edges at every level and let $\Lambda=\mathbb{Z}/2\mathbb{Z}$. Label three edges by $0$ and two by $1$. The diagram of the covering has stationary incidence matrix
  $\left(\begin{smallmatrix}3&2\\2&3\end{smallmatrix}\right)$, and thus is simple. Put
  $\Gamma_0=\operatorname{Lift}_{\Lambda}(\mathsf{A}(\mathfrak{T}_{\mathsf{B}}))$.
  This group is infinite and locally finite. It is not finitely generated.

  Its groupoid of germs is precisely $\mathfrak{T}_{\mathsf{B}^{\Lambda}}$. Indeed, all its elements are AF. Conversely, let two paths in the covering be tail equivalent. At some level their prefixes end at the same vertex of the covering. Extend the common tail far enough that the corresponding tile of the base has at least three vertices. Since the alternating group of that tile of the base is transitive, an even permutation of that tile sends one extended prefix to the other. Its canonical lift has the prescribed germ between the original two paths in the covering. Hence the groupoid of germs is the whole tail groupoid. In particular, the action is minimal.

  Nevertheless, with no adjoined permutations, $\Lambda_0=\Lambda$. At every level, restriction to the two tiles of the covering over the same tile of the base gives the diagonal alternating subgroup, and Proposition~\ref{prop:intrinsic-synchronization-obstruction} already gives noncontainment. Theorem~\ref{thm:regular-covering-containment} gives the stronger statement that the entire intersection with the AF alternating group is the corresponding centralizer of the deck transformations. Thus realizing all germs of the minimal AF core does not imply containing its AF alternating group. Adjoining one $3$-cycle supported in one tile of the covering gives
  $\langle\Gamma_0,s\rangle=\mathsf{A}(\mathfrak{T}_{\mathsf{B}^{\Lambda}})$, since the reverse containment also holds for these AF generators. The positive incidence matrix allows $M=N+1$ in the preceding proof.
\end{exmp}

\subsection{Finitely generated examples from coverings}
\label{subsec:finitely-generated-covering-examples}

We next realize the same obstruction in finitely generated groups whose groupoids of germs contain the prescribed AF core.

Let $X$ be a Cantor set and $(X,\tau)$ be a minimal subshift. Fix an integer $q\geq2$ such that $\tau^q$ is minimal. On $\widetilde{X}=X\times\mathbb{Z}/q\mathbb{Z}$ define
\[
  \widetilde{\tau}(x,u)=(\tau x,u+1),\qquad
  \delta_h(x,u)=(x,u+h).
\]
The map $\widetilde{\tau}$ is minimal: its $q$-th power restricts on each fiber to $\tau^q$, and $\widetilde{\tau}$ cycles the fibers. Let $\mathfrak{G}_{\tau}$ and $\mathfrak{G}_{\widetilde{\tau}}$ be respectively the transformation groupoids of these homeomorphisms.

Every $g\in\mathsf{F}(\mathfrak{G}_{\tau})$ has a unique continuous integer-valued function $c_g$ satisfying $g(x)=\tau^{c_g(x)}x$. The formula
\[
  \ell(g)(x,u)=(g(x),u+c_g(x)\bmod q)
\]
defines an injective homomorphism into $\mathsf{F}(\mathfrak{G}_{\widetilde{\tau}})$, since $c_{gh}(x)=c_g(hx)+c_h(x)$. Every element of its image commutes with every deck transformation. Put $\Gamma_0=\ell(\mathsf{A}(\mathfrak{G}_{\tau}))$. This is finitely generated by \cite[Theorem~1.2]{nnek19}, since a minimal subshift is expansive and has infinite orbits.

Choose $x_*\in X$ and let $\mathfrak{T}\subseteq\mathfrak{G}_{\tau}$ be the AF groupoid obtained by cutting the orbit of $x_*$ between $\tau^{-1}x_*$ and $x_*$. Thus this orbit is split into its negative and nonnegative halves and all other orbits are unchanged. The groupoid $\mathfrak{T}$ and its representation by a Bratteli diagram are recalled in \cite[Section~2]{Matui06}. Lift $\mathfrak{T}$ to the open subgroupoid $\widetilde{\mathfrak{T}}\subseteq\mathfrak{G}_{\widetilde{\tau}}$ using the same displacement modulo $q$. Use the RK-sets defining this Bratteli diagram \cite[Definition~5.2.2]{nek22}. In the covering, each tower has $q$ copies. Index a lifted tower by the fiber of its bottom. If a lower-level subtower starts at height $k$ in an upper-level tower, its bottom fiber $u$ becomes $u-k$ at the upper bottom. The corresponding edge label is therefore $-k$ modulo $q$. Hence these diagrams form a finite regular covering in the sense of Definition~\ref{def:finite-regular-covering}. In these coordinates, $\ell|_{\mathsf{A}(\mathfrak{T})}$ is the canonical lift. Since $\mathsf{A}(\mathfrak{T})\leq\mathsf{A}(\mathfrak{G}_{\tau})$, the group $\Gamma_0$ contains this lift.

The lifted AF groupoid is minimal. Each of the $q$ orbits of $(x_*,u)$ under $\widetilde{\tau}$ is cut once, and all remaining orbits are unchanged. Every resulting half-orbit is dense by minimality of $\widetilde{\tau}$. Thus the Bratteli diagram of the covering is simple, and Theorem~\ref{thm:regular-covering-containment} applies.

\begin{prop}
  \label{prop:minimal-finitely-generated-covering-example}
  In this construction, $\Gamma_0$ is finitely generated, acts minimally, and has groupoid of germs $\mathfrak{G}_{\widetilde{\tau}}$, but
  $\mathsf{A}(\widetilde{\mathfrak{T}})\nleq\Gamma_0$.
  If $s$ is a nontrivial $3$-cycle supported in one tile of the covering, then $\Gamma=\langle\Gamma_0,s\rangle$ is finitely generated, has the same groupoid of germs, and contains $\mathsf{A}(\widetilde{\mathfrak{T}})$.
\end{prop}

\begin{proof}
  Only the assertion about the groupoid of germs remains to be checked. Every germ of $\tau^n$ is realized locally by an element of $\mathsf{A}(\mathfrak{G}_{\tau})$. For $n\neq0$, choose $k$ so that $x,\tau^n x,\tau^k x$ are distinct and a clopen neighborhood $U$ of $x$ whose three indicated translates are disjoint. The dynamical $3$-cycle on $U,\tau^n U,\tau^k U$ agrees with $\tau^n$ on $U$. Its lift agrees with $\widetilde{\tau}^n$ at every point of $U\times\mathbb{Z}/q\mathbb{Z}$. Thus $\Gamma_0$ realizes every germ of $\mathfrak{G}_{\widetilde{\tau}}$ and no other germ. Its action is therefore minimal. With no adjoined permutations, $\Lambda_0=\mathbb{Z}/q\mathbb{Z}$, which is nontrivial. Hence containment fails. Adjoining $s$ gives containment by Corollary~\ref{cor:one-cycle-covering-containment}. Since $s\in\mathsf{A}(\widetilde{\mathfrak{T}})\leq\mathsf{F}(\mathfrak{G}_{\widetilde{\tau}})$, adjoining it does not change the groupoid of germs.
\end{proof}

\begin{cor}
  \label{cor:finitely-generated-covering-lattice}
  In the construction of Proposition~\ref{prop:minimal-finitely-generated-covering-example}, let $\Lambda=\mathbb Z/q\mathbb Z$. For $L\leq\Lambda$, choose $s_L$ as in Corollary~\ref{cor:deck-subgroup-lattice} and put $\Gamma_L=\langle\Gamma_0,s_L\rangle$. Then every $\Gamma_L$ is finitely generated, their groupoids of germs are all $\mathfrak G_{\widetilde\tau}$, and
  \[
    \Gamma_L\cap\mathsf A(\widetilde{\mathfrak T})=H_L,
    \qquad
    \Gamma_K\leq\Gamma_L\quad\Longleftrightarrow\quad L\leq K.
  \]
\end{cor}

\begin{proof}
  Theorem~\ref{thm:regular-covering-containment} gives the intersection formula. By Proposition~\ref{prop:minimal-finitely-generated-covering-example}, the germs of $\Gamma_0$ already generate $\mathfrak G_{\widetilde\tau}$. All germs of the additional permutations belong to its AF subgroupoid, and thus adjoining these permutations does not change the groupoid of germs. If $L\leq K$, Corollary~\ref{cor:deck-subgroup-lattice} implies $s_K\in H_K\leq H_L\leq\Gamma_L$. Hence $\Gamma_K\leq\Gamma_L$. The converse follows by intersecting with $\mathsf A(\widetilde{\mathfrak T})$ and applying the same corollary.
\end{proof}

We briefly verify bounded type in this setting. Suppose the Bratteli diagram associated with the chosen increasing sequence of RK-sets has uniformly bounded rank and incidence, and fix a finite generating set of $\Gamma_0$. Each generator is a piecewise power of $\widetilde{\tau}$. Since the generating set is finite, there is $K$ such that every displacement has absolute value at most $K$. At a sufficiently deep term of the sequence, the cocycle of every generator is constant on each tower level. On a level whose height is at least $K$ away from both ends of its tower, applying the generator stays inside that tower and gives an arrow of the elementary AF subgroupoid defined by the partition into towers. Thus only the first and last $K$ levels can give arrows outside that elementary subgroupoid. There are at most $2K$ such positions per tower. The boundary paths that persist through all refinements lie over the finitely many orbit cuts defining $\widetilde{\mathfrak{T}}$, and hence there are only finitely many boundary points of the infinite tiles. Passing to the $q$-sheeted covering multiplies the bound on the boundary at finite levels by at most $q$, while adjoining an AF permutation contributes no boundary positions at sufficiently deep levels. Hence, under the hypothesis that the rank and incidence are uniformly bounded, these finitely generated actions satisfy the definition of bounded type in \cite[Definition~3.1]{kua26no1} over the lifted Bratteli diagrams.

For example, take the Fibonacci Sturmian system and any $q\geq2$. A Sturmian system is an almost one-to-one extension of an irrational rotation. See \cite[Example~6.2]{Matui06}. Every power is minimal: a nonempty closed minimal set for that power projects onto the circle, and therefore contains the preimage of every point with singleton fiber. These preimages are dense by the almost one-to-one property, and thus the minimal set is the whole subshift. A tower presentation associated with a primitive substitution can be chosen stationary after the usual construction from return words \cite{DHS99}. Choose $x_*$ as the base point of this presentation. Its height vectors evolve by a fixed integral matrix. Modulo $q$ this is a finite deterministic sequence, and the starting heights of the subtowers are fixed sums of the preceding heights. The labels on the edges of the covering are therefore eventually periodic. Thus Proposition~\ref{prop:minimal-finitely-generated-covering-example} gives an example of bounded type whose labels on the edges of the covering are eventually periodic. No assertion about intermediate growth is used in this example.

We next return to the construction of selectors for general tile inflations.

\section{Fragmentations producing selectors}
\label{sec:fragmentations-producing-selectors}

Using fragmentations from \cite{nekrash18} and the subdirect product description in \cite[Chapter~4]{JC19}, we test prescribed restrictions through finite coordinate projections. We also give a condition for purely non-Hausdorff singularities following \cite[Section~4.5]{JC19} and \cite[Subsection~2.4]{kua26}.

\subsection{Fragmentations and finite singular germs}
\label{subsec:fragmentations-finite-singular-germs}

Let $G=\langle S\rangle$ be a group of bounded type satisfying the finite singular germ condition with boundary points of the infinite tiles $\Xi=\{\xi_1,\ldots,\xi_r\}$. We use the notation of \cite{kua26no1}. In particular, the finite group of germs at $\xi_i$ is denoted $H_i$.

For a finite group $K$ of homeomorphisms of $\Omega(\mathsf{B})$, put
\[\operatorname{Fix}(K):=\{\gamma \in\Omega(\mathsf{B}):k(\gamma)=\gamma\text{ for every }k\in K\}.
\]
Let $\mathcal{P}$ be a finite partition of $\Omega(\mathsf{B})\setminus\operatorname{Fix}(K)$ into open $K$-invariant
sets.

\begin{defn}[Fragmentation]
  \label{def:fragmentation}
  A \emph{fragmentation of $K$ with respect to $\mathcal{P}$} is a finite group $\mathcal{F}$ of homeomorphisms of $\Omega(\mathsf{B})$ satisfying the following conditions.

  \begin{enumerate}
    \item
          For every $f\in\mathcal{F}$ and every $P\in\mathcal{P}$, there is $k_{f,P}\in K$ such that $f|_P=k_{f,P}|_P$.

    \item
          Every $f\in\mathcal{F}$ acts identically on $\operatorname{Fix}(K)$.

    \item
          For every $P\in\mathcal{P}$, the restriction map $\mathcal{F}\to K|_P$ is surjective.
  \end{enumerate}
\end{defn}

The restriction maps give an injective homomorphism $\mathcal{F}\longrightarrow\prod_{P\in\mathcal{P}}K|_P$. Its image projects surjectively onto every coordinate, and thus is a \emph{subdirect product}. We write
$f=(f_P)_{P\in\mathcal{P}}$, where $f_P=f|_P$.

Conversely, a finite subdirect product of the groups $K|_P$ defines a fragmentation whenever the corresponding
piecewise transformations extend continuously to $\Omega(\mathsf{B})$. When $K=\langle h\rangle$ is cyclic, this is the usual fragmentation of the finite-order transformation $h$.

For every $i$, choose a finite group $K_i\leq G_{\xi_i}$ whose natural map to $H_i$ is surjective. We assume that every element of $K_i$ has only AF germs away from $\xi_i$. Let $\mathcal{P}_i$ be a finite partition of $\Omega(\mathsf{B})\setminus\operatorname{Fix}(K_i)$ into open $K_i$-invariant sets and let $\mathcal{F}_i$ be a fragmentation of $K_i$. In the setting of tile inflation, we additionally assume that the fragmentations are implemented by replacing labels on existing boundary and connector edges without changing the finite tiles or their boundary points, and that every boundary edge of the original inflation is retained with at least one fragment label. We put
\[\widehat{G}:=\left\langle S_{\mathrm{fin}},\mathcal{F}_1,\ldots,\mathcal{F}_r\right\rangle,\]
and call it a \emph{fragmentation group} of $G$, where $S_{\mathrm{fin}}$ is as in Definition~\ref{def:prelim-finite-singular-germ-condition}.

\begin{prop}
  \label{prop:fragmentation-preserves-finite-singular-condition}
  Under the assumptions above, $\widehat{G}$ is a group of bounded type satisfying the finite singular germ condition with the same boundary points $\Xi=\{\xi_1,\ldots,\xi_r\}$.
\end{prop}

\begin{proof}
  The fragmentations introduce only finitely many labels and do not change the finite tiles or their boundary points, while every original boundary edge is retained with a fragment label. Hence the resulting tile inflation is still of bounded type and has the same boundary points of the infinite tiles.

  Let $f\in\mathcal{F}_i$. On every piece of $\mathcal{P}_i$, the transformation $f$ agrees with an element of $K_i$. If $\zeta\notin\operatorname{Fix}(K_i)$ and $\zeta\neq\xi_i$, the germ $(f,\zeta)$ therefore belongs to $\mathfrak{T}_{\mathsf{B}}$. If $\zeta\in\operatorname{Fix}(K_i)\setminus\{\xi_i\}$, every $(k,\zeta)$ with $k\in K_i$ is an isotropy element of the principal groupoid $\mathfrak{T}_{\mathsf{B}}$ and is therefore trivial. Since $K_i$ is finite, its elements fix a common neighborhood of $\zeta$, and every fragment fixes this neighborhood as well. Thus $f(\xi_i)=\xi_i$ and $(f,\zeta)\in\mathfrak{T}_{\mathsf{B}}$ for every $\zeta\neq\xi_i$. If $(f,\xi_i)$ is trivial, then every germ of $f$ belongs to $\mathfrak{T}_{\mathsf{B}}$. By compactness, there is a level $m$ such that $f\in\mathsf{S}_m$, and hence $f$ is finitary. Otherwise, $f$ is a singular generator assigned to $\xi_i$.

  Every $\widehat{G}$-orbit lies in a $G$-orbit, since each fragment agrees pointwise with elements of $K_i\leq G$ and $S_{\mathrm{fin}}\subseteq G$. Thus the $\xi_i$, which lie in distinct $G$-orbits, remain in distinct $\widehat{G}$-orbits. Since $K_i\to H_i$ and the restriction maps $\mathcal{F}_i\to K_i|_P$ are surjective and $\mathcal{P}_i$ is finite, each $\xi_i$ remains singular.

  We verify that each $\xi_i$ remains germ-defining. Consider a nonloop cycle based at a point in the $\widehat{G}$-orbit of $\xi_i$. If one of its factor germs is non-AF, then the corresponding intermediate point is one of the boundary points $\xi_j$. The factor is a singular fragment assigned to $\xi_j$ and fixes $\xi_j$, and thus two consecutive vertices of the cycle coincide. This contradicts the definition of a nonloop cycle. Thus every factor germ is AF. The germ of the cycle is then an isotropy element of the principal groupoid $\mathfrak{T}_{\mathsf{B}}$, and hence it is trivial.

  Finally, the representative argument of \cite[Proposition~3.7]{kua26no1}, applied to the generating set $S_{\mathrm{fin}}\cup\bigcup_i\mathcal{F}_i$, shows that every germ at $\xi_i$ is represented by the subgroup generated by the singular generators assigned to $\xi_i$. This subgroup is contained in the finite group $\mathcal{F}_i$. Hence the group of germs at $\xi_i$ is a quotient of a subgroup of $\mathcal{F}_i$ and is finite.
\end{proof}

In some examples, the same singular transformation acts at several families of connectors, and the gluing arguments prescribe different actions on these families. We then write it as a product of commuting transformations of finite order, each acting on one family of connectors, and fragment these transformations separately. This is the reason for the decomposition $L=L'L''$ in the modified LMS-group in \cite[Section~5]{kua26no1}.

\subsection{Selector-producing fragmentations}
\label{subsec:selector-producing-fragmentations}

Fix finite families of alternating data $\mathscr{D}_n$ and groups $\mathsf{C}_n\leq\mathsf{A}_n$ containing the associated alternating groups and the required auxiliary subgroups. Assume that the data at the preceding level, their descendant supports, the active and protected connectors, the prescribed permutations of connecting points, and the gluing steps required in Condition~\ref{cond:selector-system:A2} have been specified. It remains to realize the prescribed permutations of connecting points by elements of $\widehat{G}$. Consider one of the gluing arguments used to produce a target datum $\mathcal{D}'\in\mathscr{D}_{n+1}$. Let $\Sigma$ be the union of the descendant supports involved in the argument, let $\mathsf{a}\subseteq\mathsf{e}_n$ be the union of its active connectors, and let $\mathsf{p}$ be the union of its protected connectors.

Suppose that, for every connector $c$ in $\mathsf{a}$, the prescribed permutation of its connecting points is induced on $U_c$ by an element of $K_i$. We choose the partition $\mathcal{P}_i$ so
that every connector neighborhood used in the prescribed gluing arguments is contained in one piece of $\mathcal{P}_i$.

For each $P\in\mathcal{P}_i$ containing $U_c$ for an active connector $c$, prescribe $a_P\in K_i|_P$ with the required restriction on $U_c$. If $P$ contains $U_c$ for a protected connector, prescribe $a_P=\Id|_P$. If several active connector neighborhoods lie in the same piece $P$, a single $a_P$ must realize all their prescribed restrictions.

\begin{defn}[Selector-producing fragmentation]
  \label{def:selector-producing-fragmentation}
  The family $\{\mathcal{F}_1,\ldots,\mathcal{F}_r\}$ is \emph{selector-producing} for $\{(\mathscr{D}_n,\mathsf{C}_n)\}_{n\geq n_0}$ if the following conditions hold.

  \begin{enumerate}[label=\textup{(F\arabic*)},ref=\textup{(F\arabic*)}]
    \item\label{cond:selector-fragmentation:F1}
          Every connector neighborhood used in the prescribed gluing arguments is contained in one piece of $\mathcal{P}_i$. The required restrictions on each piece are compatible.

    \item\label{cond:selector-fragmentation:F2}
          For every gluing argument and every $i$, the prescribed coordinates extend to an element $f=(f_P)_{P\in\mathcal{P}_i}\in\mathcal{F}_i$. That is, $f_P=a_P$ on every piece on which a coordinate has been prescribed.

    \item\label{cond:selector-fragmentation:F3}
          If one gluing argument uses connectors associated with several groups $K_i$, then the corresponding connector neighborhoods are pairwise disjoint, and the product of the selected fragments realizes the combined active and protected prescriptions.
  \end{enumerate}
\end{defn}

\begin{prop}
  \label{prop:coordinate-selector-criterion}
  Fix one gluing argument. Let $\mathcal{F}_i\leq\prod_{P\in\mathcal{P}_i}K_i|_P$ be a fragmentation used in it. Let $J\subseteq\mathcal{P}_i$ be the set of pieces on which the argument prescribes a coordinate, and let $\rho_J:\mathcal{F}_i\longrightarrow\prod_{P\in J}K_i|_P$ be the coordinate projection. Write $a_J=(a_P)_{P\in J}$ for the tuple of restrictions prescribed on the active and protected pieces. Then the prescription extends to an element of $\mathcal{F}_i$ if and only if $a_J\in\rho_J(\mathcal{F}_i)$.
  Consequently, finitely many membership tests determine whether a fixed finite fragmentation realizes every tuple in a finite list of prescribed restrictions.
\end{prop}

\begin{proof}
  An element $f\in\mathcal{F}_i$ realizes the prescribed coordinates exactly when $\rho_J(f)=a_J$. This is equivalent to $a_J$ belonging to the image of $\rho_J$. The last assertion follows since $\mathcal{F}_i$ is finite.
\end{proof}

\begin{defn}[Elementary abelian fragmentation]
  \label{def:elementary-abelian-fragmentation}
  Let $\mathcal{F}$ be a fragmentation of $K$ with respect to $\mathcal{P}$. We call $\mathcal{F}$ an \emph{elementary abelian fragmentation} if there is a prime $p$ such that every group $K|_P$, $P\in\mathcal{P}$, is cyclic of order $p$. In this case $\mathcal{F}$ is an elementary abelian $p$-group, and we say that $\mathcal{F}$ has exponent $p$.
\end{defn}

Let $\mathcal{F}$ be an elementary abelian fragmentation of exponent $p$.

\begin{defn}[Activity vector]
  \label{def:activity-vector}
  For every $P\in\mathcal{P}$, choose a generator $k_P$ of $K|_P$. For $f\in\mathcal{F}$, write $f|_P=k_P^{c_P(f)}$, where $c_P(f)\in\mathbb{F}_p$. The vector $(c_P(f))_{P\in\mathcal{P}}$ is the \emph{activity vector} of $f$. The image $C\leq\mathbb{F}_p^{\mathcal{P}}$ of this injective homomorphism is the \emph{activity subspace}.
\end{defn}

\begin{cor}
  \label{cor:linear-selector-criterion}
  Suppose that $\mathcal{F}_i$ is an elementary abelian fragmentation of exponent $p$, and let $C\leq\mathbb{F}_p^{\mathcal{P}_i}$ be its activity subspace. Let $A$ be a matrix whose columns form a basis of $C$. For a set $J$ of pieces of the partition and prescribed exponents $b\in\mathbb{F}_p^J$, a fragment with these coordinates exists if and only if the linear system $A_Jx=b$ is solvable, where $A_J$ is obtained by retaining the rows indexed by $J$. Equivalently, $\operatorname{rank}(A_J)=\operatorname{rank}([A_J\mid b])$.
\end{cor}

\begin{proof}
  This is Proposition~\ref{prop:coordinate-selector-criterion} after writing the coordinate projection of the activity subspace in linear coordinates.
\end{proof}

After numbering the $m=|\mathcal{P}|$ pieces, we regard the activity subspace as a linear code in $\mathbb{F}_p^m$. Its dual code is $C^\perp=\{a\in\mathbb{F}_p^m:\sum_{j=1}^m a_jc_j=0\text{ for every }c\in C\}$. For a nonzero code $D$, let $d(D)$ be the least number of nonzero coordinates of a nonzero word in $D$, and set $d(\{0\})=+\infty$. The following is the standard dual-distance test for coordinate projections. See, for example, \cite{MacWilliamsSloane77}.

\begin{prop}
  \label{prop:dual-distance-selector-criterion}
  Let $C\leq\mathbb{F}_p^m$ be the activity subspace of an elementary abelian fragmentation, and let $C^\perp$ be its dual code. For $t\geq1$, the following are equivalent.
  \begin{enumerate}
    \item For every $J\subseteq\{1,\ldots,m\}$ with $|J|\leq t$, the coordinate projection $C\to\mathbb{F}_p^J$ is surjective.
    \item Every prescription on at most $t$ pieces of $\mathcal{P}$ is realized by a fragment.
    \item The minimum Hamming weight of a nonzero word of $C^\perp$ is greater than $t$.
  \end{enumerate}
  Thus, if $d(C^\perp)>t$, every prescribed tuple on at most $t$ pieces of the partition is realized by a fragment. The fixed points required by a particular conjugation must then be checked separately.
\end{prop}

\begin{proof}
  The equivalence of the first two statements follows from Proposition~\ref{prop:coordinate-selector-criterion}. Fix $J$. The projection of $C$ to $\mathbb{F}_p^J$ is not surjective if and only if there is a nonzero linear functional on $\mathbb{F}_p^J$ which vanishes on its image. Extending this functional by zero outside $J$ gives a nonzero word of $C^\perp$ supported in $J$. Conversely, every nonzero dual word supported in $J$ gives such a functional. Hence all projections on at most $t$ coordinates are surjective exactly when $C^\perp$ has no nonzero word of weight at most $t$.
\end{proof}

The following elementary observation expresses the identity-coordinate property used in \cite[Proposition~4.5.1]{JC19} in terms of the activity subspace.

\begin{prop}
  \label{prop:activity-identity-coordinate}
  In the same elementary abelian setting, every nonidentity fragment fixes at least one piece of the partition pointwise if and only if the activity subspace contains no nonzero word of full support.
\end{prop}

\begin{proof}
  A coordinate of the activity vector is zero exactly when the corresponding fragment acts identically on that piece of the partition.
\end{proof}

Proposition~\ref{prop:coordinate-selector-criterion} tests finite lists of prescriptions, whereas Proposition~\ref{prop:dual-distance-selector-criterion} treats all prescriptions on at most $t$ pieces. The identity-coordinate condition of Proposition~\ref{prop:activity-identity-coordinate}, together with the accumulation of each piece at the corresponding singular point, gives purely non-Hausdorff singularities.

\begin{prop}
  \label{prop:fragmentations-produce-selectors}
  Suppose that the fragmentations are selector-producing. Then every prescribed permutation of the active connecting points and every prescribed identity restriction on protected connector neighborhoods are realized by a selector in $\widehat{G}$.
\end{prop}

\begin{proof}
  Fix a gluing argument.  For every singular point used in it, Condition~\ref{cond:selector-fragmentation:F2} gives an element of $\mathcal{F}_i$ with the prescribed restrictions on pieces containing active connector neighborhoods and the identity restriction on the other prescribed pieces.

  If the argument uses one singular point, we take this fragment as the selector. If it uses several singular points, we take the product of the corresponding fragments. Condition~\ref{cond:selector-fragmentation:F3} shows that the product has the prescribed action on $U_{\mathsf{a}}$ and acts identically on $U_{\mathsf{p}}$. It is therefore a selector in the sense of Definition~\ref{def:selector}.
\end{proof}

We recall the following definition from \cite[Definition~2.2]{nekrash18}.

\begin{defn}[Purely non-Hausdorff singularity]
  \label{def:purely-non-Hausdorff}
  A germ-defining singular point $\xi$ of an action of $G$ is \emph{purely non-Hausdorff} if $\operatorname{Int}(\operatorname{Fix}(g))$ accumulates on $\xi$ for every $g\in G_\xi$.
\end{defn}

Assume, in addition, that for every $i$:

\begin{enumerate}[label=\textup{(N\arabic*)},ref=\textup{(N\arabic*)}]
  \item\label{cond:pure-non-Hausdorff:N1}
        every piece $P\in\mathcal{P}_i$ accumulates on $\xi_i$;

  \item\label{cond:pure-non-Hausdorff:N2}
        for every $f\in\mathcal{F}_i$, there is $P\in\mathcal{P}_i$ such that $f|_P=\Id|_P$.
\end{enumerate}

For an elementary abelian fragmentation, Proposition~\ref{prop:activity-identity-coordinate} identifies Condition~\ref{cond:pure-non-Hausdorff:N2} with the absence of a nonzero full-support word in the activity subspace.

\begin{prop}
  \label{prop:identity-coordinate-pure-non-Hausdorff}
  Under Conditions~\ref{cond:pure-non-Hausdorff:N1} and~\ref{cond:pure-non-Hausdorff:N2}, every
  $\xi_i\in\Xi$ is a purely non-Hausdorff singularity of $\widehat{G}$.
\end{prop}

\begin{proof}
  Let $g\in\widehat{G}_{\xi_i}$. If $(g,\xi_i)$ is trivial, then $g$ fixes a neighborhood of $\xi_i$, and the conclusion follows.

  Suppose that $(g,\xi_i)$ is nontrivial. By \cite[Proposition~3.7]{kua26no1}, we may choose $f\in\mathcal{F}_i$ with $(f,\xi_i)=(g,\xi_i)$. Thus $g$ and $f$ agree on a neighborhood $V$ of $\xi_i$.

  By Condition~\ref{cond:pure-non-Hausdorff:N2}, there is $P\in\mathcal{P}_i$ such that $f|_P=\Id|_P$. By Condition~\ref{cond:pure-non-Hausdorff:N1}, the set $P$ accumulates on $\xi_i$. Hence $P\cap V$ is an open subset of $\operatorname{Fix}(g)$ which accumulates on $\xi_i$. Therefore $\operatorname{Int}(\operatorname{Fix}(g))$ accumulates on $\xi_i$.
\end{proof}

The preceding results determine whether the chosen fragmentations realize the prescribed finite restrictions. They do not by themselves imply Condition~\ref{cond:selector-system:A2}: any additional fixed points required by the conjugations and the resulting inclusion for $\mathsf{C}_{n+1}$ must still be verified in the application. In Section~\ref{sec:fragment-FG-modified}, this is done explicitly for the modified Fabrykowski--Gupta group.

The next example uses the three-piece fragmentation from \cite[p.~683]{nekrash18}.

\begin{exmp}
  \label{exmp:three-piece-selector-fragmentation}
  Let $h$ be an involution fixing a germ-defining singular point $\xi$, and let $\mathcal{P}=\{P_0,P_1,P_2\}$, where every $P_j$ is $h$-invariant and accumulates on $\xi$.  Define $h_j$ to act identically on $P_j$ and
  as $h$ on the other two pieces. Then
  \[\mathcal{F}=
    \langle h_0,h_1,h_2\rangle
    \cong
    (\mathbb{Z}/2\mathbb{Z})^2,
    \qquad
    h_0h_1h_2=\Id.\]
  The coordinate vectors are
  \[
    (0,0,0),\qquad
    (0,1,1),\qquad
    (1,0,1),\qquad
    (1,1,0).
  \]

  Every element fixes one of the pieces pointwise, and thus $\xi$ is purely non-Hausdorff. The element $h_j$ realizes the prescribed restrictions when these require $h$ on the two pieces different from $P_j$ and the identity on $P_j$. It is a selector whenever these are the prescribed active and protected restrictions.

  This fragmentation does not realize every tuple of prescribed restrictions. For example, it contains no element acting as $h$ on exactly one piece and identically on the other two. Thus Condition~\ref{cond:selector-fragmentation:F2} must still be verified for the chosen fragmentation and the restrictions used in the propagation argument.
\end{exmp}

\begin{rmk}
  Proposition~\ref{prop:coordinate-selector-criterion} tests membership for the chosen fragmentation, not whether another element of $\widehat{G}$ has the prescribed restrictions. A realizing fragment is a selector when Definition~\ref{def:selector} holds. Purely non-Hausdorffness instead follows from Conditions~\ref{cond:pure-non-Hausdorff:N1} and~\ref{cond:pure-non-Hausdorff:N2} by Proposition~\ref{prop:identity-coordinate-pure-non-Hausdorff}.
\end{rmk}
\section[A fragmentation group of the modified Fabrykowski--Gupta group]
 {A fragmentation group of the modified\\ Fabrykowski--Gupta group}
\label{sec:fragment-FG-modified}
We use the four-periodic fragmentation of the Fabrykowski--Gupta group constructed explicitly in \cite[Subsection~5.4]{kua26}. That subsection specifies the four pieces, the two fragments, and their wreath recursions. Cantu established the existence of suitable periodic fragmentations of the Fabrykowski--Gupta group \cite[Example~5.1.8]{JC19}. For $p=3$, the abstract activity vectors in \cite[Proposition~4.5.1]{JC19} coincide with those of the explicit generators in \cite[Subsection~5.4]{kua26}. We modify the action following \cite{nek25} and \cite[Subsection~5.3.5]{nek22}. For the resulting fragmentation group, we prove containment of the AF alternating group and compute its exact index in the topological full group.

Let $\mathsf{X}=\{0,1,2\}$, and let $G_0$ be the Fabrykowski--Gupta group constructed by J.~Fabrykowski and N.~Gupta in \cite{FG85}. Using wreath recursion notation \cite[Definition~1.5.4]{nekrash05}, its generators are given by
$a=(0\,1\,2)$ and $b=(a,\Id,b)$.

We fragment the directed generator $b$. For $i=1,2,3,4$, put
$\mathcal{F}^{(i)}=\langle b_1^{(i)},b_2^{(i)}\rangle$, where
\begin{align}\label{gen1}
  \begin{cases}
    b_1^{(1)}=(\Id,\Id,b_1^{(2)}), \\
    b_1^{(2)}=(a,\Id,b_1^{(3)}),   \\
    b_1^{(3)}=(a^2,\Id,b_1^{(4)}), \\
    b_1^{(4)}=(a,\Id,b_1^{(1)}),
  \end{cases}
\end{align}
and
\begin{align}\label{gen2}
  \begin{cases}
    b_2^{(1)}=(a,\Id,b_2^{(2)}), \\
    b_2^{(2)}=(a,\Id,b_2^{(3)}), \\
    b_2^{(3)}=(a,\Id,b_2^{(4)}), \\
    b_2^{(4)}=(\Id,\Id,b_2^{(1)}).
  \end{cases}
\end{align}
We write $b_1=b_1^{(1)}$ and $b_2=b_2^{(1)}$.

Consider the one-sided Markov shift $\mathcal{M}$ over $\widehat{\mathsf{X}}=\{0_0,0_1,1,2\}$ with allowed transitions
\[
  \begin{gathered}
    0_0\,0_1,\quad
    0_1\,0_0,\quad 0_1\,1,\quad 0_1\,2,\\
    1\,0_0,\quad 1\,0_1,\quad 1\,1,\quad 1\,2,\\
    2\,0_0,\quad 2\,0_1,\quad 2\,1,\quad 2\,2.
  \end{gathered}
\]
We realize $\mathcal{M}$ as $\Omega(\mathsf{B})$ by taking one initial vertex, one edge from it to each symbol of $\widehat{\mathsf{X}}$, and thereafter the stationary incidence pattern determined by the allowed transitions. Thus a length-$n$ path is encoded by an admissible word of $n$ symbols. The stationary part of $\mathsf{B}$ is shown in Figure~\ref{fig:FG-modified-bratteli}.

\begin{figure}[ht]
  \centering
  \begin{tikzpicture}[
      x=2.4cm,
      y=2.2cm,
      vertex/.style={circle,fill=black,inner sep=1.8pt},
      lab/.style={fill=none,draw=none},
      edge/.style={line width=0.55pt}
    ]

    \node[vertex] (t00) at (0,1) {};
    \node[vertex] (t01) at (1,1) {};
    \node[vertex] (t1)  at (2,1) {};
    \node[vertex] (t2)  at (3,1) {};

    \node[vertex] (b00) at (0,0) {};
    \node[vertex] (b01) at (1,0) {};
    \node[vertex] (b1)  at (2,0) {};
    \node[vertex] (b2)  at (3,0) {};

    \node[lab,above=6pt] at (t00) {$0_0$};
    \node[lab,above=6pt] at (t01) {$0_1$};
    \node[lab,above=6pt] at (t1)  {$1$};
    \node[lab,above=6pt] at (t2)  {$2$};

    \node[lab,below=6pt] at (b00) {$0_0$};
    \node[lab,below=6pt] at (b01) {$0_1$};
    \node[lab,below=6pt] at (b1)  {$1$};
    \node[lab,below=6pt] at (b2)  {$2$};

    \draw[edge] (t00) .. controls +(0.35,-0.35) and +(-0.35,0.35) .. (b01);

    \draw[edge] (t01) -- (b00);
    \draw[edge] (t01) -- (b1);
    \draw[edge] (t01) -- (b2);

    \draw[edge] (t1) -- (b00);
    \draw[edge] (t1) -- (b01);
    \draw[edge] (t1) -- (b1);
    \draw[edge] (t1) -- (b2);

    \draw[edge] (t2) -- (b00);
    \draw[edge] (t2) -- (b01);
    \draw[edge] (t2) -- (b1);
    \draw[edge] (t2) -- (b2);

  \end{tikzpicture}
  \caption{One level of the Bratteli diagram $\mathsf{B}$ modeling $\mathcal{M}$.}
  \label{fig:FG-modified-bratteli}
\end{figure}
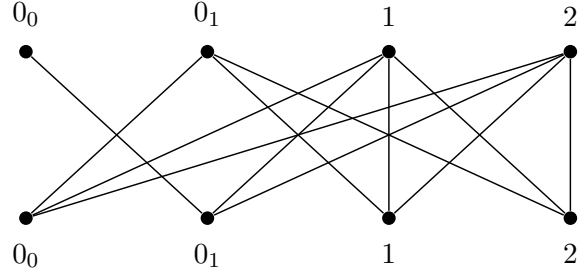

Let $\lambda:\mathcal{M}\to\mathsf{X}^{\omega}$ erase the lower indices. We use the same symbol for the coordinatewise erasure map on finite admissible words, and for a finite path $\gamma$ we denote its cylinder in $\mathcal M$ by $[\gamma]$. As in \cite[Subsection~5.3.5]{nek22}, the fiber $\lambda^{-1}(w)$ is a singleton unless $w$ is eventually $0$, in which case it has two points. For example, $\lambda^{-1}(0^{\omega})=\{(0_0\,0_1)^{\omega},(0_1\,0_0)^{\omega}\}$. For $j\in\{0,1,2,3\}$, put
\[P_j:=\lambda^{-1}\left(\bigsqcup_{\substack{k\geq0\\k\equiv j\pmod4}}[2^k0]\right),\]
where $[2^k0]$ is a cylinder in $\mathsf{X}^{\omega}$. These are the four $\langle b\rangle$-invariant open pieces of $\mathcal{M}\setminus\operatorname{Fix}(\langle b\rangle)$ used in the fragmentation, and every $P_j$ accumulates on $\xi=2^\omega$. For $k\equiv j\pmod4$, the section of $b_i$ at $2^k$ is $b_i^{(j+1)}$, with the superscript read modulo $4$. Its $0$-section determines the activity on $P_j$.

The action of $G_0$ lifts through $\lambda$. We write the lifted action of $a$ as a product of four disjoint $3$-cycles according
to the second symbol, as shown in Figure~\ref{fig:four-triangular-connectors}. Together with the lifted directed
generator $b$, these transformations generate the \emph{modified Fabrykowski--Gupta group} $G_{\mathrm{FG}}:=\langle
  a_{0_0},a_{0_1},a_1,a_2,b\rangle$. Replacing $b$ by its fragments $b_1,b_2$, we obtain the group
\[
  \widehat{G}_{\mathrm{FG}}
  :=
  \langle a_{0_0},a_{0_1},a_1,a_2,b_1,b_2\rangle,
\]
which we call a fragmentation group of the modified Fabrykowski--Gupta group. We put
$\mathfrak{G}_{\mathrm{FG}}:=\operatorname{Germ}(\widehat{G}_{\mathrm{FG}}\curvearrowright\Omega(\mathsf{B}))$.

\begin{figure}[ht]
  \centering
  \begin{tikzpicture}[
      >=Stealth,
      vertex/.style={circle,fill=black,inner sep=1.6pt},
      vlabel/.style={font=\small},
      elabel/.style={font=\small, fill=white, inner sep=1pt},
      tcaption/.style={font=\small}
    ]

    \newcommand{\drawconnector}[7]{%
      \begin{scope}[shift={(#1,#2)}]
        \coordinate (L) at (0,0);
        \coordinate (R) at (3.2,0);
        \coordinate (T) at (1.6,2.25);

        \node[vertex] at (L) {};
        \node[vertex] at (R) {};
        \node[vertex] at (T) {};

        \node[vlabel,below left=3pt]  at (L) {$#3$};
        \node[vlabel,below right=3pt] at (R) {$#4$};
        \node[vlabel,above=4pt]       at (T) {$#5$};

        \draw[->,line width=0.8pt] (T) -- node[elabel,left=2pt]  {$#6$} (L);
        \draw[->,line width=0.8pt] (L) -- node[elabel,below=2pt] {$#6$} (R);
        \draw[->,line width=0.8pt] (R) -- node[elabel,right=2pt] {$#6$} (T);

        \node[tcaption] at (1.6,-0.75) {$#7$};
      \end{scope}
    }

    \drawconnector{0}{0}
    {0_1\,0_0w}
    {1\,0_0w}
    {2\,0_0w}
    {a_{0_0}}
    {\mathcal{T}_{0_0,2}}

    \drawconnector{6.4}{0}
    {0_0\,0_1w}
    {1\,0_1w}
    {2\,0_1w}
    {a_{0_1}}
    {\mathcal{T}_{0_1,2}}

    \drawconnector{0}{-4.8}
    {0_1\,1w}
    {1\,1w}
    {2\,1w}
    {a_{1}}
    {\mathcal{T}_{1,2}}

    \drawconnector{6.4}{-4.8}
    {0_1\,2w}
    {1\,2w}
    {2\,2w}
    {a_{2}}
    {\mathcal{T}_{2,2}}

  \end{tikzpicture}
  \caption{The four lifted $3$-cycles $a_{0_0},a_{0_1},a_1,a_2$ and the corresponding level-$2$ tiles.}
  \label{fig:four-triangular-connectors}
\end{figure}
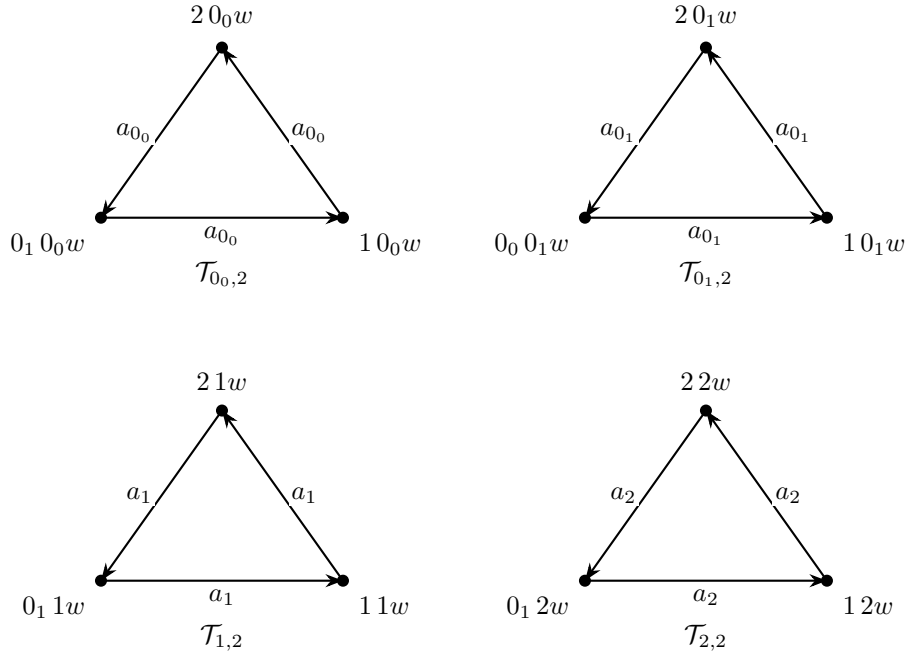

The transformations $a_{0_0},a_{0_1},a_1,a_2$ are extended identically to the rest of $\mathcal{M}$. There are four
tiles on every level, denoted $\mathcal{T}_{0_0,n}$, $\mathcal{T}_{0_1,n}$, $\mathcal{T}_{1,n}$ and $\mathcal{T}_{2,n}$. The level-$2$ tiles are shown in Figure~\ref{fig:four-triangular-connectors} after omitting the
suffix $w$. For $n\geq2$, the inflation rule is shown in Figure~\ref{fig:FG-inflation-rule}.

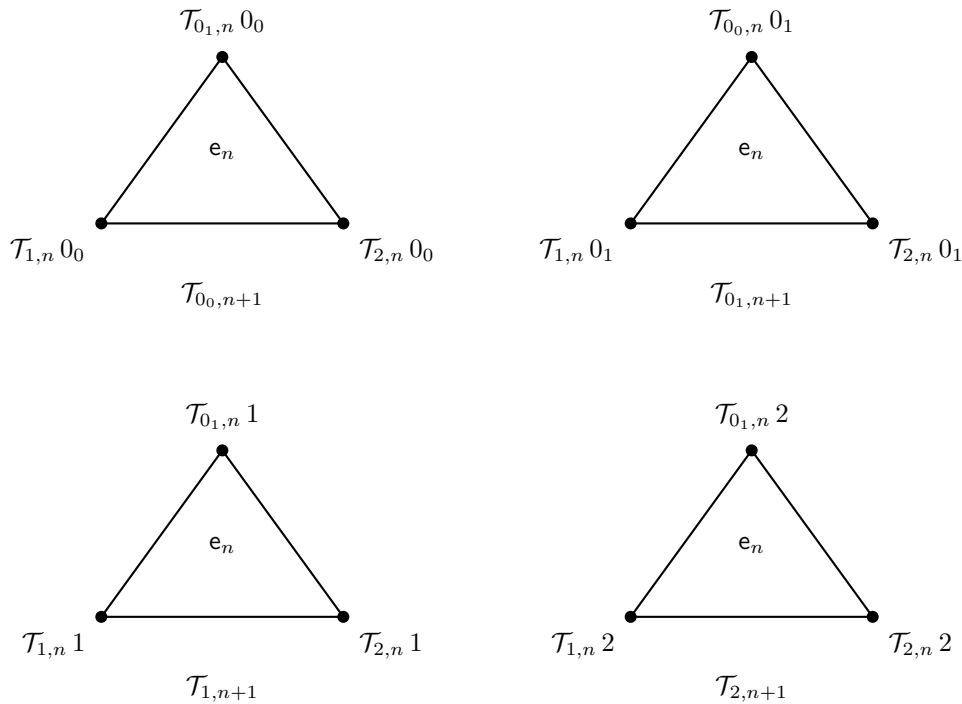
\begin{figure}[ht]
  \centering
  \begin{tikzpicture}[
      vertex/.style={circle,fill=black,inner sep=1.6pt},
      vlabel/.style={font=\small},
      tcaption/.style={font=\small},
      innerlab/.style={font=\small}
    ]

    \newcommand{\drawinflation}[6]{%
      \begin{scope}[shift={(#1,#2)}]
        \coordinate (L) at (0,0);
        \coordinate (R) at (3.2,0);
        \coordinate (T) at (1.6,2.2);

        \draw[line width=0.8pt] (L) -- (R) -- (T) -- cycle;

        \node[vertex] at (L) {};
        \node[vertex] at (R) {};
        \node[vertex] at (T) {};

        \node[vlabel,below left=3pt]  at (L) {$#3$};
        \node[vlabel,below right=3pt] at (R) {$#4$};
        \node[vlabel,above=4pt]       at (T) {$#5$};

        \node[innerlab] at (1.6,0.95) {$\mathsf{e}_n$};

        \node[tcaption] at (1.6,-0.95) {$#6$};
      \end{scope}
    }

    \drawinflation{0}{0}
    {\mathcal{T}_{1,n}\,0_0}
    {\mathcal{T}_{2,n}\,0_0}
    {\mathcal{T}_{0_1,n}\,0_0}
    {\mathcal{T}_{0_0,n+1}}

    \drawinflation{7.0}{0}
    {\mathcal{T}_{1,n}\,0_1}
    {\mathcal{T}_{2,n}\,0_1}
    {\mathcal{T}_{0_0,n}\,0_1}
    {\mathcal{T}_{0_1,n+1}}

    \drawinflation{0}{-5.2}
    {\mathcal{T}_{1,n}\,1}
    {\mathcal{T}_{2,n}\,1}
    {\mathcal{T}_{0_1,n}\,1}
    {\mathcal{T}_{1,n+1}}

    \drawinflation{7.0}{-5.2}
    {\mathcal{T}_{1,n}\,2}
    {\mathcal{T}_{2,n}\,2}
    {\mathcal{T}_{0_1,n}\,2}
    {\mathcal{T}_{2,n+1}}

  \end{tikzpicture}
  \caption{Inflation rule for the tiles of the modified Fabrykowski--Gupta group. The edges are unoriented. Their orientations are specified by the connector graph $\mathsf{e}_n$.}
  \label{fig:FG-inflation-rule}
\end{figure}

The connector graphs $\mathsf{e}_n$, with the convention following Definition~\ref{def:prelim-connector}, are shown in Figure~\ref{fig:FG-periodic-connectors}. The incidence matrix $A_{\mathrm{FG}}$ in Subsection~\ref{subsec:FG-AF-parity} satisfies $A_{\mathrm{FG}}^3>0$, and thus $\mathsf{B}$ is simple. The four tile types in Figure~\ref{fig:FG-inflation-rule} and the period-$4$ connector types in Figure~\ref{fig:FG-periodic-connectors} give uniform bounds on the numbers of tile types, copies in each inflation, and boundary points on each finite tile. Hence $G_{\mathrm{FG}}$ and $\widehat{G}_{\mathrm{FG}}$ are groups of bounded type. The recursions~\eqref{gen1} and~\eqref{gen2} only replace labels on the existing boundary and connector edges, and thus satisfy the assumptions on tile inflation preceding Proposition~\ref{prop:fragmentation-preserves-finite-singular-condition}.

\begin{rmk}
  Since each edge $e$ on $\mathsf{B}$ in Figure~\ref{fig:FG-modified-bratteli} is uniquely determined by $\mathbf{s}(e)$ and $\mathbf{r}(e)$, each path $\gamma$ is represented by a sequence of vertices in $\widehat{\mathsf{X}}$. Whenever we continue a finite path $\gamma\in\Omega_n(\mathsf{B})$, instead of writing $\gamma e_x$, we will write $\gamma x$ for $x\in\widehat{\mathsf{X}}$. When adjacent symbols carry lower indices, we insert a thin space to separate them, as in $q_0\,0_1$ and $\alpha_0\,0_1x$.
\end{rmk}

\subsection[Containment of the AF alternating group]
{Containment of $\mathsf{A}(\mathfrak{T}_{\mathsf{B}})$}
\label{FGcontain}

We first choose the fragments that realize the required permutations of connecting points.

\begin{lemma}\label{selectorexist}
  For every $n\geq2$, there is a selector $s_n\in\langle b_1,b_2\rangle$ which acts nontrivially on the connecting points of $\mathsf{e}_n$ and identically on $U_{\mathsf{e}_{n+1}}$. We may choose
  \[
    \begin{array}{c|c|c}
      n\pmod4 & \text{connector type} & s_n      \\ \hline
      2       & \textup{(a)}          & b_1^2b_2 \\
      3       & \textup{(b)}          & b_1b_2   \\
      0       & \textup{(c)}          & b_2      \\
      1       & \textup{(d)}          & b_1.
    \end{array}
  \]
\end{lemma}

\begin{proof}
  With the four coordinates indexed by the pieces $P_0,P_1,P_2,P_3$, or equivalently by the connector types \textup{(a)}, \textup{(b)}, \textup{(c)}, \textup{(d)}, the activity vector records the exponent of $a$ induced on each piece. The recursions~\eqref{gen1} and~\eqref{gen2} give
  \[
  b_1\longmapsto(0,1,2,1),
  \qquad
  b_2\longmapsto(1,1,1,0).
  \]
  On the piece $P_j$, the restriction of a fragment is determined by the corresponding exponent of the lifted rooted cycle $a$: if the projected word has the form $2^k0xw$ with $k\equiv j\pmod4$, then the fragment changes only the letter $x$ by the indicated power of $a$ and leaves the continuation $w$ unchanged. Thus a fragment with zero activity on all four pieces is the identity on $\mathcal{M}\setminus\operatorname{Fix}(\langle b\rangle)$, and every fragment fixes $\operatorname{Fix}(\langle b\rangle)$. Hence the map $\mathcal{F}^{(1)}\to\mathbb{F}_3^4$ sending a fragment to its activity vector is injective. The two activity vectors are linearly independent, and every coordinate projection is nontrivial. Hence
  \[
  \mathcal{F}^{(1)}=\langle b_1,b_2\rangle\cong(\mathbb{Z}/3\mathbb{Z})^2,
  \]
  and it is a fragmentation of $\langle b\rangle$.

  Let $C\leq\mathbb{F}_3^4$ be its activity subspace. The two generating vectors are orthogonal to one another and to themselves. Since $\dim C=2$, we have $C=C^\perp$. Moreover, every pair of coordinate columns in the matrix
  \[
  \begin{pmatrix}
  0&1&2&1\\
  1&1&1&0
  \end{pmatrix}
  \]
  is linearly independent. The corresponding four-periodic fragments were constructed in \cite[Subsection~5.4]{kua26}. At the level of abstract activity vectors, their generators coincide with the case $p=3$ of \cite[Proposition~4.5.1]{JC19}. The calculation above identifies $C$ with the ternary self-dual $[4,2,3]$ code given in \cite[Chapter~1, Code~6, p.~7]{MacWilliamsSloane77}, also called the tetracode. In particular, $d(C^\perp)=3$, and Proposition~\ref{prop:dual-distance-selector-criterion} shows that arbitrary restrictions on any two of $P_0,P_1,P_2,P_3$ are realizable.

  The gluing argument prescribes a nonzero coordinate on the type of the current connector and the zero coordinate on the type of the next connector. The four elements displayed in the statement are explicit choices. Their activity vectors are
  \[
  (1,0,2,2),\qquad(1,2,0,1),\qquad(1,1,1,0),\qquad(0,1,2,1),
  \]
  respectively. Hence each $s_n$ acts nontrivially on the connecting points of $\mathsf{e}_n$ and identically on $U_{\mathsf{e}_{n+1}}$. It induces one of the two cyclic permutations of the three connecting points on every active copy of $\mathsf{e}_n$, while $s_n^{-1}$ induces the other. Thus $s_n$ and $s_n^{-1}$ satisfy the two conditions in Definition~\ref{def:selector}.
\end{proof}

For the remainder of Subsection~\ref{FGcontain}, put $I=\{0_0,1,2\}$. For $n\geq2$, let
$Y_n:=\mathcal{T}_{0_1,n-1}\sqcup\mathcal{T}_{1,n-1}\sqcup\mathcal{T}_{2,n-1}$ and, for $x\in I$, define
$\phi_{n,x}:Y_n\to\Omega_n(\mathsf{B})$ by $\phi_{n,x}(\gamma)=\gamma x$. For $n\geq3$, put
$\mathcal{D}_n:=(Y_n;\phi_{n,0_0},\phi_{n,1},\phi_{n,2})$ and
\begin{equation}
  \mathsf{D}_n:=\mathsf{A}(\mathcal{D}_n)=\Delta\bigl(\mathsf{A}(\mathcal{T}_{0_0,n}),\mathsf{A}(\mathcal{T}_{1,n}),\mathsf{A}(\mathcal{T}_{2,n})\bigr).
  \label{diagonn}
\end{equation}
For $x\in I$, define $\psi_{n,x}:Y_{n-1}\to\Omega_n(\mathsf{B})$ by
$\psi_{n,x}(\gamma)=\phi_{n-1,x}(\gamma)0_1$. Then
$\mathcal{E}_n:=(Y_{n-1};\psi_{n,0_0},\psi_{n,1},\psi_{n,2})$ is a synchronized alternating datum whose coordinate supports lie in $\mathcal{T}_{0_1,n}$, and we put
\begin{equation}
  \mathsf{E}_n:=\mathsf{A}(\mathcal{E}_n)=\Delta\bigl(\mathsf{A}(\mathcal{T}_{0_0,n-1}0_1),\mathsf{A}(\mathcal{T}_{1,n-1}0_1),\mathsf{A}(\mathcal{T}_{2,n-1}0_1)\bigr).
  \label{diagonnmin1}
\end{equation}

The coordinate calculation must be supplemented by a check on the finite supports used in the conjugations. Write $\bar x\in\{0,1,2\}$ for the image of $x\in\widehat{\mathsf X}$ after erasing a lower index.

\begin{lemma}
  \label{lem:FG-selector-phase-support}
  Put $r\equiv n-2\pmod4$. The connecting points of every copy of $\mathsf e_n$ used below have projections
  \[
    \lambda(p_{j,x})=2^{n-2}0j\bar x,
    \qquad j\in\{0,1,2\}.
  \]
  Hence their cylinders are contained in $P_r$, and $s_n$ induces the prescribed cyclic permutation on them. The four phases are
  \[
  \begin{array}{c|c|c|c|c}
  n\pmod4 & r & U_{\mathsf e_n}\subseteq & U_{\mathsf e_{n+1}}\subseteq & s_n \\ \hline
  2 & 0 & P_0 & P_1 & b_1^2b_2 \\
  3 & 1 & P_1 & P_2 & b_1b_2 \\
  0 & 2 & P_2 & P_3 & b_2 \\
  1 & 3 & P_3 & P_0 & b_1
  \end{array}
  \]
  The points used in the remaining factors of the conjugations may be chosen so that they are fixed by $s_n$. More precisely:
  \begin{enumerate}
    \item at the level-$3$ construction, the two auxiliary points in $\mathcal T_{0_0,2}0_1$ have projections $100$ and $200$. The same fixed-point verification applies to the auxiliary points $\alpha_x,\beta_x\in\mathcal T_{0_1,2}x$ used below, whose projections are $10\bar x$ and $20\bar x$.
    \item in the first induction step $n=3$, the two auxiliary points used to construct $\mathsf D_4$ have projections $110$ and $120$.
    \item for $n\geq4$, write $p_0=q_0\,0_1$. The corresponding points in the three coordinates of $\mathsf E_n$ satisfy
          \[
          \lambda(q_0)=2^{n-2}0,\qquad
          \lambda(q_1)=2^{n-2}1,\qquad
          \lambda(q_2)=2^{n-1}.
          \]
          After appending $0_1x$, the descendants of $q_1$ have projections $2^{n-2}10\bar x$ and the descendants of $q_2$ have projections $2^{n-1}0\bar x$.
    \item for $n\geq4$, the other two points in each cycle used to construct $\mathsf D_{n+1}$ may be chosen from the corresponding model paths with projections $1^{n-2}$ and $1^{n-3}2$. After applying the descendant embeddings and appending $0_1x$, all their projections begin with $1$.
    \item the auxiliary points used to construct $\mathsf M_{n+1}$ may be chosen from the model paths with projections $1^{n-1}$ and $1^{n-2}2$, and thus their projections also begin with $1$.
  \end{enumerate}
  In particular, every auxiliary point occurring in the conjugations is fixed by $s_n$, and $s_n$ is the identity on the connector neighborhoods of $\mathsf e_{n+1}$.
\end{lemma}

\begin{proof}
  A fragment acts nontrivially on a projected word only after the first prefix $2^k0$. If $k\equiv j\pmod4$, its $P_j$-coordinate applies the indicated power of the rooted cycle $a$ to the following letter and leaves the remaining tail unchanged. A word whose first non-$2$ letter is $1$ is fixed by every fragment.

  At level $2$, Figure~\ref{fig:four-triangular-connectors} gives the current projected prefix $0$. The recursions~\eqref{gen1} and~\eqref{gen2} continue the unique non-AF branch through the letter $2$. Induction on $n$ therefore gives the projected prefix $2^{n-2}0$ at the $n$th connector. The three connecting points are obtained by appending $0,1,2$, and a further terminal symbol $\bar x$ records the descendant tile. Hence
  $\lambda(p_{j,x})=2^{n-2}0j\bar x$. Thus the cylinders of all three connecting points are contained in $P_r$, where $r\equiv n-2\pmod4$. The table in Lemma~\ref{selectorexist} gives a nonzero $P_r$-coordinate and a zero $P_{r+1}$-coordinate for $s_n$.

  For the level-$3$ construction, the three vertices of $\mathcal T_{0_0,2}0_1$ have projections $000$, $100$, and $200$. The first is the connecting point $p_0$. We take the other two as the points denoted $x,y$ in the proof below. The point with projection $100$ is fixed since its projection begins with $1$, while the cylinder of the point with projection $200$ is contained in $P_1$ and the selector $s_2=b_1^2b_2$ has zero $P_1$-coordinate. In the first induction step, $\mathcal T_{0_1,3}$ contains the admissible paths with projections $110$ and $120$. Both projections begin with $1$, and thus $s_3$ fixes their cylinders pointwise.

  Now suppose $n\geq4$. In the notation used below, $p_0=q_0\,0_1$. The three points corresponding under the diagonal group $\mathsf E_n$ satisfy
  \[
    \lambda(q_0)=2^{n-2}0,
    \qquad
    \lambda(q_1)=2^{n-2}1,
    \qquad
    \lambda(q_2)=2^{n-1}.
  \]
  After appending $0_1x$, their projections are $2^{n-2}00\bar x$, $2^{n-2}10\bar x$, and $2^{n-1}0\bar x$, respectively. The first points are $p_{0,x}$. The projections of the second paths have first non-$2$ letter $1$. Hence every fragment fixes their cylinders pointwise. The cylinders of the third paths are contained in $P_{r+1}$ and are fixed pointwise by $s_n$. These are exactly the factors of $\epsilon'$ coming from the $q_1$- and $q_2$-coordinates in~\eqref{extractfunctioningpart}.

  It remains to choose the other points in the cycles. The model set $Y_{n-1}$ contains the two admissible paths whose projections are $1^{n-2}$ and $1^{n-3}2$. Let $\alpha_0,\beta_0$ be their images under $\phi_{n-1,0_0}$, and let $\alpha_1,\beta_1$ and $\alpha_2,\beta_2$ be their images under $\phi_{n-1,1}$ and $\phi_{n-1,2}$, respectively. After appending $0_1x$, all of their projections still begin with $1$, and hence every fragment fixes the corresponding cylinders pointwise. This proves the assertion about the full support of $\epsilon'$ and the two auxiliary points in the displayed cycles used for $\mathsf D_{n+1}$.

  For the construction of $\mathsf M_{n+1}$, use the two model paths in $Y_n$ whose projections are $1^{n-1}$ and $1^{n-2}2$, and take their images under the three maps $\psi_{n+1,x}$, $x\in I$. Their projections again begin with $1$. Finally, the connector $\mathsf e_{n+1}$ corresponds to the piece $P_{r+1}$, on which $s_n$ is the identity. This proves all the stated fixed-point assertions.
\end{proof}

\begin{prop}
  \label{prop:FG-selector-producing-fragmentation}
  The fragmentation of $\langle b\rangle$ shown in \eqref{gen1} and \eqref{gen2} is selector-producing for the gluing arguments used below.
\end{prop}

\begin{proof}
  Let $P_0,P_1,P_2,P_3$ be the four pieces defined above, indexed by the connector types \textup{(a)}, \textup{(b)}, \textup{(c)}, \textup{(d)}, and let $C\leq\mathbb{F}_3^4$ be the activity subspace computed in Lemma~\ref{selectorexist}. Every connector neighborhood of one type is contained in the corresponding piece, and thus Condition~\ref{cond:selector-fragmentation:F1} holds. In the gluing arguments below, the connectors of the current type meeting the supports of the cycles are active, with the same prescribed permutation under the identifications given by the descendant embeddings. The connectors of the next type are protected. Lemma~\ref{lem:FG-selector-phase-support} verifies the additional fixed-point requirements in the conjugation arguments. Since $d(C^\perp)=3$, Proposition~\ref{prop:dual-distance-selector-criterion} gives, for each required pair of restrictions on the current and next pieces, a fragment realizing that pair. Hence Condition~\ref{cond:selector-fragmentation:F2} holds.

  Only one singular point is used in these gluing arguments, and thus Condition~\ref{cond:selector-fragmentation:F3} is automatic. Every $P_j$ accumulates on $\xi=2^\omega$, which verifies Condition~\ref{cond:pure-non-Hausdorff:N1}. The subspace $C$ has the additional property that every nonzero word has weight exactly $3$: the kernels of the four coordinate projections are the four one-dimensional subspaces of $\mathbb{F}_3^2$. Hence every nonidentity fragment fixes exactly one of the four pieces pointwise. Proposition~\ref{prop:activity-identity-coordinate} gives Condition~\ref{cond:pure-non-Hausdorff:N2}.
\end{proof}

\begin{figure}[ht]
  \centering

  \begin{tikzpicture}[
      >=Stealth,
      vtx/.style={circle,fill=black,inner sep=1.9pt},
      elab/.style={font=\small, fill=white, inner sep=1pt},
      line/.style={->,line width=0.8pt}
    ]

    \newcommand{\connI}[2]{%
      \begin{scope}[shift={(#1,#2)}]
        \coordinate (L) at (0,0);
        \coordinate (R) at (3.6,0);
        \coordinate (T) at (1.8,2.55);

        \node[vtx] (Lv) at (L) {};
        \node[vtx] (Rv) at (R) {};
        \node[vtx] (Tv) at (T) {};

        \draw[line] (T) -- node[elab,left=3pt] {$b_2$} (L);
        \draw[line] (L) -- node[elab,below=2pt] {$b_2$} (R);
        \draw[line] (R) -- node[elab,right=3pt] {$b_2$} (T);

        \draw[line]
        (Tv) to[out=135,in=45,looseness=14]
        node[elab,above=7pt] {$b_1$} (Tv);

        \draw[line]
        (Lv) to[out=210,in=120,looseness=14]
        node[elab,left=8pt] {$b_1$} (Lv);

        \draw[line]
        (Rv) to[out=-30,in=60,looseness=14]
        node[elab,right=8pt] {$b_1$} (Rv);
      \end{scope}
    }

    \newcommand{\connII}[2]{%
      \begin{scope}[shift={(#1,#2)}]
        \coordinate (L) at (0,0);
        \coordinate (R) at (3.8,0);
        \coordinate (T) at (1.9,2.45);

        \node[vtx] (Lv) at (L) {};
        \node[vtx] (Rv) at (R) {};
        \node[vtx] (Tv) at (T) {};

        \draw[line]
        (T) to[bend right=18]
        node[elab,left=5pt,pos=.42] {$b_1$} (L);
        \draw[line]
        (T) to[bend left=18]
        node[elab,right=5pt,pos=.60] {$b_2$} (L);

        \draw[line]
        (L) to[bend left=6]
        node[elab,above=3pt,pos=.48] {$b_1$} (R);
        \draw[line]
        (L) to[bend right=15]
        node[elab,below=5pt,pos=.52] {$b_2$} (R);

        \draw[line]
        (R) to[bend right=18]
        node[elab,right=5pt,pos=.42] {$b_1$} (T);
        \draw[line]
        (R) to[bend left=18]
        node[elab,left=5pt,pos=.60] {$b_2$} (T);
      \end{scope}
    }

    \newcommand{\connIII}[2]{%
      \begin{scope}[shift={(#1,#2)}]
        \coordinate (L) at (0,0);
        \coordinate (R) at (3.8,0);
        \coordinate (T) at (1.9,2.45);

        \node[vtx] (Lv) at (L) {};
        \node[vtx] (Rv) at (R) {};
        \node[vtx] (Tv) at (T) {};

        \draw[line]
        (L) to[bend left=18]
        node[elab,left=5pt,pos=.42] {$b_1$} (T);
        \draw[line]
        (T) to[bend left=18]
        node[elab,right=5pt,pos=.60] {$b_2$} (L);

        \draw[line]
        (T) to[bend right=18]
        node[elab,left=3pt,pos=.42] {$b_1$} (R);
        \draw[line]
        (R) to[bend right=18]
        node[elab,right=4pt,pos=.60] {$b_2$} (T);

        \draw[line]
        (R) to[bend right=8]
        node[elab,above=3pt,pos=.48] {$b_1$} (L);
        \draw[line]
        (L) to[bend right=15]
        node[elab,below=3pt,pos=.52] {$b_2$} (R);
      \end{scope}
    }

    \newcommand{\connIV}[2]{%
      \begin{scope}[shift={(#1,#2)}]
        \coordinate (L) at (0,0);
        \coordinate (R) at (3.6,0);
        \coordinate (T) at (1.8,2.55);

        \node[vtx] (Lv) at (L) {};
        \node[vtx] (Rv) at (R) {};
        \node[vtx] (Tv) at (T) {};

        \draw[line] (T) -- node[elab,left=3pt] {$b_1$} (L);
        \draw[line] (L) -- node[elab,below=2pt] {$b_1$} (R);
        \draw[line] (R) -- node[elab,right=3pt] {$b_1$} (T);

        \draw[line]
        (Tv) to[out=135,in=45,looseness=14]
        node[elab,above=7pt] {$b_2$} (Tv);

        \draw[line]
        (Lv) to[out=210,in=120,looseness=14]
        node[elab,left=8pt] {$b_2$} (Lv);

        \draw[line]
        (Rv) to[out=-30,in=60,looseness=14]
        node[elab,right=8pt] {$b_2$} (Rv);
      \end{scope}
    }

    \begin{scope}[shift={(0,0)}]
      \connI{0}{0}
      \node at (1.8,-1.15)
      {\small (a) Connector $\mathsf{e}_n$, $n\equiv 2 \pmod{4}$};
    \end{scope}

    \begin{scope}[shift={(8.0,0)}]
      \connII{0}{0}
      \node at (1.9,-1.15)
      {\small (b) Connector $\mathsf{e}_n$, $n\equiv 3 \pmod{4}$};
    \end{scope}

    \begin{scope}[shift={(0,-6.2)}]
      \connIII{0}{0}
      \node at (1.9,-1.15)
      {\small (c) Connector $\mathsf{e}_n$, $n\equiv 0 \pmod{4}$};
    \end{scope}

    \begin{scope}[shift={(8.0,-6.2)}]
      \connIV{0}{0}
      \node at (1.8,-1.15)
      {\small (d) Connector $\mathsf{e}_n$, $n\equiv 1 \pmod{4}$};
    \end{scope}

  \end{tikzpicture}

  \caption{The connector graphs $\mathsf{e}_n$ used in the tile inflation of the modified Fabrykowski--Gupta group. The pattern starts at $\mathsf{e}_2$ and repeats with period $4$.}
  \label{fig:FG-periodic-connectors}
\end{figure}

The propagation below follows \cite[Lemma~5.3.17]{nek22}. Since the directed generator is fragmented, a selector must permute the current connector while fixing the next. The self-dual ternary $[4,2,3]$ code computed above supplies these selectors.

\begin{thm}
  \label{prop:FG-selector-system}
  The group $\widehat{G}_{\mathrm{FG}}$ admits an alternating propagation system with seed level $3$ and a finite refinement scheme whose reachable graph has the three edges
  \[
  0_0\longrightarrow0_1,\qquad
  1\longrightarrow0_1,\qquad
  2\longrightarrow0_1,
  \]
  with no outgoing edge from $0_1$. Consequently, Theorem~A gives absorption delay at most $2$ and
  \[
  \mathsf{A}(\mathfrak{T}_{\mathsf{B}})\leq\widehat{G}_{\mathrm{FG}}.
  \]
\end{thm}

\begin{proof}
  We write embedded elements as disjoint products on descendant copies. If $g\in\mathsf{S}(\mathcal{T}_{0_1,n})$, then $\iota_{n,n+1}(g)=\prod_{x\in I}g_x$, where $g_x$ acts as $g$ on $\mathcal{T}_{0_1,n}x$. At the next level,
  \[\iota_{n,n+2}(g)=\left(\prod_{x\in I}g_{x\,0_1}\right)\left(\prod_{x\in I}g_{1x}\right)\left(\prod_{x\in I}g_{2x}\right).\]
  Each factor acts as $g$ on the copy indicated by its subscript and identically outside that copy.

  We argue by induction on $n$. The levels $2$ and $3$ are treated separately in the following two lemmas.

  \begin{lemma}
    The group $\widehat{G}_{\mathrm{FG}}$ contains the internal direct product
    \[\mathsf{A}(\mathcal{T}_{0_0,2})\times\mathsf{A}(\mathcal{T}_{0_1,2})\times\mathsf{A}(\mathcal{T}_{1,2})\times\mathsf{A}(\mathcal{T}_{2,2}).\]
  \end{lemma}
  \begin{proof}
    For every $x\in\{0_0,0_1,1,2\}$, the tile $\mathcal{T}_{x,2}$ has exactly three vertices, and $a_x$ acts on these vertices as the displayed $3$-cycle in Figure~\ref{fig:four-triangular-connectors}. Therefore $\langle a_x\rangle=\mathsf{A}(\mathcal{T}_{x,2})$. Their supports are disjoint, and therefore $\mathsf A_2=\prod_{x\in\{0_0,0_1,1,2\}}\mathsf A(\mathcal T_{x,2})\leq\widehat G_{\mathrm{FG}}$.
  \end{proof}

  \begin{lemma}
    The group $\widehat{G}_{\mathrm{FG}}$ contains $\mathsf{A}(\mathcal{T}_{0_1,3})$ and the diagonal subgroup of $\mathsf{A}(\mathcal{T}_{0_0,3})\times\mathsf{A}(\mathcal{T}_{1,3})\times\mathsf{A}(\mathcal{T}_{2,3})$.
  \end{lemma}
  \begin{proof}
    Since $0_0$ has only the successor $0_1$, the image of $\mathsf{A}(\mathcal{T}_{0_0,2})$ under $\iota_{2,3}$ acts on $\mathcal{T}_{0_0,2}0_1$ and identically outside it. The groups $\mathsf{A}(\mathcal{T}_{1,2})$ and $\mathsf{A}(\mathcal{T}_{2,2})$ embed diagonally on their four descendant copies. Thus, for $\sigma\in\mathsf{A}(\mathcal{T}_{1,2})$ and $\delta\in\mathsf{A}(\mathcal{T}_{2,2})$, we have
    \[\iota_{2,3}(\sigma)=\sigma_{0_0}\sigma_{0_1}\sigma_1\sigma_2,\qquad \iota_{2,3}(\delta)=\delta_{0_0}\delta_{0_1}\delta_1\delta_2.\]

    The tile $\mathcal{T}_{0_1,3}$ is the union of $\mathcal{T}_{0_0,2}0_1$, $\mathcal{T}_{1,2}0_1$ and $\mathcal{T}_{2,2}0_1$, joined by $\mathsf{e}_2$. Let $p_0,p_1,p_2$ be the corresponding connecting points. Take $s=b_1^2b_2$. By Lemma~\ref{selectorexist}, $s$ acts nontrivially on the connecting points of $\mathsf{e}_2$ and identically on $U_{\mathsf{e}_3}$. Replacing $s$ by $s^{-1}$ if necessary, we may assume that it induces $(p_0\,p_1\,p_2)$. Let $x,y$ be the two points of $\mathcal{T}_{0_0,2}0_1\setminus\{p_0\}$ whose projections are $100$ and $200$, respectively. By Lemma~\ref{lem:FG-selector-phase-support}, both are fixed by $s$. Then
    \[\alpha_0=(p_0\,x\,y),\qquad \alpha_1=s\alpha_0s^{-1}=(p_1\,x\,y),\qquad \alpha_2=s^2\alpha_0s^{-2}=(p_2\,x\,y)\]
    belong to $\widehat{G}_{\mathrm{FG}}$. For every $u\in\mathcal{T}_{1,2}0_1$, choose $\sigma_u\in\mathsf{A}(\mathcal{T}_{1,2})$ such that $\iota_{2,3}(\sigma_u)(p_1)=u$. Similarly, for every $v\in\mathcal{T}_{2,2}0_1$, choose $\delta_v\in\mathsf{A}(\mathcal{T}_{2,2})$ such that $\iota_{2,3}(\delta_v)(p_2)=v$. Then
    \[\iota_{2,3}(\sigma_u)\alpha_1\iota_{2,3}(\sigma_u)^{-1}=(u\,x\,y),\qquad \iota_{2,3}(\delta_v)\alpha_2\iota_{2,3}(\delta_v)^{-1}=(v\,x\,y).\]
    The supports $\{u,x,y\}$ and $\{v,x,y\}$ form a connected cover of $\mathcal{T}_{1,2}0_1\cup\mathcal{T}_{2,2}0_1\cup\{x,y\}$. Hence Corollary~\ref{cor:connected-cover-gluing} gives
    \[H:=\langle(u\,x\,y),(v\,x\,y):u\in\mathcal{T}_{1,2}0_1,\ v\in\mathcal{T}_{2,2}0_1\rangle=\mathsf{A}(\mathcal{T}_{1,2}0_1\cup\mathcal{T}_{2,2}0_1\cup\{x,y\}).\]
    The support of $H$ meets $\mathcal{T}_{0_0,2}0_1$ in $\{x,y\}$. Hence Lemma~\ref{lem:overlap-gluing} gives
    \[\langle\mathsf{A}(\mathcal{T}_{0_0,2}0_1),H\rangle=\mathsf{A}(\mathcal{T}_{0_1,3})\leq\widehat{G}_{\mathrm{FG}}.\]
    See Figure~\ref{fig:triangular-selector-local-picture} for a schematic description.

    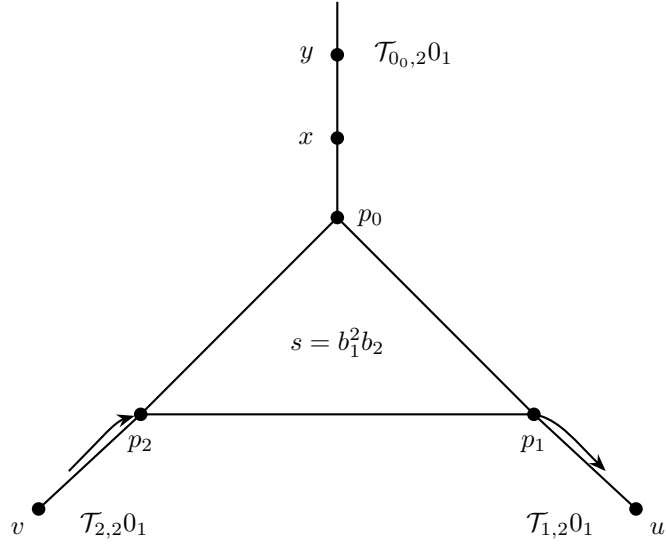
\begin{figure}[ht]
      \centering
      \begin{tikzpicture}[
          >=Stealth,
          vertex/.style={circle,fill=black,inner sep=1.8pt},
          lab/.style={font=\small},
          thickline/.style={line width=0.8pt}
        ]

        \coordinate (P0) at (0,2.6);
        \coordinate (P2) at (-2.6,0);
        \coordinate (P1) at (2.6,0);

        \coordinate (X)  at (0,3.65);
        \coordinate (Y)  at (0,4.75);

        \coordinate (V)  at (-3.95,-1.25);
        \coordinate (U)  at (3.95,-1.25);

        \draw[thickline] (P0) -- (P2) -- (P1) -- cycle;

        \draw[thickline] (P0) -- (0,5.45);
        \draw[thickline] (P2) -- (V);
        \draw[thickline] (P1) -- (U);

        \node[vertex] at (P0) {};
        \node[vertex] at (P2) {};
        \node[vertex] at (P1) {};
        \node[vertex] at (X)  {};
        \node[vertex] at (Y)  {};
        \node[vertex] at (V)  {};
        \node[vertex] at (U)  {};

        \node[lab,right=4pt] at (P0) {$p_0$};
        \node[lab,below=4pt] at (P2) {$p_2$};
        \node[lab,below=4pt] at (P1) {$p_1$};

        \node[lab,left=5pt] at (X) {$x$};
        \node[lab,left=5pt] at (Y) {$y$};

        \node[lab,below left=2pt] at (V) {$v$};
        \node[lab,below right=2pt] at (U) {$u$};

        \node[lab] at (0,0.95) {$s=b_1^2b_2$};

        \node[lab,right=10pt] at (0,4.75) {$\mathcal{T}_{0_0,2}0_1$};
        \node[lab,below=6pt] at (-2.95,-0.95) {$\mathcal{T}_{2,2}0_1$};
        \node[lab,below=6pt] at (2.95,-0.95) {$\mathcal{T}_{1,2}0_1$};

        \draw[->,thickline]
        (-3.55,-0.75)
        .. controls (-3.15,-0.35) and (-2.95,-0.10)
        .. (-2.68,-0.02);

        \draw[->,thickline]
        (2.68,-0.02)
        .. controls (2.95,-0.10) and (3.15,-0.35)
        .. (3.55,-0.75);

      \end{tikzpicture}
      \caption{A selector at the triangular connector graph.}
      \label{fig:triangular-selector-local-picture}
    \end{figure}

    Now consider the embedded groups $\iota_{2,3}(\mathsf{A}(\mathcal{T}_{1,2}))$ and $\iota_{2,3}(\mathsf{A}(\mathcal{T}_{2,2}))$. Their elements have the forms $\sigma_{0_0}\sigma_{0_1}\sigma_1\sigma_2$ and $\delta_{0_0}\delta_{0_1}\delta_1\delta_2$, respectively. Since $\mathsf{A}(\mathcal{T}_{0_1,3})\leq\widehat{G}_{\mathrm{FG}}$, we may cancel the $0_1$-components and obtain
    \[\Delta\bigl(\mathsf{A}(\mathcal{T}_{1,2}0_0),\mathsf{A}(\mathcal{T}_{1,2}1),\mathsf{A}(\mathcal{T}_{1,2}2)\bigr),\quad \Delta\bigl(\mathsf{A}(\mathcal{T}_{2,2}0_0),\mathsf{A}(\mathcal{T}_{2,2}1),\mathsf{A}(\mathcal{T}_{2,2}2)\bigr)\leq\widehat{G}_{\mathrm{FG}}.\]
    Since $0_1$ has the successors $0_0,1,2$, the diagonal embedding also gives
    \[\Delta\bigl(\mathsf{A}(\mathcal{T}_{0_1,2}0_0),\mathsf{A}(\mathcal{T}_{0_1,2}1),\mathsf{A}(\mathcal{T}_{0_1,2}2)\bigr)\leq\widehat{G}_{\mathrm{FG}}.\]

    For every $x\in I$, the tile $\mathcal{T}_{x,3}$ is the disjoint union $\mathcal{T}_{0_1,2}x\sqcup\mathcal{T}_{1,2}x\sqcup\mathcal{T}_{2,2}x$, joined by $\mathsf{e}_2$ at connecting points $p_{0,x},p_{1,x},p_{2,x}$. Using the same selector $s$, and replacing it by its inverse if necessary, we have
    \[s=\prod_{x\in I}(p_{0,x}\,p_{1,x}\,p_{2,x})\]
    on these connecting points. Let $\alpha_x,\beta_x$ be the two points of $\mathcal{T}_{0_1,2}x\setminus\{p_{0,x}\}$, chosen to correspond under the identifications given by the descendant embeddings for $x\in I$. By Lemma~\ref{lem:FG-selector-phase-support}, the selector fixes all of these auxiliary points. The third diagonal group above contains
    \[\pi_0:=\prod_{x\in I}(p_{0,x}\,\alpha_x\,\beta_x).\]
    Hence
    \[\pi_1:=s\pi_0s^{-1}=\prod_{x\in I}(p_{1,x}\,\alpha_x\,\beta_x),\qquad \pi_2:=s^2\pi_0s^{-2}=\prod_{x\in I}(p_{2,x}\,\alpha_x\,\beta_x).\]

    Under the identifications given by the descendant embeddings, write $p_{1,x}=\bar p_1x$ and $p_{2,x}=\bar p_2x$, where $\bar p_1\in\mathcal{T}_{1,2}$ and $\bar p_2\in\mathcal{T}_{2,2}$. Let $u\in\mathcal{T}_{1,2}$ and choose $h_u\in\mathsf{A}(\mathcal{T}_{1,2})$ with $h_u(\bar p_1)=u$. The cancellation above gives an element $\widehat{h}_u:=\prod_{x\in I}(h_u)_x\in\widehat{G}_{\mathrm{FG}}$, where $(h_u)_x$ acts on $\mathcal{T}_{1,2}x$. Therefore
    \[g_u:=\widehat{h}_u\pi_1\widehat{h}_u^{-1}=\prod_{x\in I}(ux\,\alpha_x\,\beta_x).\]
    Similarly, for $v\in\mathcal{T}_{2,2}$ choose $h_v\in\mathsf{A}(\mathcal{T}_{2,2})$ with $h_v(\bar p_2)=v$. The cancellation above gives $\widehat{h}_v:=\prod_{x\in I}(h_v)_x\in\widehat{G}_{\mathrm{FG}}$ such that
    \[g_v:=\widehat{h}_v\pi_2\widehat{h}_v^{-1}=\prod_{x\in I}(vx\,\alpha_x\,\beta_x).\]
    We then have
    \[g_u^{-1}g_v=\prod_{x\in I}(ux\,\beta_x\,vx).\]
    This is a product of three corresponding $3$-cycles. In each coordinate, its cycle meets all three copies of the preceding tiles. Those three supports, together with the support of the cycle, have connected intersection graph. Corollary~\ref{cor:diagonal-connected-cover-gluing} therefore gives
    \[\Delta\bigl(\mathsf{A}(\mathcal{T}_{0_0,3}),\mathsf{A}(\mathcal{T}_{1,3}),\mathsf{A}(\mathcal{T}_{2,3})\bigr)\leq\widehat{G}_{\mathrm{FG}}.\]
  \end{proof}

  At level $3$, the preceding lemma gives $\mathsf{D}_3\leq\widehat{G}_{\mathrm{FG}}$ and $\mathsf{A}(\mathcal{T}_{0_1,3})\leq\widehat{G}_{\mathrm{FG}}$. We take
  $\mathsf{M}_3:=\mathsf{A}(\mathcal{T}_{0_1,3})$ and put $\mathsf{C}_3:=\langle\mathsf{D}_3,\mathsf{E}_3,\mathsf{M}_3\rangle$. Since $\mathsf{E}_3\leq\mathsf{M}_3$, we have $\mathsf{C}_3\leq\widehat{G}_{\mathrm{FG}}$.

  Suppose now that $n\geq3$ and that a transitive subgroup $\mathsf{M}_n\leq\widehat{G}_{\mathrm{FG}}\cap\mathsf{A}(\mathcal{T}_{0_1,n})$ has been constructed. Put $\mathsf{C}_n:=\langle\mathsf{D}_n,\mathsf{E}_n,\mathsf{M}_n\rangle$. We prove that
  \[\mathsf{C}_{n+1}\leq
    \left\langle\iota_{n,n+1}(\mathsf{C}_n),s_n\right\rangle\qquad(n\geq3).\]
  Put $\mathsf{K}_{n+1}:=\langle\iota_{n,n+1}(\mathsf{C}_n),s_n\rangle$. For every $x\in I$, we have
  \begin{equation}
    \mathcal{T}_{x,n+1}=\mathcal{T}_{0_1,n}x\sqcup\mathcal{T}_{1,n}x\sqcup\mathcal{T}_{2,n}x
    \label{01decompose}
  \end{equation}
  as sets of vertices. Let $p_{0,x}\in\mathcal{T}_{0_1,n}x$, $p_{1,x}\in\mathcal{T}_{1,n}x$ and $p_{2,x}\in\mathcal{T}_{2,n}x$ be the connecting points, and write $p_{j,x}=p_jx$ for their truncations at level $n$. Let $s_n$ be the selector supplied by Lemma~\ref{selectorexist}. It acts nontrivially on the connecting points of $\mathsf{e}_n$ and identically on $U_{\mathsf{e}_{n+1}}$. Replacing it by its inverse if necessary, we assume that
  \[s_n=\prod_{x\in I}(p_{0,x}\,p_{1,x}\,p_{2,x})\]
  on the connecting points.

  We first construct $\mathsf{D}_{n+1}$. Let $p_0\in\mathcal{T}_{0_1,n}$ be the truncation of $p_{0,x}$. We construct an element
  \[
    \pi_0=\epsilon'\prod_{x\in I}(p_{0,x}\,\alpha_x\,\beta_x)
    \in\iota_{n,n+1}(\mathsf C_n),
  \]
  where the support of $\epsilon'$ is disjoint from the supports of the displayed cycles and is fixed pointwise by $s_n$.

  For $n=3$, recall that $\mathsf M_3=\mathsf A(\mathcal T_{0_1,3})$. The tile $\mathcal T_{0_1,3}$ contains the two paths whose projections are $110$ and $120$. Denote them by $\bar\alpha,\bar\beta$. Lemma~\ref{lem:FG-selector-phase-support} shows that both are fixed by $s_3$. Thus
  $(p_0\,\bar\alpha\,\bar\beta)\in\mathsf M_3$, and we take
  \[
    \pi_0=\iota_{3,4}(p_0\,\bar\alpha\,\bar\beta)
    =\prod_{x\in I}(p_{0,x}\,\bar\alpha x\,\bar\beta x),
    \qquad \epsilon'=\Id.
  \]

  Suppose now that $n\geq4$. The connecting point $p_0$ lies in the subcopy $\mathcal{T}_{0_0,n-1}0_1$. Write $p_0=q_0\,0_1$ with $q_0\in\mathcal{T}_{0_0,n-1}$. Choose the corresponding points $\alpha_j,\beta_j$ as in Lemma~\ref{lem:FG-selector-phase-support}. Then
  \begin{equation}
    \epsilon=(q_0\,0_1\;\alpha_0\,0_1\;\beta_0\,0_1)(q_1\,0_1\;\alpha_1\,0_1\;\beta_1\,0_1)(q_2\,0_1\;\alpha_2\,0_1\;\beta_2\,0_1)\in\mathsf{E}_n.
    \label{longproduct}
  \end{equation}
  Put $\bar\alpha=\alpha_0\,0_1$ and $\bar\beta=\beta_0\,0_1$. Refining $\epsilon$ gives
  \begin{equation}
    \pi_0:=\iota_{n,n+1}(\epsilon)
    =\epsilon'\prod_{x\in I}(p_{0,x}\,\bar\alpha x\,\bar\beta x),
    \label{extractfunctioningpart}
  \end{equation}
  where $\epsilon'$ is the product of the descendants of the other two cycles in~\eqref{longproduct}. Lemma~\ref{lem:FG-selector-phase-support} shows that $s_n$ fixes its support and fixes $\bar\alpha x,\bar\beta x$ for every $x\in I$.

  In both cases put $\alpha_x=\bar\alpha x$ and $\beta_x=\bar\beta x$. Then
  \[
    \pi_1=s_n\pi_0s_n^{-1}=\epsilon'\prod_{x\in I}(p_{1,x}\,\alpha_x\,\beta_x),
    \qquad
    \pi_2=s_n^2\pi_0s_n^{-2}=\epsilon'\prod_{x\in I}(p_{2,x}\,\alpha_x\,\beta_x).
  \]

  Let $u\in\mathcal{T}_{1,n}$. The $\mathcal{T}_{1,n}$-coordinate of $\mathsf{D}_n$ is the full alternating group, and thus we may choose $d_u\in\mathsf{D}_n$ such that $d_u(p_1)=u$. The other coordinates of $\iota_{n,n+1}(d_u)$ are disjoint from the support of $\pi_1$, and therefore
  \[\varphi_u:=\iota_{n,n+1}(d_u)\pi_1\iota_{n,n+1}(d_u)^{-1}=\epsilon'\prod_{x\in I}(ux\,\alpha_x\,\beta_x).\]
  Similarly, for every $v\in\mathcal{T}_{2,n}$, we obtain $\psi_v=\epsilon'\prod_{x\in I}(vx\,\alpha_x\,\beta_x)$. For distinct $u,u'\in\mathcal{T}_{1,n}$, we have
  \[\varphi_{u'}^{-1}\varphi_u=\prod_{x\in I}(ux\,u'x\,\beta_x).\]
  As $u,u'$ vary, these elements generate
  \[\mathsf{L}_1:=\Delta_{x\in I}\mathsf{A}(\mathcal{T}_{1,n}x\cup\{\beta_x\}).\]
  Likewise, the corresponding elements obtained from the $\psi_v$ generate
  \[\mathsf{L}_2:=\Delta_{x\in I}\mathsf{A}(\mathcal{T}_{2,n}x\cup\{\beta_x\}).\]
  By Corollary~\ref{cor:diagonal-common-anchor-gluing},
  \[\mathsf{L}_{12}:=\langle\mathsf{L}_1,\mathsf{L}_2\rangle=\Delta_{x\in I}\mathsf{A}(\mathcal{T}_{1,n}x\cup\mathcal{T}_{2,n}x\cup\{\beta_x\}).\]
  Choose distinct $u_1,u_2\in\mathcal{T}_{1,n}$. Then $\mathsf{L}_{12}$ contains $\pi_\beta=\prod_{x\in I}(u_1x\,u_2x\,\beta_x)$. For $\gamma\in\mathcal{T}_{0_1,n}$, choose $r_\gamma\in\mathsf{M}_n$ with $r_\gamma(\bar\beta)=\gamma$. We obtain
  \[\iota_{n,n+1}(r_\gamma)\pi_\beta\iota_{n,n+1}(r_\gamma)^{-1}=\prod_{x\in I}(u_1x\,u_2x\,\gamma x).\]
  Put $U_x:=\mathcal{T}_{1,n}x\sqcup\mathcal{T}_{2,n}x$ and $Z_x:=\{u_1x,u_2x,\gamma x\}$. The group $\mathsf{L}_{12}$ contains $\Delta_{x\in I}\mathsf{A}(U_x)$, and the preceding product of $3$-cycles generates $\Delta_{x\in I}\mathsf{A}(Z_x)$. Since $U_x\cap Z_x=\{u_1x,u_2x\}$, Corollary~\ref{cor:diagonal-connected-cover-gluing} gives $\Delta_{x\in I}\mathsf{A}(U_x\sqcup\{\gamma x\})\leq\mathsf{K}_{n+1}$ for every $\gamma\in\mathcal{T}_{0_1,n}$. Applying Corollary~\ref{cor:adjoining-points-common-support} in the model copy and Lemma~\ref{lem:diagonal-lifting}, we obtain
  \[\mathsf{D}_{n+1}=\Delta\bigl(\mathsf{A}(\mathcal{T}_{0_0,n+1}),\mathsf{A}(\mathcal{T}_{1,n+1}),\mathsf{A}(\mathcal{T}_{2,n+1})\bigr)\leq\mathsf{K}_{n+1}.\]

  We next construct $\mathsf{E}_{n+1}$. Let $\sigma\in\mathsf{D}_n$. Since $0_0$ has only the successor $0_1$, while $1$ and $2$ have the successors $0_0,0_1,1,2$, we may write $\iota_{n,n+1}(\sigma)=h_{0_1}(\sigma)h_I(\sigma)$, where
  \[h_{0_1}(\sigma):=
    \sigma_{\mathcal{T}_{0_0,n}0_1}\sigma_{\mathcal{T}_{1,n}0_1}\sigma_{\mathcal{T}_{2,n}0_1}\]
  and
  \[h_I(\sigma):=
    \prod_{x\in I}\sigma_{\mathcal{T}_{1,n}x}\sigma_{\mathcal{T}_{2,n}x}.
  \]
  The element $h_I(\sigma)$ belongs to $\mathsf{D}_{n+1}$: in every target tile it acts as the same even permutation on the $\mathcal{T}_{1,n}x$- and $\mathcal{T}_{2,n}x$-subcopies and trivially on the $\mathcal{T}_{0_1,n}x$-subcopy. Therefore $h_{0_1}(\sigma)=\iota_{n,n+1}(\sigma)h_I(\sigma)^{-1}$ belongs to $\mathsf{K}_{n+1}$. As $\sigma$ varies, these elements generate
  \[\mathsf{E}_{n+1}=
    \Delta\bigl(\mathsf{A}(\mathcal{T}_{0_0,n}0_1),\mathsf{A}(\mathcal{T}_{1,n}0_1),\mathsf{A}(\mathcal{T}_{2,n}0_1)\bigr)\leq\mathsf{K}_{n+1}.\]

  We finally construct $\mathsf{M}_{n+1}$. The tile $\mathcal{T}_{0_1,n+1}$ is the disjoint union $\mathcal{T}_{0_0,n}0_1\sqcup\mathcal{T}_{1,n}0_1\sqcup\mathcal{T}_{2,n}0_1$. Let $p_x\in\mathcal{T}_{x,n}0_1$, $x\in I$, be the connecting points. In the common model $Y_n$, choose the two paths whose projections are $1^{n-1}$ and $1^{n-2}2$, and let $\alpha_x,\beta_x$ be their corresponding images in $\mathcal{T}_{x,n}0_1$. Lemma~\ref{lem:FG-selector-phase-support} shows that all these points are fixed by $s_n$. Then $\pi:=\prod_{x\in I}(p_x\,\alpha_x\,\beta_x)$ belongs to $\mathsf{E}_{n+1}$. On the connecting points, $s_n$ acts as $(p_{0_0}\,p_1\,p_2)$. Put
  \[h_1=(s_n\pi s_n^{-1})\pi^{-1},\qquad h_2=(s_n^2\pi s_n^{-2})\pi^{-1}.\]
  Evaluating these products at $p_{0_0}$ gives $h_1(p_{0_0})=p_1$ and $h_2(p_{0_0})=p_2$. Both elements are even and supported in $\mathcal{T}_{0_1,n+1}$. Hence
  \[\mathsf{M}_{n+1}:=\langle\mathsf{E}_{n+1},h_1,h_2\rangle\leq\mathsf{K}_{n+1}\cap\mathsf{A}(\mathcal{T}_{0_1,n+1}).\]
  The group $\mathsf{E}_{n+1}$ is transitive on each of its three coordinate supports, while $h_1$ and $h_2$ connect their connecting points. Therefore $\mathsf{M}_{n+1}$ acts transitively on $\mathcal{T}_{0_1,n+1}$.

  The constructions above give
  \[\mathsf{C}_{n+1}\leq\mathsf{K}_{n+1}
    =\left\langle\iota_{n,n+1}(\mathsf{C}_n),s_n\right\rangle.\]

  Put $\mathscr{D}_n:=\{\mathcal{D}_n,\mathcal{E}_n\}$. Then $\mathsf{C}_n$ contains the alternating groups associated with both data and the auxiliary subgroup $\mathsf{M}_n$. The level-$3$ construction verifies Condition~\ref{cond:selector-system:A1}, and the displayed inclusion verifies Condition~\ref{cond:selector-system:A2}. Hence $\{(\mathscr{D}_n,\mathsf{C}_n)\}_{n\geq3}$ is an alternating propagation system.

  The next lemma records the subgroup inclusions after one refinement that are needed in the final step of the containment argument. For $w\in\widehat{\mathsf{X}}$ and $m\geq3$, let $\mathscr{R}_m(w)$ be the collection of independent alternating data $\mathcal{R}=(Y;\phi)$ at level $m$ whose coordinate support is contained in $\mathcal{T}_{w,m}$. The groups $\mathsf{D}_{m+1}$ and $\mathsf{E}_{m+1}$ are diagonal copies of full alternating groups. They therefore contain the simultaneous copies of every even permutation on a subset of the model set. This allows the refinement scheme to treat arbitrary such $\mathcal{R}$ rather than only factors corresponding to whole tiles.

  \begin{lemma}
    \label{lem:FG-refinement-graph}
    The collections $\mathscr{R}_m(w)$ form a finite refinement scheme with vertex set $\widehat{\mathsf{X}}=\{0_0,0_1,1,2\}$ and directed edges
    \[
    0_0\longrightarrow0_1,\qquad
    1\longrightarrow0_1,\qquad
    2\longrightarrow0_1.
    \]
    There is no outgoing edge from $0_1$.
  \end{lemma}

  \begin{proof}
    Every canonical tile datum $\mathcal{U}_{w,m}$ belongs to $\mathscr{R}_m(w)$, and thus Condition~\ref{cond:refinement:R1} holds. Let $\mathcal{R}=(Y;\phi)\in\mathscr{R}_m(w)$ and let $\sigma\in\mathsf{A}(Y)$. We examine the successors of $w$ in Figure~\ref{fig:FG-modified-bratteli}.

    If $w=0_1$, its successors are exactly the three vertices in $I=\{0_0,1,2\}$. The embedded element is the product of the three corresponding copies of $\phi\sigma\phi^{-1}$ on the descendant supports $\phi(Y)x$, $x\in I$. These supports occupy corresponding positions in the three tiles $\mathcal{T}_{x,m+1}$, and therefore the product belongs to $\mathsf{D}_{m+1}\leq\mathsf{C}_{m+1}^{\uparrow}$. Thus no group assigned to a successor vertex is needed for $0_1$.

    If $w=0_0$, its only successor is $0_1$. Hence $\iota_{m,m+1}(\mathsf{A}(\mathcal{R}))=\mathsf{A}(\mathcal{R}')$ for an independent datum $\mathcal{R}'\in\mathscr{R}_{m+1}(0_1)$. This gives the edge $0_0\to0_1$.

    Suppose that $w=1$ or $w=2$. Its successors are $0_0,0_1,1,2$. The product of the copies of $\phi\sigma\phi^{-1}$ with terminal vertices in $I$ belongs to $\mathsf{D}_{m+1}$. The remaining copy acts on $\phi(Y)0_1$ and belongs to $\mathsf{A}(\mathcal{R}')$, where $\mathcal{R}'\in\mathscr{R}_{m+1}(0_1)$ is the independent datum obtained by appending $0_1$. Therefore
    \[
    \iota_{m,m+1}\bigl(\mathsf{A}(\mathcal{R})\bigr)
    \leq
    \left\langle
    \mathsf{C}_{m+1}^{\uparrow},
    \mathsf{A}(\mathcal{R}')
    \right\rangle
    \]
    for some $\mathcal{R}'\in\mathscr{R}_{m+1}(0_1)$. This gives the edges $1\to0_1$ and $2\to0_1$ and verifies Condition~\ref{cond:refinement:R2}.
  \end{proof}

  The reachable graph in Lemma~\ref{lem:FG-refinement-graph} is acyclic and has maximal path length one. Theorem~A therefore gives absorption delay at most $2$ and $\mathsf{A}(\mathfrak{T}_{\mathsf{B}})\leq\widehat{G}_{\mathrm{FG}}$.
\end{proof}

\begin{cor}
  \label{fg:bounded-absorption}
  For every $n\geq3$, the stronger inclusion
  \[
  \iota_{n,n+2}(\mathsf{A}_n)
  \leq\langle\mathsf{E}_{n+2},\mathsf{D}_{n+2}\rangle
  \leq\mathsf{C}_{n+2}
  \]
  holds.
\end{cor}

\begin{proof}
  Let $g\in\mathsf{A}(\mathcal{T}_{0_0,n})$. Then
  \[
  \iota_{n,n+2}(g)=\prod_{x\in I}g_{0_1x}\in\mathsf{D}_{n+2}.
  \]
  If $g\in\mathsf{A}(\mathcal{T}_{0_1,n})$, then
  \[
  \iota_{n,n+2}(g)=
  \left(\prod_{x\in I}g_{x\,0_1}\right)
  \left(\prod_{x\in I}g_{1x}\right)
  \left(\prod_{x\in I}g_{2x}\right).
  \]
  The first factor belongs to $\mathsf{E}_{n+2}$ and the last two belong to $\mathsf{D}_{n+2}$.

  Finally, let $g\in\mathsf{A}(\mathcal{T}_{1,n})$ or $g\in\mathsf{A}(\mathcal{T}_{2,n})$. Then
  \[
  \iota_{n,n+2}(g)=
  \left(\prod_{x\in I}g_{x\,0_1}\right)
  \left(\prod_{x\in I}g_{0_1x}\right)
  \left(\prod_{x\in I}g_{1x}\right)
  \left(\prod_{x\in I}g_{2x}\right).
  \]
  The first factor belongs to $\mathsf{E}_{n+2}$ and the other three belong to $\mathsf{D}_{n+2}$. Since $\mathsf{A}_n$ is the direct product of the four tile factors, the stated inclusion follows.
\end{proof}

The preceding construction requires a cyclic permutation on the connecting points of $\mathsf e_n$ and the identity on the neighborhoods of $\mathsf e_{n+1}$. The other activity coordinates of the selector do not enter the conjugations. We vary the fragmentation while keeping the Bratteli diagram, lifted directed generator $b$, and the finitary generators $a_{0_0},a_{0_1},a_1,a_2$ unchanged.

For an integer $m\geq2$, put
\[
P_j^{(m)}:=\lambda^{-1}\left(\bigsqcup_{\substack{k\geq0\\k\equiv j\pmod m}}[2^k0]\right),
\qquad j\in\mathbb Z/m\mathbb Z.
\]
Let $C\leq\mathbb F_3^m$ be a linear subspace. For $c=(c_j)\in C$, define $f_c$ to act as the lifted $b^{c_j}$ on $P_j^{(m)}$ and identically on their complement. Put
\[
\mathcal F_C:=\{f_c:c\in C\},\qquad
\widehat G(C):=\langle a_{0_0},a_{0_1},a_1,a_2,\mathcal F_C\rangle.
\]
For $m=4$ and the activity subspace of Lemma~\ref{selectorexist}, we recover $\widehat G(C)=\widehat G_{\mathrm{FG}}$.

\begin{prop}
\label{prop:FG-periodic-activity-containment}
Suppose that for every $j\in\mathbb Z/m\mathbb Z$ the coordinate projection
\[
C\longrightarrow\mathbb F_3^2,
\qquad c\longmapsto(c_j,c_{j+1})
\]
is surjective. Then
\[
\mathsf A(\mathfrak T_{\mathsf B})\leq\widehat G(C).
\]
More precisely, the diagonal groups $\mathsf D_n,\mathsf E_n$ defined above and transitive groups $\mathsf M_n(C)\leq\mathsf A(\mathcal T_{0_1,n})$ give an alternating propagation system with seed level $3$ and absorption delay at most $2$. In particular, for every $n\geq3$,
\[
\iota_{n,n+2}(\mathsf A_n)
\leq\langle\mathsf E_{n+2},\mathsf D_{n+2}\rangle
\leq\widehat G(C).
\]
\end{prop}

\begin{proof}
The sets $\lambda^{-1}([2^k0])$ are disjoint clopen sets invariant under the lifted $b$. The only possible new boundary point of a union $P_j^{(m)}$ is $\xi=2^\omega$, which is fixed by $b$. Thus the piecewise maps $f_c$ are homeomorphisms with inverses $f_{-c}$. Their actions commute, since on each piece they are powers of $b$, and the coordinate actions are faithful. Hence $\mathcal F_C\cong(C,+)$. Since $m$ is finite, these fragments relabel the same connectors with finitely many phases and introduce no additional boundary positions in the tile inflation.

For $n\geq2$ put $j\equiv n-2\pmod m$. By surjectivity, choose $c^{(j)}\in C$ with
\[
c^{(j)}_j=1,\qquad c^{(j)}_{j+1}=0,
\]
and let $s_n=f_{c^{(j)}}$. The connecting points of every copy of the current connector $\mathsf e_n$ have projections beginning with $2^{n-2}0$. Hence $s_n$ induces the same cyclic permutation of these three points in each copy. The connector neighborhood $U_{\mathsf e_{n+1}}$ is contained in $P_{j+1}^{(m)}$, and thus $s_n$ fixes it pointwise.

We next verify that the selector fixes the auxiliary cylinders required by the conjugations. At level $2$, the auxiliary points have projections $100$ and $200$. The cylinder of the first point is fixed by every fragment, since its projection begins with $1$. The cylinder of the second point is contained in $P_{j+1}^{(m)}$ for $j=0$, and hence is fixed pointwise by $s_2$. At level $3$, the auxiliary points have projections $110$ and $120$. Since both projections begin with $1$, their cylinders are fixed pointwise by every fragment. For $n\geq4$, the points corresponding to $q_1$ and $q_2$ in the refined element~\eqref{extractfunctioningpart} have projections beginning with $2^{n-2}1$ or $2^{n-1}0$. The cylinders corresponding to the first family are fixed pointwise by every fragment. Those corresponding to the second family are fixed pointwise by $s_n$, since its $(j+1)$-coordinate is zero. All other auxiliary points used to construct $\mathsf D_{n+1}$ and $\mathsf M_{n+1}$ have projections beginning with $1$, and hence their cylinders are fixed pointwise as well. These are the conditions on cylinders verified in Lemma~\ref{lem:FG-selector-phase-support}. The argument uses only the current and next coordinates, with no restriction to period four.

For linear $C$, surjectivity of the projections onto consecutive coordinates is equivalent to the existence of these selectors at every phase. Indeed, if $c^{(j)}$ has coordinates $(1,0)$ at $(j,j+1)$ for every $j$, then $c^{(j+1)}$ has a nonzero $(j+1)$-coordinate. Thus the images of $c^{(j)}$ and $c^{(j+1)}$ in coordinates $(j,j+1)$ are linearly independent.

The four finitary generators give $\mathsf A_2\leq\widehat G(C)$. The level-$3$ construction in the proof of Theorem~\ref{prop:FG-selector-system} uses the inherited groups at level $2$ and the action of $s_2$ on the connecting points and the two fixed auxiliary cylinders. The preceding verification supplies exactly these actions. Hence $\mathsf A(\mathcal T_{0_1,3})$ and $\mathsf D_3$ are contained in $\widehat G(C)$. Set $\mathsf M_3(C)=\mathsf A(\mathcal T_{0_1,3})$ and
\[
\mathsf C_3(C)=\langle\mathsf D_3,\mathsf E_3,\mathsf M_3(C)\rangle\leq\widehat G(C).
\]

Suppose inductively that $\mathsf D_n,\mathsf E_n$ and a transitive subgroup $\mathsf M_n(C)\leq\mathsf A(\mathcal T_{0_1,n})$ have been constructed. Put
\[
\mathsf C_n(C)=\langle\mathsf D_n,\mathsf E_n,\mathsf M_n(C)\rangle,
\quad
K_{n+1}=\langle\iota_{n,n+1}(\mathsf C_n(C)),s_n\rangle.
\]
For $n=3$, the equality $\mathsf M_3(C)=\mathsf A(\mathcal T_{0_1,3})$ supplies the required $3$-cycle. Its image under refinement is a synchronized product of cycles through the connecting points, with no extra factor. For $n\geq4$, the group $\mathsf E_n$ supplies an element of the form
\[
\pi_0=\epsilon'\prod_{x\in I}(p_{0,x}\,\alpha_x\,\beta_x)
\in\iota_{n,n+1}(\mathsf C_n(C)).
\]
The preceding verification shows that $s_n$ fixes the support of $\epsilon'$ pointwise, fixes $\alpha_x$ and $\beta_x$, and cyclically permutes $p_{0,x},p_{1,x},p_{2,x}$. The diagonal group $\mathsf D_n$ moves $p_{1,x}$ and $p_{2,x}$ through their respective component supports. Conjugating $\pi_0$ by $s_n,s_n^2$ and these inherited permutations gives elements
\[
\varphi_u=\epsilon'\prod_{x\in I}(ux\,\alpha_x\,\beta_x),
\qquad
\psi_v=\epsilon'\prod_{x\in I}(vx\,\alpha_x\,\beta_x).
\]
For distinct $u,u'\in\mathcal T_{1,n}$, cancellation gives
\[
\varphi_{u'}^{-1}\varphi_u
=\prod_{x\in I}(ux\,u'x\,\beta_x),
\]
with an analogous formula on $\mathcal T_{2,n}$. The gluing lemmas and the transitivity of $\mathsf M_n(C)$ then give $\mathsf D_{n+1}\leq K_{n+1}$, exactly as in the proof of Theorem~\ref{prop:FG-selector-system}. We decompose each refined element of $\mathsf D_n$ into its component with terminal symbol $0_1$ and its components on the other descendant copies. The latter product belongs to $\mathsf D_{n+1}$, and canceling it gives $\mathsf E_{n+1}\leq K_{n+1}$. Finally, the two elements $(s_n\pi s_n^{-1})\pi^{-1}$ and $(s_n^2\pi s_n^{-2})\pi^{-1}$, with $\pi\in\mathsf E_{n+1}$ as in that proof, connect the three component supports inside $\mathcal T_{0_1,n+1}$. Together with $\mathsf E_{n+1}$ they generate a transitive group $\mathsf M_{n+1}(C)\leq K_{n+1}\cap\mathsf A(\mathcal T_{0_1,n+1})$. Thus
\[
\mathsf C_{n+1}(C)
\leq\langle\iota_{n,n+1}(\mathsf C_n(C)),s_n\rangle
\leq\widehat G(C).
\]

The refinement calculation in Lemma~\ref{lem:FG-refinement-graph} depends only on the successor pattern of $\mathsf B$ and the groups $\mathsf D_n,\mathsf E_n$. Its reachable graph has maximal path length one. The calculation in Corollary~\ref{fg:bounded-absorption} gives the stronger inclusion after two refinements. Taking the direct limit gives $\mathsf A(\mathfrak T_{\mathsf B})\leq\widehat G(C)$.
\end{proof}

The condition in Proposition~\ref{prop:FG-periodic-activity-containment} is equivalent to the existence, at every phase $j$, of a fragment $f_c\in\mathcal F_C$ with $c_j=1$ and $c_{j+1}=0$. This condition is sufficient for containment, but no converse is asserted: other elements of $\widehat G(C)$ may provide the required permutations.

\begin{exmp}
\label{exmp:FG-five-periodic}
Let $m=5$, and let $C\leq\mathbb F_3^5$ be the row space of
\[
\begin{pmatrix}
0&1&0&2&1\\
1&1&1&1&0
\end{pmatrix}.
\]
The determinants of the five consecutive column pairs, including the last and first columns, are $2,1,1,2,1$ in $\mathbb F_3$. Thus Proposition~\ref{prop:FG-periodic-activity-containment} gives $\mathsf A(\mathfrak T_{\mathsf B})\leq\widehat G(C)$. Since the first and third columns are equal, $(1,0,-1,0,0)\in C^\perp$. Every column is nonzero, and hence $d(C^\perp)=2$. Therefore the sufficient condition $d(C^\perp)>2$ in Proposition~\ref{prop:dual-distance-selector-criterion} fails for this example. No index computation for $\widehat G(C)$ is asserted here.
\end{exmp}

\subsection{AF parity and the AF symmetric group}
\label{subsec:FG-AF-parity}

We now compute AF parity and localize singular germs for the original four-periodic group $\widehat{G}_{\mathrm{FG}}$, in order to apply the completion theorem of \cite{kua26no1}. The index calculation does not concern general $\widehat G(C)$.

With the vertices ordered as $0_0,0_1,1,2$, the incidence matrix is
\[
  A_{\mathrm{FG}}=
  \begin{pmatrix}
    0 & 1 & 0 & 0 \\
    1 & 0 & 1 & 1 \\
    1 & 1 & 1 & 1 \\
    1 & 1 & 1 & 1
  \end{pmatrix}.
\]
With column vectors, the bonding matrix used to compute the homology and parity groups of the AF groupoid is
\[
  M_{\mathsf{B}}=
  A_{\mathrm{FG}}^{\mathsf{T}}=
  \begin{pmatrix}
    0 & 1 & 1 & 1 \\
    1 & 0 & 1 & 1 \\
    0 & 1 & 1 & 1 \\
    0 & 1 & 1 & 1
  \end{pmatrix}.
\]

For the calculation of the index, we only need the reduction of $M_{\mathsf B}$ modulo $2$.

The following lemma treats a fixed word whose product belongs to the full group of the tail groupoid. It replaces the designated factors by AF homeomorphisms without changing that product.

\begin{lemma}
  \label{lem:finite-word-AF-truncation}
  Let $\mathfrak{T}_{\mathsf{B}}$ be the tail groupoid of a Bratteli diagram $\mathsf{B}$ whose space of infinite paths is a Cantor set. Let $K$ be a finite group of homeomorphisms fixing $\xi$, and suppose that $(k,x)\in\mathfrak{T}_{\mathsf{B}}$ for every $k\in K$ and $x\neq\xi$. Consider a word $g=g_m\cdots g_1$ whose factors are designated either as elements of $K$ or as elements of $\mathsf{S}(\mathfrak{T}_{\mathsf{B}})$. Suppose that $g\in\mathsf{S}(\mathfrak{T}_{\mathsf{B}})$.

  There is a clopen $K$-invariant neighborhood $U$ of $\xi$ such that replacing every designated factor $k\in K$ by the AF homeomorphism
  \[
    k^{[U]}(x)=
    \begin{cases}
      x,&x\in U,\\
      k(x),&x\notin U
    \end{cases}
  \]
  leaves the product $g$ unchanged. The map $k\mapsto k^{[U]}$ is a homomorphism $K\to\mathsf{S}(\mathfrak{T}_{\mathsf{B}})$.
\end{lemma}

The neighborhood $U$ may depend on the word. No representation of the whole acting group by permutations at one finite level is assumed.

\begin{proof}
  Put $w_0=\Id$ and $w_i=g_i\cdots g_1$. Every factor preserves each $\mathfrak{T}_{\mathsf{B}}$-orbit as a set: a factor in $K$ fixes $\xi$ and has an AF germ elsewhere.

  For every position $i$ designated as a $K$-factor, put $x_i=w_{i-1}^{-1}(\xi)$, and let $E$ be the finite set of these points. If $E$ is empty, any clopen $K$-invariant neighborhood of $\xi$ suffices, since there is no factor to replace. Otherwise let $Q$ consist of $\xi$ and all points $w_j(x)$ with $x\in E$ and $0\leq j\leq m$. Every point of $Q$ lies in the AF orbit of $\xi$. For each $z\in Q$, choose a compact open AF bisection inducing a partial homeomorphism $\alpha_z$ from a neighborhood of $\xi$ to a neighborhood of $z$, with $\alpha_z(\xi)=z$ and $\alpha_\xi=\Id$.

  Consider a step $z=w_{i-1}(x)\mapsto g_i z$ along one of these finitely many trajectories. If $g_i$ is a designated $K$-factor and $z=\xi$, put $k_{i,z}=g_i$. In every other case put $k_{i,z}=\Id$. Since the other factor germs are AF and the AF groupoid is principal, the germs at $\xi$ of $g_i\alpha_z$ and $\alpha_{g_i z}k_{i,z}$ agree.

  Shrink a common clopen neighborhood of $\xi$ so that all these finitely many equalities of germs become equalities of maps, and so that the images under the $\alpha_z$, for distinct $z\in Q$, are disjoint. Also require that every element of $K$ with trivial germ at $\xi$ acts identically on this neighborhood. This is possible since $K$ is finite. Intersect its $K$-translates to obtain a clopen $K$-invariant neighborhood $U$ on which all the equalities still hold. Write $U_z=\alpha_z(U)$. In particular, $U_\xi=U$.

  Along a trajectory starting at $x\in E$, composition gives
  \[
  w_j\alpha_x=\alpha_{w_j(x)}h_{j,x}
  \quad\text{on }U
  \]
  for some $h_{j,x}\in K$. Since $g$ is AF, $h_{m,x}$ has trivial germ at $\xi$, and hence acts identically on $U$. Therefore $g\alpha_x=\alpha_{g(x)}$ on $U$.

  Let $\widetilde g_i$ be the truncated factor at a designated $K$-position and the unchanged factor otherwise. If a step is at $z=\xi$ and is designated as a $K$-factor, truncation replaces $k_{i,z}$ by the identity. At a step with $z\neq\xi$, the set $U_z$ is disjoint from $U$, and thus truncation leaves the factor unchanged there. A factor not designated for truncation is unchanged everywhere. It follows in every case that $\widetilde g_i\alpha_z=\alpha_{g_i z}$ on $U$. Thus the truncated word agrees with $g$ on every $U_x$, $x\in E$.

  At a designated $K$-position $i$, the preceding formula for $w_{i-1}$ and invariance of $U$ under $K$ give
  $w_{i-1}^{-1}(U)=U_{x_i}$.
  A point outside $\bigcup_{x\in E}U_x$ therefore never enters $U$ before such a factor in the original word. By induction on the positions, the truncated and original trajectories agree there as well. The products consequently agree everywhere.

  Since $U$ and its complement are $K$-invariant, truncation is a homomorphism. Each truncated element has only AF germs: it is the identity on $U$ and agrees with $k$ on its complement. By compactness, it belongs to $\mathsf{S}(\mathfrak{T}_{\mathsf{B}})$.
\end{proof}

\begin{cor}
  \label{cor:odd-singular-AF-parity}
  Under the hypotheses on $K$ in Lemma~\ref{lem:finite-word-AF-truncation}, suppose that $|K|$ is odd. If $R\leq\mathsf{A}(\mathfrak{T}_{\mathsf{B}})$, then
  \[
  \langle R,K\rangle\cap\mathsf{S}(\mathfrak{T}_{\mathsf{B}})
  \leq\mathsf{A}(\mathfrak{T}_{\mathsf{B}}).
  \]
\end{cor}

\begin{proof}
  Express an AF element as a word in $R$ and $K$ and apply the lemma. Every truncated $K$-factor has odd order, and thus its AF parity is zero. The $R$-factors also have zero AF parity. Their product therefore belongs to the kernel $\mathsf{A}(\mathfrak{T}_{\mathsf{B}})$ of $\varepsilon_{\mathsf{B}}$.
\end{proof}

We apply Corollary~\ref{cor:odd-singular-AF-parity} to the subgroup generated by the fragments of $\widehat{G}_{\mathrm{FG}}$.

\begin{prop}
  \label{prop:FG-AF-parity}
  We have
  \[
    \mathsf{P}_{\mathsf{B}}:=
    H_0\bigl(\mathfrak{T}_{\mathsf{B}};\mathbb{Z}/2\mathbb{Z}\bigr)
    \cong
    (\mathbb{Z}/2\mathbb{Z})^2.\]
  Moreover,
  \[\mathsf{P}_{\widehat{G}_{\mathrm{FG}}}:=\varepsilon_{\mathsf{B}}
    \bigl(\widehat{G}_{\mathrm{FG}}\cap\mathsf{S}(\mathfrak{T}_{\mathsf{B}})\bigr)=0.\]
  Consequently,
  \[\widehat{G}_{\mathrm{FG}}\cap\mathsf{S}(\mathfrak{T}_{\mathsf{B}})=\mathsf{A}(\mathfrak{T}_{\mathsf{B}}).\]
\end{prop}

\begin{proof}
  Let $\overline{M}_{\mathsf{B}}$ be the reduction of
  $M_{\mathsf{B}}$ modulo $2$, and let
  \[
    \overline{u}=
    \begin{pmatrix}
      0 \\1\\0\\0
    \end{pmatrix},
    \qquad
    \overline{v}=
    \begin{pmatrix}
      1 \\0\\1\\1
    \end{pmatrix}.
  \]
  The image of $\overline{M}_{\mathsf{B}}$ is the two-dimensional space generated by $\overline{u},\overline{v}$, and
  \[
    \overline{M}_{\mathsf{B}}\overline{u}=
    \overline{v},
    \qquad
    \overline{M}_{\mathsf{B}}\overline{v}=
    \overline{u}.
  \]
  The restriction of $\overline{M}_{\mathsf{B}}$ to its image is therefore an automorphism, and we obtain
  $\mathsf{P}_{\mathsf{B}}\cong(\mathbb{Z}/2\mathbb{Z})^2$.

  We next compute the parity realized by $\widehat{G}_{\mathrm{FG}}$. Put $R=\langle a_{0_0},a_{0_1},a_1,a_2\rangle$ and $K=\langle b_1,b_2\rangle$. The four generators of $R$ are the finitary $3$-cycles in Figure~\ref{fig:four-triangular-connectors}, and thus $R\leq\mathsf{A}(\mathfrak{T}_{\mathsf{B}})$. Lemma~\ref{selectorexist} gives $K\cong(\mathbb{Z}/3\mathbb{Z})^2$.

  The recursions~\eqref{gen1} and~\eqref{gen2} show that every element of $K$ fixes $\xi=2^\omega$. At a point different from $\xi$, each directed recursion exits the $2$-coordinate after finitely many steps and all remaining sections are finitary. Thus each generator of $K$ has only AF germs away from $\xi$. Since the generators fix $\xi$, their products cannot send another point to $\xi$. Thus every germ of $K$ away from $\xi$ is AF. Corollary~\ref{cor:odd-singular-AF-parity} therefore gives
  \[
    \widehat{G}_{\mathrm{FG}}\cap\mathsf{S}(\mathfrak{T}_{\mathsf{B}})
    =\langle R,K\rangle\cap\mathsf{S}(\mathfrak{T}_{\mathsf{B}})
    \leq\mathsf{A}(\mathfrak{T}_{\mathsf{B}}).
  \]
  The reverse inclusion is Theorem~\ref{prop:FG-selector-system}. This proves both the intersection formula and $\mathsf{P}_{\widehat{G}_{\mathrm{FG}}}=0$.
\end{proof}

For every $N\geq 2$, let $\tau_{0_0}\in\mathsf{S}(\mathcal{T}_{0_0,N})$ and
$\tau_{0_1}\in\mathsf{S}(\mathcal{T}_{0_1,N})$ be any transpositions. Their parity vectors at level $N$ are the basis vectors $e_{0_0}$ and
$e_{0_1}$. After one refinement,
\[\overline{M}_{\mathsf{B}}e_{0_0}=
  \overline{u},
  \qquad
  \overline{M}_{\mathsf{B}}e_{0_1}=
  \overline{v}.\]
Their classes therefore form a basis of $\mathsf{P}_{\mathsf{B}}$. Put
\[\widehat{G}_{\mathrm{FG}}^{\mathrm{par}}=\left\langle\widehat{G}_{\mathrm{FG}},\tau_{0_0},\tau_{0_1}\right\rangle.\]

\begin{cor}
  \label{cor:FG-AF-symmetric-completion}
  We have $\mathsf{S}(\mathfrak{T}_{\mathsf{B}})\leq\widehat{G}_{\mathrm{FG}}^{\mathrm{par}}$.
\end{cor}

\begin{proof}
  The parity classes of $\tau_{0_0}$ and $\tau_{0_1}$ generate $\mathsf{P}_{\mathsf{B}}$, while $\mathsf{A}(\mathfrak{T}_{\mathsf{B}})\leq\widehat{G}_{\mathrm{FG}}$. We apply the exact sequence for AF parity.
\end{proof}

\subsection{Singular germs and the topological full group}
\label{subsec:FG-singular-full-group}
Let $\xi=2^\omega$.  For $n\geq1$, let $H_n$ be the shifted group of germs at $\xi^{(n)}$. We use the containment
$\mathsf{A}(\mathfrak{T}_{\mathsf{B}})\leq\widehat{G}_{\mathrm{FG}}$ proved in Theorem~\ref{prop:FG-selector-system}.

\begin{prop}
  \label{prop:FG-singular-germs-localization}
  The point $\xi$ is the unique boundary point of the infinite tiles in this construction, and it is germ-defining relative to the generating set of $\widehat{G}_{\mathrm{FG}}$.  Consequently, $\widehat{G}_{\mathrm{FG}}$ satisfies the finite singular germ condition of Definition~\ref{def:prelim-finite-singular-germ-condition} with $\Xi=\{\xi\}$.

  For every $n\geq1$, we have $H_n\cong(\mathbb{Z}/3\mathbb{Z})^2$, and $\xi$ is purely non-Hausdorff.

  Moreover, let $h\in H_n$, and let $h_m\in H_m$ be its canonical image, where $m\geq n$. After increasing $m$, there is an element $\widehat{h}\in\widehat{G}_{\mathrm{FG}}$ such that:

  \begin{enumerate}
    \item
          the cylinder $[2^m]$ is $\widehat{h}$-invariant;

    \item
          the transformation $\widetilde{h}_m=\kappa_{2^m}^{-1}\widehat{h}\kappa_{2^m}$ represents $h_m$;

    \item
          every germ of $\widehat{h}$ outside $[2^m]$ belongs to $\mathfrak{T}_{\mathsf{B}}$;

    \item
          $\widehat{h}$ has order dividing $3$.
  \end{enumerate}

  We may also choose $m$ so that the tile containing $2^m$ has at least three vertices. The AF truncation associated with this localization has trivial AF parity. Consequently, every shifted singular germ can be localized by an element of $\widehat{G}_{\mathrm{FG}}$ to every sufficiently deep cylinder whose prefix ends at the same vertex of $\mathsf{B}$ as $2^m$.
\end{prop}

\begin{proof}
  Every boundary edge that persists under inflation is carried by a section of the directed generator $b$. The only non-AF directed section that persists under further inflation occurs in the $2$-coordinate and again has the same form at the next level. Hence a boundary point on a finite tile can remain on the boundary under every further inflation only by continuing with the letter $2$ at each step. Thus $\xi=2^\omega$ is the only infinite path represented by boundary points at every level. Replacing the finitary generator $a$ by its disjoint $3$-cycles $a_{0_0},a_{0_1},a_1,a_2$ changes only edges inside the tiles, while fragmenting $b$ only changes the labels on the existing boundary edges. Neither operation creates another such infinite path. Hence the set of boundary points is $\{\xi\}$.

  The transformations $a_{0_0},a_{0_1},a_1,a_2$ are finitary. Every non-AF section of $b^{(1)}_1$ or $b^{(1)}_2$
  continues only along the letter $2$. Thus both fragmented directed generators fix $\xi$, have a non-AF germ at $\xi$, and have only AF germs at all other points.

  We verify that $\xi$ is germ-defining. Let $g=s_k\cdots s_1$ label a nonloop cycle based at $\zeta\in\widehat{G}_{\mathrm{FG}}\xi$, and put $\zeta_j=s_j\cdots s_1(\zeta)$. Suppose that $(s_j,\zeta_{j-1})$ is
  non-AF for some $j$. Then $\zeta_{j-1}=\xi$. Since every singular generator fixes $\xi$, we have $\zeta_j=\zeta_{j-1}$, contradicting the definition of a nonloop cycle. Therefore every factor germ belongs to $\mathfrak{T}_{\mathsf{B}}$. The germ $(g,\zeta)$ is then an isotropy element of the principal groupoid $\mathfrak{T}_{\mathsf{B}}$, and hence it is trivial. Thus $\xi$ is germ-defining.

  We now compute the shifted groups of germs. Let $\mathcal{P}=\{P_0,P_1,P_2,P_3\}$ be the partition into four pieces used in the fragmentation of $\langle b\rangle$. For every shifted level $n$, let $P_{0,n},P_{1,n},P_{2,n},P_{3,n}$ be the corresponding shifted pieces. Every $P_{j,n}$ accumulates on $\xi^{(n)}$.

  By the canonical isomorphisms obtained by restriction and the representatives of singular germs in \cite[Subsection~3.3 and Proposition~3.7]{kua26no1}, every element of $H_n$ is represented by a pathwise section of the subgroup generated by the fragments. Let $\ell_{1,n}$ and $\ell_{2,n}$ be the germs at $\xi^{(n)}$ of the pathwise sections of $b^{(1)}_1$ and $b^{(1)}_2$, respectively. Lemma~\ref{selectorexist} gives their activity vectors as $(0,1,2,1)$ and $(1,1,1,0)$ at the initial position of the period. After shifting the level, these vectors are cyclically permuted, and therefore remain linearly independent.

  Since every $P_{j,n}$ accumulates on $\xi^{(n)}$, a nonzero activity vector gives a nontrivial germ. Hence the map sending a germ to its activity vector is injective on the group generated by $\ell_{1,n}$ and $\ell_{2,n}$, and we obtain $H_n=\langle\ell_{1,n},\ell_{2,n}\rangle\cong(\mathbb{Z}/3\mathbb{Z})^2$.

  The point $\xi$ is purely non-Hausdorff by Propositions~\ref{prop:FG-selector-producing-fragmentation} and~\ref{prop:identity-coordinate-pure-non-Hausdorff}.

  Now let $h\in H_n$. Restriction gives a canonical isomorphism from $H_n$ to the group of germs at $\xi$ of the subgroup
  $\mathcal{F}_b:=\langle b^{(1)}_1,b^{(1)}_2\rangle\cong(\mathbb{Z}/3\mathbb{Z})^2$.
  Choose $\widehat{h}\in\mathcal{F}_b$ whose pathwise section at level $n$ represents $h$. Since $\mathcal{F}_b$ has
  exponent $3$, the order of $\widehat{h}$ divides $3$.

  Along the singular path, every nontrivial transition of a fragment has label $2|2$. Hence, after increasing
  $m$, every fragment occurring in $\widehat{h}$ preserves the cylinder $[2^m]$, and so does $\widehat{h}$. Composing the pathwise sections shows that $\kappa_{2^m}^{-1}\widehat{h}\kappa_{2^m}$ represents the
  canonical image $h_m\in H_m$.

  Every germ of each fragmented directed generator away from $\xi$ is AF. Moreover, all fragmented directed generators fix $\xi$. Therefore, a point different from $\xi$ cannot be mapped to $\xi$ by a word in these generators. It follows that every germ of $\widehat{h}$ away from $\xi$, and hence every germ outside $[2^m]$, belongs to $\mathfrak{T}_{\mathsf{B}}$.

  Since $\mathsf{B}$ is simple and its space of infinite paths is a Cantor set, we may increase $m$ so that the tile
  containing $2^m$ has at least three vertices.

  It remains to verify the parity of the AF truncation. Define
  \[
    \widetilde{a}(x)=
    \begin{cases}
      x,              & x\in[2^m],    \\
      \widehat{h}(x), & x\notin[2^m].
    \end{cases}
  \]
  The cylinder and its complement are $\widehat{h}$-invariant, and thus $\widetilde{a}^3=\Id$. We also have $\widetilde{a}\in\mathsf{S}(\mathfrak{T}_{\mathsf{B}})$. At a sufficiently deep finite level, the restriction of $\widetilde{a}$ to every tile is a product of fixed points and $3$-cycles. Therefore $\varepsilon_{\mathsf{B}}(\widetilde{a})=0$, and the exact sequence for AF parity gives $\widetilde{a}\in\mathsf{A}(\mathfrak{T}_{\mathsf{B}})\leq\widehat{G}_{\mathrm{FG}}$.
  Consequently, $\widetilde{a}^{-1}\widehat{h}\in\widehat{G}_{\mathrm{FG}}$ is supported on $[2^m]$ and represents $h_m$ there. Since the tile containing $2^m$ has at least three vertices, its finite alternating group is transitive. Conjugating $\widetilde{a}^{-1}\widehat{h}$ by this alternating group gives a representative supported in any cylinder whose prefix has the same terminal vertex as $2^m$.
\end{proof}

The containment, parity and localization computations, together with the completion theorem of \cite{kua26no1}, prove Corollary~D.

\begin{cor}
  \label{cor:FG-near-full-completion}
  The group $\widehat{G}_{\mathrm{FG}}$ has index $4$ in $\mathsf{F}(\mathfrak{G}_{\mathrm{FG}})$, and
  \[
    \mathsf{F}(\mathfrak{G}_{\mathrm{FG}})
    =\mathsf{S}(\mathfrak{T}_{\mathsf{B}})\widehat{G}_{\mathrm{FG}}
    =\widehat{G}_{\mathrm{FG}}^{\mathrm{par}}.
  \]
\end{cor}

\begin{proof}
  Theorem~\ref{prop:FG-selector-system} gives containment of the AF alternating group. Proposition~\ref{prop:FG-singular-germs-localization} verifies for this action the finite singular germ condition, the localization condition, and the parity condition for the AF truncations. Proposition~\ref{prop:FG-AF-parity} gives $\mathsf{P}_{\widehat{G}_{\mathrm{FG}}}=0$ and $|\mathsf{P}_{\mathsf{B}}|=4$. Theorem~\ref{thm:prelim-full-group-completion}, recalled from \cite{kua26no1}, now gives the displayed formula for the full group and the index.
\end{proof}

\def\BState{\State\hskip-\ALG@thistlm}
\makeatother
\let\oldbibitem\bibitem
\renewcommand{\bibitem}{\setlength{\itemsep}{-4pt}\oldbibitem}






\end{document}